\documentclass[reqno]{amsart}

\usepackage{graphicx, amsmath, amssymb, amsfonts, amsthm, stmaryrd, tikz, amscd}

\usepackage{enumerate}
\usepackage{times}
\usepackage[normalem]{ulem}

\usepackage{hyperref}

\usepackage{pdfsync}
\usepackage{xparse}

\usepackage[all]{xy}
\usepackage{enumerate}
\usetikzlibrary{matrix,arrows,decorations.pathmorphing}

\makeatletter
\newcommand*{\transpose}{%
  {\mathpalette\@transpose{}}%
}
\newcommand*{\@transpose}[2]{%
  \raisebox{\depth}{$\m@th#1\intercal$}%
}
\makeatother

\makeatletter
\newcommand*{\da@rightarrow}{\mathchar"0\hexnumber@\symAMSa 4B }
\newcommand*{\da@leftarrow}{\mathchar"0\hexnumber@\symAMSa 4C }
\newcommand*{\xdashrightarrow}[2][]{%
  \mathrel{%
    \mathpalette{\da@xarrow{#1}{#2}{}\da@rightarrow{\,}{}}{}%
  }%
}
\newcommand{\xdashleftarrow}[2][]{%
  \mathrel{%
    \mathpalette{\da@xarrow{#1}{#2}\da@leftarrow{}{}{\,}}{}%
  }%
}
\newcommand*{\da@xarrow}[7]{%
  \sbox0{$\ifx#7\scriptstyle\scriptscriptstyle\else\scriptstyle\fi#5#1#6\m@th$}%
  \sbox2{$\ifx#7\scriptstyle\scriptscriptstyle\else\scriptstyle\fi#5#2#6\m@th$}%
  \sbox4{$#7\dabar@\m@th$}%
  \dimen@=\wd0 %
  \ifdim\wd2 >\dimen@
    \dimen@=\wd2 %
  \fi
  \count@=2 %
  \def\da@bars{\dabar@\dabar@}%
  \@whiledim\count@\wd4<\dimen@\do{%
    \advance\count@\@ne
    \expandafter\def\expandafter\da@bars\expandafter{%
      \da@bars
      \dabar@ 
    }%
  }%
  \mathrel{#3}%
  \mathrel{%
    \mathop{\da@bars}\limits
    \ifx\\#1\\%
    \else
      _{\copy0}%
    \fi
    \ifx\\#2\\%
    \else
      ^{\copy2}%
    \fi
  }%
  \mathrel{#4}%
}
\makeatother

\usepackage{mathtools}
\DeclarePairedDelimiter{\paren}{(}{)}

\newcommand{\Of}[2]{{\operatorname{#1}} {\paren*{#2}}}

\DeclareMathOperator{\Mp}{Mp}
\DeclareMathOperator{\Mat}{Mat}
\DeclareMathOperator{\GL}{GL}

\DeclareMathOperator{\Ad}{Ad}

\DeclareMathOperator{\Hom}{Hom}

\DeclareMathOperator{\End}{End}

\def\eps{\varepsilon}

\DeclareMathOperator{\U}{U}
\DeclareMathOperator{\trace}{trace}

\DeclareMathOperator{\ind}{ind}

\DeclareMathOperator{\Aut}{Aut}

\DeclareMathOperator{\Gal}{Gal}

\DeclareMathOperator{\diag}{diag}

\DeclareMathOperator{\Lie}{Lie}

\DeclareMathOperator{\stab}{stab}

\def\O{\operatorname{O}}

\DeclareMathOperator{\rank}{rank}

\DeclareMathOperator{\Ind}{Ind}

\DeclareMathOperator{\vol}{vol}

\theoremstyle{plain} \newtheorem{theorem} {Theorem}  \newtheorem{corollary} [theorem] {Corollary} \newtheorem{proposition} [theorem] {Proposition} 
\theoremstyle{definition} \newtheorem{definition} [theorem] {Definition}

\newtheorem{example} [theorem] {Example}  
    \newtheorem{problem}[theorem] {Problem}    \newtheorem{remark} [theorem] {Remark}

\newtheoremstyle{itplain} 
{6pt}                    
{5pt\topsep}                    
{\itshape}                   
{}                           
{\itshape}                   
{.}                          
{5pt plus 1pt minus 1pt}                       
{}  

\theoremstyle{itplain} 
\newtheorem{lemma}[theorem]{Lemma}

\newtheorem*{lemma*}{Lemma}
\newtheorem*{proposition*}{Proposition}
\newtheorem*{definition*}{Definition}
\newtheorem*{notation*}{Notation}
\newtheorem*{example*}{Example}

\newtheorem*{results*}{Results}

\newtheorem{Assumption}[theorem]{Assumption} \numberwithin{equation}{section} \numberwithin{theorem}{section}
\newcommand{\mfq}{\mathfrak{q}}

\newcommand{\mfp}{\mathfrak{p}}

\newcommand{\mfop}{\mathfrak{l}}

\begin{document}
\title{Subconvex bounds for $U_{n+1}\times U_n$ in horizontal aspects}

\author{Yueke Hu} \address{YMSC, Tsinghua University, Beijing, China} \email{yhumath@tsinghua.edu.cn}

\author{Paul D. Nelson} \address{Aarhus University, Aarhus, Denmark} \email{paul.nelson@math.au.dk}


\begin{abstract}
  For $L$-functions attached to automorphic representations of unitary groups $U_{n+1}\times U_n$, we establish a subconvex bound valid in certain horizontal aspects, where the set of ramified places is allowed to vary.
\end{abstract}
\maketitle
\tableofcontents

\section{Introduction}

\subsection{Overview}
This paper concerns the quantitative analysis of automorphic $L$-functions on higher rank classical groups.  Recent works on this topic include
\begin{itemize}
\item an approach to the analytic test vector problem based on the orbit method, applied to the asymptotic evaluations of moments (see~\cite{nelson-venkatesh-1}), and
\item subconvex bounds on $\U_{n+1} \times \U_n$ and $\GL_n$ (see~\cite{2023arXiv2309.16667, 2020arXiv201202187N, 2021arXiv210915230N}).
\end{itemize}
These works study sequences of automorphic forms whose ramification increases inside some \emph{fixed} finite set of places, a setup known as the \emph{depth} aspect.

By a \emph{horizontal} aspect, we mean one where the set of ramified places is itself allowed to vary.  For example, one can study Dirichlet characters of conductor $p^m$ for fixed $p$ as $m \rightarrow \infty$ (depth) and for fixed $m$ as $p \rightarrow \infty$ (horizontal).

This paper establishes subconvex bounds on $U_{n +1} \times U_n$ in certain horizontal aspects.  Our results are uniform enough to apply also in the depth aspect (indeed, as ``$p^m \rightarrow \infty$''), but this is not the main novelty.

The depth aspect often has a Lie-algebraic flavor, involving tools such as stationary phase analysis and Taylor approximation.  For instance, the works noted above make heavy use of the exponential map and Lie algebra for groups such as $\GL_n(F)$, with $F$ a fixed local field.  These techniques are applied to spectral problems (construction and analysis of test vectors) and geometric problems (the ``volume bound'').

In horizontal aspects, Lie-algebraic techniques are often less relevant.  For this paper, we restrict to ``even depth'' cases (``$p^{2 m}$'') where the spectral analysis can still be carried out using Lie-algebraic techniques, but the geometric analysis requires new arguments, of algebro-geometric rather than Lie-algebraic flavor.

\subsection{Motivating goal: subconvexity for twists}
The main result of this paper is a bit technical to state, but our work is motivated by the simpler problem of establishing subconvexity for twists of standard $L$-functions: for a fixed cuspidal representation $\pi$ of $\GL_n$ and a varying Hecke character $\chi$ of $\GL_1$ (over $\mathbb{Q}$, say), we would like to show that
\begin{equation}\label{eq:cj54locczh}
L (\pi \otimes \chi, \tfrac{1}{2} ) \ll_{\pi} C(\chi)^{n/4-\delta}
\end{equation}
for some fixed $\delta > 0$.  Such estimates have been known for a while when $n \leq 2$ \cite{michel-2009}, and in some essential cases when $n = 3$ \cite{MR3369905, MR3418527, RHPNtwists, 2018arXiv181000539M,  MR4416133}.

For general $n$, the preprint \cite{2021arXiv210915230N} treated (among others) the case
\begin{equation*}
  \chi = |.|^{it}.
\end{equation*}
The proof involves ``compatibly microlocalized'' \cite[\S1.7--1.9]{nelson-venkatesh-1} automorphic forms for $\GL_n$ and $\GL_{n-1}$, integrated over the non-compact adelic quotients of $\GL_{n-1}$.  It divides roughly into estimates pertaining to the ``bulk'' and ``cusp'' of such quotients.  The ``bulk estimates'' were established in the earlier paper \cite{2020arXiv201202187N} and applied there to $L$-functions on certain unitary groups, for which the corresponding quotients are compact.  Such bulk estimates reduce to a ``volume bound''~\cite[Thm 15.2]{2020arXiv201202187N} that was established in generality sufficient for the general ``vertical'' case of \eqref{eq:cj54locczh}, where $\chi$ is taken unramified outside some \emph{fixed} finite set of places.  This includes the case that
\begin{equation*}
 \text{$\chi$  has conductor $p^m$ with $m \rightarrow \infty$ and $p$ fixed.}
\end{equation*}
Marshall \cite{2023arXiv2309.16667} established the volume bound in the ``wall-avoidance'' case (see \S\ref{sec:cj4vjryk3u}), which excludes the twist aspect \eqref{eq:cj54locczh}, via an independent method.

The present paper supplies the ``bulk estimates'' (or ``volume bound'') needed to extend \eqref{eq:cj54locczh} ``horizontally'' to the case that the $\chi$ is merely assumed to have square finite conductor, with no restriction on the set of ramified places.  This includes the case that
\begin{equation*}
\text{$\chi$ has conductor $p^{2 m} \rightarrow \infty$,}
\end{equation*}
with $p$ allowed to vary.  We apply those estimates here to $L$-functions on certain unitary groups, like in \cite{2020arXiv201202187N}.  We expect that one could adapt the ``cusp estimates'' of \cite{2021arXiv210915230N} to the present setting, yielding a true extension of \eqref{eq:cj54locczh}, but leave this to future work.  The remaining essential case of \eqref{eq:cj54locczh} is then when $n \geq 4$ and the finite conductor of $\chi$ is prime.


\subsection{The refined Gan--Gross--Prasad conjectures}\label{sec:cj3siudn3c} 
We aim now to describe main result (Theorem \ref{theorem:cj3ngw7u2s}) and the main ideas of its proof, emphasizing new features encountered in horizontal aspects.  As preparation, we recall the statement of the refined Gan--Gross--Prasad conjectures.  These conjectures, now known in many cases, provide a link between values of $L$-functions and integrals of automorphic forms.  The body of this paper addresses the local problems that arise in estimating such integrals.  The link described here is not otherwise applied in this paper, but provides context for interpreting our results.

Let $F$ be a number field with adele ring $\mathbb{A}$, let $E/F$ be a quadratic extension, let $V$ be an $(n+1)$-dimensional hermitian space over $E$, and let $W$ be an $n$-dimensional nondegenerate subspace of $V$.  Define the pair of unitary groups $(G,H) := (\U(V), \U(W))$ over $F$.  Given a pair of cuspidal automorphic representations $\pi$ and $\sigma$ of $G$ and $H$ that are locally distinguished, one may attach a branching coefficient $\mathcal{L}(\pi,\sigma)$ quantifying how automorphic forms in $\pi$ correlate against those in $\sigma$.  The definition depends upon the choice of a finite set $S$ of places of $F$, taken large enough to contain every place that is archimedean or at which $\pi$ or $\sigma$ is ramified, and requires $\pi$ and $\sigma$ to be (nearly) tempered inside $S$.  It is characterized by the following family of identities: for all $v \in \pi $ and $u \in \sigma$ that are unramified outside $S$, and with suitable normalization of measure,
\begin{equation}\label{eq:cj3tfrrhlc}
  \left\lvert \int_{H(F) \backslash H(\mathbb{A})} v \bar{u}
  \right\rvert^2
  =
  \mathcal{L}(\pi, \sigma)
  \int_{H(F_S)} \langle h v, v   \rangle \langle u, h u \rangle \, d h.
\end{equation}
We refer to~\cite[\S3.7]{2020arXiv201202187N} for further details. It has been conjectured by Ichino--Ikeda~\cite{MR2585578} and N. Harris~\cite{MR3159075} that if $S$ is large enough in the sense recorded in~\cite[\S1]{MR2585578}, then $\mathcal{L}(\pi,\sigma)$ is given by a ratio of $L$-values, namely, with notation as in~\cite{MR4426741},
\begin{equation}\label{eqn:20230517061322}
  \mathcal{L}(\pi,\sigma)
  =
  2^{- \beta}
  \frac{L^{(S)}(\pi_E \otimes  \sigma_E^{\vee},1/2)}{L^{(S)}(\Ad, \pi \boxtimes {\sigma}^{\vee},1)}
  \Delta_G^{(S)}.
\end{equation}
This expectation has been proved at least when $\pi$ and $\sigma$ are tempered at all places~\cite{MR4426741}.

\subsection{Main result}\label{sec:cj4vjf3x2x}
We establish a subconvex bound on $\U_{n+1} \times \U_n$, with $\U_n$ anisotropic, for pairs $(\pi,\sigma)$ whose ramification concentrates at some (possibly varying) finite place $\mathfrak{p}$, provided that the conductor does not drop, under some local assumptions (e.g., principal series of even depth).  In the remainder of this subsection, we describe these conditions and our result in more detail.

We fix $F, E, G, H$ as in \S\ref{sec:cj3siudn3c}, as well as a finite set $S$ of places of $F$, large enough in the sense specified in~\cite[\S3.6]{2020arXiv201202187N}.  We then fix a family $\mathcal{F}$ consisting of tuples $(\pi,\sigma, \mathfrak{p}, \mathfrak{q})$ satisfying the following conditions, whose informal content is that $\pi$ and $\sigma$ are ``uniformly uninteresting'' away from the ``interesting'' place $\mathfrak{p}$, where their local components have ``depth $\mathfrak{q}^2$'' and are ``away from conductor dropping''.
\begin{enumerate}[(a)]
\item $\pi$ and $\sigma$ are cuspidal automorphic representations of $G$ and $H$, respectively, having unitary central characters.
\item $\mathfrak{p} \notin S$ is a non-archimedean place (or prime ideal) of $F$ at which $E/F$ splits, so that (see~\cite[\S3.3]{2020arXiv201202187N})
  \begin{equation*}
    (G(F_\mfp), H(F_\mfp)) \cong (\GL_{n+1}(F_\mfp), \GL_n(F_\mfp)).
  \end{equation*}
\item $\mathfrak{q}$ is a positive power of the prime ideal $\mathfrak{p}$.
\item $(\pi,\sigma)$ is locally distinguished: there is a nonzero $H(\mathbb{A})$-invariant functional $\pi \rightarrow \sigma$.
\item Inside $S$, the representations $\pi$ and $\sigma$ are tempered, and their depth is uniformly bounded as $(\pi,\sigma)$ traverses $\mathcal{F}$.  That is to say, at each place in $S$, there are compact open subgroups, independent of $\mathcal{F}$, under which $\pi$ and $\sigma$ admit invariant vectors (compare with~\cite[\S1.3]{2020arXiv201202187N}).
\item\label{itemize:cj3ubog35g} Outside $S \cup \{\mathfrak{p} \}$, the representations $\pi$ and $\sigma$ are unramified, and $\sigma$ satisfies a uniform bound towards Ramanujan at places where $E/F$ splits: it is $\vartheta$-tempered (see~\cite[\S5.2.1]{2020arXiv201202187N}) at such places for some $0 \leq \vartheta < 1/2$ not depending upon $\mathcal{F}$.
\item\label{enumerate:cj54lhzz2y} At the ``interesting'' place $\mathfrak{p}$, the representations $\pi$ and $\sigma$ are ``stable at depth $\mathfrak{q}^2$'' in the sense illustrated in Example \ref{example:cj54lh128k} and specified in Definition~\ref{definition:let-pi-sigma-be-repr-pair-gener-line-groups-over-f}.
\end{enumerate}

We assume that $V$ (hence also $W$) is positive-definite.  In particular, $H$ is anisotropic, and at each archimedean place of $F$, the groups $G$ and $H$ are compact.  We define the branching coefficients $\mathcal{L}(\pi,\sigma)$ relative to the set $S \sqcup \{\mathfrak{p} \}$.  We use the notation
\begin{equation}\label{eq:cj3ubsuj6v}
  T := \text{the absolute norm of } \mathfrak{q}^2.
\end{equation}

\begin{example}\label{example:cj54lh128k}
  The ``stable at depth $\mathfrak{q}^2$'' hypothesis \eqref{enumerate:cj54lhzz2y} contains the case of the principal series representations induced by characters of conductor dividing $\mathfrak{q}^2$ for which the local analytic conductor for $\mathcal{L}(\pi,\sigma)$ at $\mathfrak{p}$ is as large as possible.  In more detail, suppose the local components at the ``interesting'' place $\mathfrak{p}$ are the normalized inductions
  \begin{equation}\label{eqn:cj2dzujg9c}
    \pi_\mfp = \chi_1 \boxplus \dotsb \boxplus \chi_{n+1}, \qquad \sigma_\mfp = \eta_1 \boxplus \dotsb \boxplus \eta_n
  \end{equation}
  of some characters $\chi_i$ and $\eta_j$ of $F_\mfp^\times$.  For a character $\omega$ of $F_\mfp^\times$, we denote by $C(\omega)$ the analytic conductor, i.e., the absolute norm of the largest integral ideal $\mathfrak{a}$ in $F_\mfp$ such that $\omega(x) = 1$ whenever $x - 1 \in \mathfrak{a}$.  The ``stable at depth $\mathfrak{q}^2$'' assumption then says that
  \begin{equation}\label{eqn:20230516205128}
    C(\chi_i) \leq T, \qquad
    C(\eta_j) \leq T,
  \end{equation}
  \begin{equation}\label{eqn:cj2dzunqa6}
    \prod_{i, j} {C(\chi_i/\eta_j)}^2 = T^{2 n(n+1)},
  \end{equation}
  the latter of which is equivalent to requiring that for all $i$ and $j$,
  \begin{equation}\label{eqn:cj3ngvwrm4}
    C(\chi_i / \eta_j) = T.
  \end{equation}
  We note that these conditions force $T$ to be the maximum of the $C(\chi_i)$ and $C(\eta_j)$.  The left hand side of \eqref{eqn:cj2dzunqa6} is the analytic conductor for $\mathcal{L}(\pi,\sigma)$ at $\mathfrak{p}$, so the condition \eqref{eqn:cj2dzunqa6} says that this conductor is as large as possible subject to the constraint \eqref{eqn:20230516205128}.
\end{example}

Our main result is as follows:
\begin{theorem}\label{theorem:cj3ngw7u2s}
  There exists $\delta_n > 0$ and $c_{\mathcal{F}} \geq 0$ so that for all $(\pi,\sigma, \mathfrak{p}, \mathfrak{q}) \in \mathcal{F}$, we have
  \begin{equation}\label{eq:cj3tfrulcn}
    \mathcal{L}(\pi,\sigma) \leq c_{\mathcal{F}} T^{2 n (n + 1)(1/4 - \delta_n)}.
  \end{equation}
\end{theorem}
In cases where $\mathcal{L}(\pi,\sigma)$ is known to coincide with an $L$-value (e.g., the case that $\pi$ and $\sigma$ are everywhere tempered) and where we have computed the conductor at $\mathfrak{p}$ of $\mathcal{L}(\pi,\sigma)$ (e.g., Example \ref{example:cj54lh128k}, Lemma \ref{Lem:supercuspidalstablepair} and cases derived from those via parabolic induction), the estimate~\eqref{eq:cj3tfrulcn} improves upon the convexity bound, which would assert the same for some $\delta_n \leq 0$ (see~\cite[Proof of Cor 1.2]{2020arXiv201202187N} for details).  In the case $n+1 = 2$, such an estimate has been known for a while~\cite{michel-2009}.  When $n+1 \geq 3$, such an estimate is new (outside the ``$\mathfrak{p}$ fixed'' case discussed below in~\S\ref{sec:cj4vjryk3u}); we will show then that~\eqref{eq:cj3tfrulcn} holds for any
\begin{equation}\label{eq:cj3ubs6lvw}
  \delta_n < \frac{1 - 2 \vartheta }{4 n(n+1) (A + 1 - 2 \vartheta )},
  \quad
  A := (2 {(n + 1 )}^2 - n) (n + 1),
\end{equation}
with $\vartheta$ as in assumption~\eqref{itemize:cj3ubog35g} above (and $c_\mathcal{F}$ allowed to depend upon $\delta_n$).

\begin{remark}\label{remark:cj3u9049bv}
  We do not address the interesting challenge of improving the numerical strength of the exponent~\eqref{eq:cj3ubs6lvw}.  Several avenues for doing so were mentioned in~\cite[Rmk 1.4]{2020arXiv201202187N}, many of which could be pursued in the present context.
\end{remark}
\begin{remark}\label{remark:cj3u9062fs}
  The ``large conductor'' assumption~\eqref{eqn:cj2dzunqa6} is the most serious one --- it is an open problem to give any genuine subconvex bound for $\GL_3$ or higher in any case where the conductor drops (see the final paragraph of~\cite[\S1.4]{2021arXiv210915230N} and references).
\end{remark}

\subsection{Related results}\label{sec:cj4vjryk3u}
The special case of Theorem~\ref{theorem:cj3ngw7u2s} in which the following assumptions are valid was established by Marshall \cite{2023arXiv2309.16667}:
\begin{enumerate}[(i)]
\item\label{enumerate:cj54li4h1k} (Principal series) $\pi$ and $\sigma$ are as in Example \ref{example:cj54lh128k}.
\item\label{enumerate:20230516210021} (Wall-avoidance) The inducing characters for $\pi$ and $\sigma$ at $\mathfrak{p}$ satisfy, for all $i \neq j$,
  \begin{equation}\label{eq:cj54lgf0ez}
    C(\chi_i / \chi_j) = T, \qquad C(\eta_i / \eta_j) = T. 
  \end{equation}
\item\label{enumerate:20230516210903} (Depth aspect) $\mfp$ is a \emph{fixed} (i.e., independent of $\mathcal{F}$) non-archimedean place.
\end{enumerate}
An important input to his arguments is a certain \emph{volume bound}, discussed further in \S\ref{sec:cj3ubuyr8u}, which Marshall established in the required cases (see \cite[Prop 6.2]{2023arXiv2309.16667}).

The paper~\cite{2020arXiv201202187N} gave analogues of Theorem~\ref{theorem:cj3ngw7u2s} at an archimedean place $\mfp$ (without assuming ``principal series'' or ``$E/F$ split at $\mathfrak{p}$'').  Parts of the proof apply to the non-archimedean ``$\mathfrak{p}$ fixed'' case of Theorem~\ref{theorem:cj3ngw7u2s}; for instance, the volume bound~\cite[Thm 15.2]{2020arXiv201202187N} was established over any fixed local field, without assuming wall-avoidance.

In summary, the ``$\mathfrak{p}$ fixed'' case of Theorem~\ref{theorem:cj3ngw7u2s} --- an estimate like~\eqref{eq:cj3tfrulcn}, but with $c_{\mathcal{F}}$ allowed to depend also upon $\mathfrak{p}$ --- is closely related to existing results.  The main novelty here is to allow $\mathfrak{p}$ to vary.

\begin{remark}
  Strictly speaking, one could quantify the available arguments to obtain \emph{some} uniformity, e.g.,~\eqref{eq:cj3tfrulcn} with $c_{\mathcal{F}}$ depending polynomially upon $\mathfrak{p}$.  This yields subconvexity under a ``sufficient depth'' restriction, namely, that $\mathfrak{q} = \mathfrak{p}^m$ with $m$ large enough in terms of the rank $n$.  Our main novelty is thus to allow $\mathfrak{p}$ to vary while $\mathfrak{q}$ is a small power of $\mathfrak{p}$, e.g., $\mathfrak{q} = \mathfrak{p}$.  Such ``horizontal'' cases do not follow from direct quantification.
\end{remark}

\subsection{Division of the proof}\label{sec:cj4vjf3iqs}
Like in previous works, the basic object of study is the integral
\begin{equation}\label{eq:cj3v66d7v4}
  \int_{[H]} \pi(\omega) v \cdot \overline{u}
\end{equation}
for suitable vectors $v \in \pi$ and $u \in \sigma$ and a ``convolution kernel'' or ``amplifier'' $\omega \in C_c^\infty(G(\mathbb{A}))$, chosen so that $\pi(\omega) v$ approximates $v$.  One seeks to bound the integrals~\eqref{eq:cj3v66d7v4} simultaneously
\begin{itemize}
\item from below, using~\eqref{eq:cj3tfrrhlc}, in terms of $\mathcal{L}(\pi,\sigma)$, and
\item from above, by their second moment over ``all'' $\pi$ and $v$, using its ``relative trace formula'' expansion
  \begin{equation}\label{eq:cj3tv4jpta}
    \int_{x, y \in [H]} \bar{u}(x) u(y) \sum_{\gamma \in G(F)}  \omega (x^{-1} \gamma y) \, d x \, d y.
  \end{equation}
\end{itemize}
We work with factorizable vectors $v$ and $u$ and a factorizable test function $\omega$.  Their local components at the ``uninteresting'' places are chosen in a soft and general way~\cite[\S5]{2020arXiv201202187N}.  The key point is to choose the components $v_\mathfrak{p}, u_\mathfrak{p}$ and $\omega_\mathfrak{p}$ at the ``interesting'' place and then to establish the required estimates.  The proof may be divided into roughly the following steps (see~\cite[\S2]{2020arXiv201202187N} for a more leisurely overview).
\begin{enumerate}
\item\label{enumerate:cj3ubymeey} For an individual representation $\pi$ of $G$, a construction of vectors $v \in \pi$ that are \emph{microlocalized} with respect to some parameter $\tau$~\cite[\S1.7]{nelson-venkatesh-1}.
\item\label{enumerate:cj3ubymf5o} For a pair $(\pi,\sigma)$ as above, a study of the pairs of parameters $(\tau, \tau_H)$, at which microlocalized vectors $v \in \pi$ and $u \in \sigma$ exist, that are \emph{compatible} in the sense that $\tau$ restricts to $\tau_H$~\cite[\S1.9, \S13-14]{nelson-venkatesh-1}.
\item\label{enumerate:cj3ubymh5v} Spectral estimates: the estimation (in particular, lower bound) of integrals of matrix coefficients of microlocalized vectors~\cite[\S1.10, \S18]{nelson-venkatesh-1}.
\item\label{enumerate:cj3ubymi4d} Geometric estimates, namely the ``volume bound'', which is the key technical problem that arises when attempting to estimate~\eqref{eq:cj3tv4jpta} (see \cite[\S6]{2023arXiv2309.16667}, \cite[\S15--17]{2020arXiv201202187N}).
\end{enumerate}
The first three of these steps adapt readily to the setting of this paper:
\begin{enumerate}
\item The ``stable at depth $\mathfrak{q}^2$'' assumption makes the construction of microlocalized vectors particularly concrete (see \S\ref{sec:part-body-paper} and Remark~\ref{remark:cj3u9047n5}).
\item The study of compatible parameters was addressed in~\cite[\S13-14]{nelson-venkatesh-1} over any base field of characteristic zero; we have found it convenient here to extend that study to a general base ring (\S\ref{sec:20230514080526}), using arguments quite similar to those in \emph{loc.\ cit.}
\item The spectral estimates proceed in our context as in~\cite[\S18]{nelson-venkatesh-1}, with simplification.
\end{enumerate}
The focus of this paper is thus on the geometric estimates, namely, in establishing a volume bound that is uniform in $\mathfrak{p}$.  We discuss this further starting in \S\ref{sec:cj3ubuyr8u}.

\begin{remark}
  The method and ideas employed here have a long history.  We refer to~\cite[\S2.6]{2020arXiv201202187N} for an overview, but emphasize the use of amplification following Duke--Friedlander--Iwaniec~\cite{MR1942691, DFI94, MR1923476}, the systematic study of period integrals and test vectors following Bernstein--Reznikov~\cite{MR2726097}, Venkatesh~\cite{venkatesh-2005} and Michel--Venkatesh~\cite{michel-2009}, and the influential papers of Sarnak~\cite{MR780071} and Iwaniec--Sarnak~\cite{iwan-sar}.
\end{remark}

\begin{remark}\label{remark:cj3u9047n5}
  The local conditions at $\mathfrak{p}$ in the hypotheses of Theorem \ref{theorem:cj3ngw7u2s} have been formulated (Definition~\ref{definition:let-pi-sigma-be-repr-pair-gener-line-groups-over-f}) in terms of what is needed by our argument, namely, the existence of suitable vectors.  We show in \S\ref{sec:part-body-paper} that these conditions are closed under parabolic induction and apply to certain supercuspidals (e.g., characters), but do not exhaustively classify the representations to which they apply.  Giving such a classification is an interesting problem in the representation theory of supercuspidals that belongs to step~\eqref{enumerate:cj3ubymeey} in the proof strategy outlined above.  For the other steps~\eqref{enumerate:cj3ubymf5o},~\eqref{enumerate:cj3ubymh5v} and~\eqref{enumerate:cj3ubymi4d}, our treatment is general.
\end{remark}
\begin{remark}\label{remark:cj3ubzzdb9}
  It would be desirable to generalize Theorem~\ref{theorem:cj3ngw7u2s} by removing the parity condition on the conductor exponent, i.e., by replacing ``$\mathfrak{q}^2$'' with $\mathfrak{q}$ in the statement, or equivalently, by allowing $T$ to be the norm of any ideal (rather than the square an ideal) in the discussion following Theorem~\ref{theorem:cj3ngw7u2s}.  Such a generalization seems accessible in the depth aspect, but would require new ideas in horizontal aspects.  The issues touch upon all steps~\eqref{enumerate:cj3ubymeey},~\eqref{enumerate:cj3ubymf5o},~\eqref{enumerate:cj3ubymh5v} and~\eqref{enumerate:cj3ubymi4d} in the proof strategy outlined above; for instance, in place of the volume bound, one likely needs to estimate character sums (see the third paragraph of~\cite[\S1.4]{2021arXiv210915230N} and references).  This seems to us an interesting direction for future work.
\end{remark}
\begin{remark}\label{remark:cj3ubsydvn}
  Using the methods of~\cite{2021arXiv210915230N}, it should be possible to extend Theorem~\ref{theorem:cj3ngw7u2s} to the split case $E = F \times F$ and to Eisenstein series.  This would yield subconvex bounds in horizontal aspects for standard $L$-functions away from conductor dropping.  Such bounds should apply to character twists $L(\pi \otimes \chi, \tfrac{1}{2} + i t)$, with $\pi$ fixed and $\chi$ of \emph{square} conductor.  More generally, it should be possible to establish subconvex bounds for standard $L$-functions $L(\pi,\tfrac{1}{2})$ when $\pi$ is induced at each finite place from characters all having the same square conductor.  To remove the squareness assumption would require extending our methods as in Remark~\ref{remark:cj3ubzzdb9}.
\end{remark}
\begin{remark}\label{remark:cj3u905a9t}
  One motivation for studying horizontal aspects comes from conjectures and results of Lapid--Mao~\cite{MR3267120, MR3619910, MR3649366}, which relate quadratic twists of self-dual standard $L$-functions on $\GL_n$ to Fourier coefficients of automorphic forms on $\Mp_{2n}$.  Estimates for such coefficients could have applications to representation problems concerning quadratic forms.  The present work does not apply to such twists due to the squareness restriction in Remark~\ref{remark:cj3ubsydvn}, but may be understood as a step in that direction.
\end{remark}

\subsection{Geometric estimates: the uniform volume bound}\label{sec:cj3ubuyr8u}
Continuing the discussion of \S\ref{sec:cj4vjf3iqs}, we formulate the volume bound below as Problem~\ref{Probleminformal}.  The asymptotic determination of~\eqref{eq:cj3tv4jpta} reduces to the volume bound as in previous work (see~\cite[\S6]{2023arXiv2309.16667}, \cite[\S1.5.3]{2020arXiv201202187N} or~\S\ref{Sec:bilinear}).  Cases of the volume bound sufficient for the depth aspect were established in ~\cite[\S6]{2023arXiv2309.16667} and~\cite[\S15--17]{2020arXiv201202187N}, using reductions specific to that aspect (see \S\ref{sec:cj4t64wm4h} for details).  We develop a different approach (\S\ref{sec:cj4t647ljn}) that gives the volume bound in general, uniformly in $\mathfrak{p}$.

The volume bound is a local assertion concerning the ``interesting'' place $\mathfrak{p}$.  To simplify notation, we write simply $G, H$, etc., for points over $F_\mathfrak{p}$, and drop the subscripts $\mathfrak{p}$.  Thus, $F$ now denotes a non-archimedean local field, arising as the local component ``$F_\mathfrak{p}$'' of the global field considered above.  We write simply $\mathfrak{p}$ for its maximal ideal, $\mathfrak{o}$ for its maximal order, and $q := \left\lvert \mathfrak{o} / \mathfrak{p} \right\rvert$ for the residue field cardinality.  We set
\begin{equation*}
  (G,H,M, M_H) = (\GL_{n+1}(F), \GL_n(F), \Mat_{n+1}(F), \Mat_n(F)),
\end{equation*}
where $\Mat_n$ means ``$n \times n$ matrices''.

Let $K = \GL_{n+1}(\mathfrak{o}) < G$ denote the standard maximal compact subgroup.  Recall from Theorem~\ref{theorem:cj3ngw7u2s} the positive power $\mathfrak{q}$ of $\mathfrak{p}$, with square $\mathfrak{q}^2$ of norm $T$, that controls the depths of the representations $\pi$ of $G$ and $\sigma$ of $H$ that we consider.

Let $\tau \in M$, with upper-left block $\tau_H \in M_H$.  Let $G_\tau$ and $H_{\tau_H}$ denote the respective centralizers of $\tau$ in $G$ and of $\tau_H$ in $H$.  We say that $\tau$ is \emph{stable} if it has no eigenvalues in common with $\tau_H$ over an algebraic closure, or equivalently, if the characteristic polynomials of $\tau$ and $\tau_H$ generate the unit ideal.  This is an avatar for the ``large conductor'' assumption~\eqref{eqn:cj3ngvwrm4} (see~\cite[\S15]{nelson-venkatesh-1}) and is equivalent to the geometric invariant theory notion of stability (see~\cite[\S14]{nelson-venkatesh-1}).

\begin{problem}\label{Probleminformal}
  Fix a stable element $\tau \in M$.  Fix $a \in G - H Z$, where $Z < G$ denotes the center.  Give a nontrivial bound for the volume of the set of all $y \in H_{\tau_H} \cap K$ for which $a y$ is congruent modulo $\mathfrak{q}$ to an element of $H G_\tau$.
\end{problem}
Here ``nontrivial'' means a power saving in $T$ over the trivial bound, $\vol(H_{\tau_H} \cap K)$.  Strictly speaking, we need mild refinements of the problem statement depending quantitatively upon $a$.

This problem was solved by Marshall~\cite[\S6]{2023arXiv2309.16667} in the $p$-adic depth aspect, assuming the ``wall-avoidance'' condition noted in  \S\ref{sec:cj4vjryk3u}.  The general depth aspect (archimedean or $p$-adic, and without assuming ``wall-avoidance'') was treated in~\cite{2020arXiv201202187N}.

A key result of this paper, Theorem~\ref{theorem:volume-bound}, gives a solution to Problem~\ref{sec:cj4t647ljn} that is uniform with respect to variation of the underlying local field $F$.  In \S\ref{sec:cj4t647ljn}, we summarize the proof.  In \S\ref{sec:cj4t64wm4h}, we explain why the available depth aspect treatments are not uniform.

\subsection{The approach of this paper}\label{sec:cj4t647ljn}
The volume bound controls elements of $H_{\tau_H}$ lying close to a certain subvariety, with ``close'' quantified by the ideal $\mathfrak{q}$.  We deduce it from a purely algebraic statement, Theorem~\ref{theorem:cj2dzotgu8}, whose informal content is that, among the equations defining that subvariety, there is at least one whose coefficients are at least as large as the distance from $a$ to $H Z$.

To formulate that algebraic statement, we denote now by $(\mathbf{G}, \mathbf{H}, \mathbf{M}, \mathbf{M}_H)$ the group schemes $(\GL_{n+1}, \GL_n, \Mat_{n+1}, \Mat_n)$ of invertible and all square matrices of the indicated dimensions, and write $\mathbf{Z}$ for the center of $\mathbf{G}$.  Let $R$ be a ring, and let $\tau \in \mathbf{M}(R)$ satisfy the following ``stability'' hypothesis (\S\ref{sec:20230514080526}): the characteristic polynomials of $\tau$ and its upper-left block $\tau_H$ generate the unit ideal.  Let $\mathbf{G}_\tau$ and $\mathbf{H}_{\tau_H}$ denote the centralizers of $\tau$ and $\tau_H$, regarded as subgroup schemes of $\mathbf{G}$ and $\mathbf{H}$ defined over $R$, and let $a \in \mathbf{G}(R)$.  For each $R$-algebra $R'$, we define the following set (compare with Problem~\ref{Probleminformal}):
\begin{equation*}
  \mathbf{X}_{\tau,a}(R') := \left\{ y \in \mathbf{H}_{\tau_H}(R') : a y \in \mathbf{H}(R') \mathbf{G}_{\tau}(R') \right\}.
\end{equation*}
The stability hypothesis on $\tau$ turns out to imply that $\mathbf{X}_{\tau,a}$ is naturally a closed subscheme of $\mathbf{H}_{\tau_H}$ over $R$, defined by finitely many polynomials equations (see Lemma~\ref{lemma:each-ring-extens-r-r-we-have-begin-mathbfhr-m}).

It is easy to see that if $a$ lies in $\mathbf{H}(R) \mathbf{Z}(R)$, then $\mathbf{X}_{\tau,a} = \mathbf{H}_{\tau_H}$ as schemes over $R$ (i.e., their point sets coincide for all $R'$).  We establish a converse:
\begin{theorem}[Theorem~\ref{theorem:main-transversality-general-ring}]\label{theorem:cj2dzotgu8}
  Assume that $n+1 \geq 3$, that $2$ is a unit in $R$, and that $\mathbf{X}_{\tau,a} = \mathbf{H}_{\tau_H}$ as schemes over $R$.  Then $a \in \mathbf{H}(R) \mathbf{Z}(R)$.
\end{theorem}
We refer to \S\ref{sec:transversality-statement-results} for refinements and discussion of why the hypotheses are necessary.  Informally, the condition $\mathbf{X}_{\tau,a} = \mathbf{H}_{\tau_H}$ says that the defining equations for $\mathbf{X}_{\tau,a}$ inside $\mathbf{H}_{\tau_H}$ are tautological, i.e., their coefficients are all zero.  The informal content of Theorem~\ref{theorem:cj2dzotgu8} is thus that if $a \notin \mathbf{H}(R) \mathbf{Z}(R)$, then we can find a nonzero (bounded degree) polynomial on $\mathbf{H}_{\tau_H}$ over $R$ whose zero locus contains $\mathbf{X}_{\tau,a}$.  This result, applied in the context of Problem~\ref{Probleminformal} with $R$ a suitable quotient of $\mathfrak{o}/\mathfrak{q}$, implies that the set whose volume we must bound is contained in the locus of a polynomial whose coefficients are not all too small.  The required uniform volume bound (Theorem~\ref{theorem:volume-bound}) then follows via general bounds for solutions to polynomial congruences (\S\ref{sec:cj3v60uw7y}).

To prove Theorem~\ref{theorem:cj2dzotgu8}, we apply the assumed equality between $\mathbf{X}_{\tau,a}$ and $\mathbf{H}_{\tau_H}$ first over the ring of dual numbers $R' = R[\eps]/(\eps^2)$, then over $R'' = R[\eps_1,\eps_2]/(\eps_1,\eps_2)$.  We refer to these steps as linear and quadratic analysis, respectively.  They may be understood as studying the consequences of the vanishing of linear and quadratic coefficients of the defining equations for $\mathbf{X}_{\tau,a}$ inside $\mathbf{H}_{\tau_H}$.  The arguments sketched below were found after extensive numerical study of these coefficients,\footnote{using SAGE~\cite{sage2023} (and its components Singular~\cite{DGPS} and GiNaC) and Emacs Calc} and represent a main novelty of this paper.

The key case is when $a \in \mathbf{G}_\tau(R)$; we must show then that $a \in \mathbf{Z}(R)$.

From the linear analysis, we deduce that $a^2 \in \mathbf{Z}(R)$.  To do so, we determine the linear coefficients with respect to the basis for $\Lie(\mathbf{H}_{\tau_H})$ given by powers of $\tau_H$, and observe that they may be related via an upper-triangular substitution to some invariants of $\tau$ and $a$ whose vanishing forces $a^2$ to be central.  We refer to \S\ref{sec:analysis-first-derivatives} for details.  This step of the argument generalizes the approach of~\cite{2020arXiv201202187N}, recalled below in \S\ref{sec:cj4t64wm4h}, which amounts to studying just the linear coefficient for the central direction in $\Lie(\mathbf{H}_{\tau_H})$.

From the quadratic analysis, we construct (Lemma~\ref{lemma:second-derivatives-yield-homomorphism-property}) a multiplicative linear map $\Lie(\mathbf{H}_{\tau_H}) \rightarrow \Lie(\mathbf{G}_{\tau})$.  Our argument then divides according to whether this map preserves unit elements (residually).  If it does not, then we may apply multiplicativity in the central direction to see that $a$ is central.  If it does, then we show by two applications of the linear analysis --- first for $a$ over $R$, then for suitable ``nearby'' $b \in \mathbf{G}_{\tau}(R')$ over $R' := R[\eps_2]/(\eps_2^2)$, with $a(1+\eps_1\tau_H) \in \mathbf{H}(R') b$ --- that $\tau$ satisfies a monic polynomial of degree $2$, contrary our assumption $n+1\geq 3$.  We refer to \S\ref{sec:analysis-second-derivatives} for details.

\begin{remark}
  In the depth aspect, we may assume via the ``near identity'' reduction (\S\ref{sec:cj4t69an73}) that $a$ is close to the identity. Then $a^2 \in \mathbf{Z}(R) \implies a \in \mathbf{Z}(R)$, so the ``linear'' part of the above argument suffices, yielding a simpler proof than in~\cite{2020arXiv201202187N} of the volume bound required in the depth aspect (see Remark~\ref{Remark:recoverLiealgproof}).
\end{remark}

\subsection{Available treatment in the depth aspect}\label{sec:cj4t64wm4h}
We indicate here why the approaches to the volume bound given in \cite[\S6]{2023arXiv2309.16667} and \cite[\S15--17]{2020arXiv201202187N} do not apply in horizontal aspects, motivating the approach described above and pursued in this paper.

\subsubsection{The short spectral projector $\omega$}
We construct in \S\ref{sec:cj4t6ynizh} a convolution kernel $\omega \in C_c^\infty(K)$, the local component of the global amplifier discussed after~\eqref{eq:cj3v66d7v4}.  It is supported on a certain open subgroup $J_\tau < K$, and given there by a multiple of a certain character of $J_\tau$, depending upon $\tau$ and some additional data.  Roughly, $J_\tau$ is the inverse image modulo $\mathfrak{q}$ of the image of $G_\tau \cap K$.  For example, when $\tau$ is a (regular) diagonal matrix, $J_\tau$ consists of elements of $K$ that are congruent modulo $\mathfrak{q}$ to the diagonal.

\subsubsection{Reductions}
The depth aspect case of the volume bound (i.e., $\mathfrak{p}$ fixed) was addressed in~\cite{2020arXiv201202187N} following a series of reductions:
\begin{enumerate}[(i)]
\item \textbf{Reduction to the near-identity case}. As explained in~\cite[\S15]{2020arXiv201202187N}, it is not necessary to solve Problem~\ref{Probleminformal} in general --- it suffices to treat the ``near-identity'' case, where $a$ lies in an arbitrarily small (but fixed) neighborhood of the identity (see \S\ref{sec:cj4t69an73}).  A similar reduction is employed in \cite[\S5]{2023arXiv2309.16667}.
\item \textbf{Reduction to the Lie algebra}. The near-identity case reduces further, via Lie-theoretic arguments~\cite[\S16]{2020arXiv201202187N}, to a Lie algebra problem, where $\Lie(H_{\tau_H})$ plays the role of $H_{\tau_H}$.
\item \textbf{Reduction to central directions}. The Lie algebra problem concerns certain subspaces of $\Lie(H_{\tau_H})$, whose definition we omit here.  The problem is to show that the subspaces are not the entire space.  This problem was addressed in~\cite[\S17]{2020arXiv201202187N} by showing that the subspaces do not contain the one-dimensional central subspace; in effect, this shows that central cosets in $H_{\tau_H} \cap K$ suffice to solve Problem~\ref{Probleminformal}.
\end{enumerate}
These reductions fail in horizontal aspects, for independent reasons:
\begin{enumerate}[(i)]
\item The near-identity reduction in the depth aspect comes from the freedom to take the convolution kernel $\omega$ supported close to the identity, which is unavailable in horizontal aspects (see \S\ref{sec:cj4t69an73}).
\item The reduction to the Lie algebra may be adapted beyond the near-identity case, but the resulting Lie algebra problem is unsolvable: there are natural families of counterexamples (see Theorem~\ref{theorem:characterization-of-tangential-points} and Remark~\ref{remark:cj4u2pnri2}).
\item Central directions do not suffice to solve the general case of Problem~\ref{Probleminformal}; we have again identified counterexamples (see Remark~\ref{remark:gl6-example}).
\end{enumerate}

\begin{remark}
  Marshall's arguments \cite[Prop 6.2]{2023arXiv2309.16667} do not reduce to the Lie algebra or to central directions, but do reduce to the near-identity case.  Moreover, his method is specific to the ``wall-avoidance'' case noted in \S\ref{sec:cj4vjryk3u}, which excludes the motivating twist aspect noted in Remarks~\ref{remark:cj3ubsydvn} and~\ref{remark:cj3u905a9t}.
\end{remark}

\subsubsection{The near-identity reduction}\label{sec:cj4t69an73}
We conclude \S\ref{sec:cj4t64wm4h} by explaining the ``near-identity'' reduction mentioned above, whose failure in horizontal aspects motivates why the geometric estimates in this paper are of ``algebro-geometric'' rather than ``Lie-algebraic'' flavor compared to those in~\cite{2020arXiv201202187N}.

Fix a small natural number $d$.  Let $K[d]$ denote the $d$th principal congruence subgroup of $K$.  The convolution kernel $\omega$ is supported in $K$, but not in $K[1]$.  It projects onto a ``short'' family: the integral operator $\pi(\omega)$ vanishes unless the irreducible representation $\pi$ lies in such a family. The normalized restriction $\omega^{[d]}$ of $\omega$ to $K[d]$ projects onto a larger family; in global settings, it is larger by a factor of roughly
\begin{equation}\label{eq:cj4t64h1mo}
  q^{{(n+1)}^2 d}.
\end{equation}
\begin{example}
  If $\tau$ is diagonal, then the operator $\pi(\omega)$ (resp.\ $\pi(\omega^{[d]})$) vanishes unless $\pi$ is a principal series representation $\chi_1 \boxplus \dotsb \boxplus \chi_{n+1}$ induced by characters $\chi_j$ of $F^\times$ having prescribed restrictions to $\mathfrak{o}^\times$ (resp.\ to $1 + \mathfrak{p}^d$).
\end{example}
The near-identity reduction arises from the freedom to replace $\omega$, supported on $G_\tau \cap K$, with $\omega^{(d)}$, supported on the smaller group $G_\tau \cap K[d]$.  We can do so in the depth aspect, where $q$ is fixed, because the factors~\eqref{eq:cj4t64h1mo} are then harmless.

In horizontal aspects --- where $q$ is large --- such factors ruin the near-identity reduction: any small power of $q$ saved by amplification is swamped by the large power of $q$ in~\eqref{eq:cj4t64h1mo}, yielding estimates that fail even to recover convexity.  We must thus work with projectors $\omega$ supported on the full maximal compact subgroup $K$.

\subsection{Organization of this paper}
This paper consists almost exclusively of local analysis at the ``interesting'' place $\mathfrak{p}$.  It culminates in our main local result, Theorem~\ref{theorem:main-local-result}, which is a direct analogue of its archimedean counterpart,~\cite[Theorem 4.2]{2020arXiv201202187N}.  The auxiliary arguments of~\cite[\S4--6]{2020arXiv201202187N} (concerning ``uninteresting'' places, amplification, counting, etc.)\ combine with our main local result to yield a subconvex bound, exactly as in~\cite{2020arXiv201202187N}.  While our local analysis is logically self-contained, its global motivation might be clarified by skimming~\cite[\S6]{2020arXiv201202187N}.

We conclude with the detailed breakdown.  \S\ref{sec:notation}--\S\ref{sec:cyclic-matrices} contain preliminaries.  \S\ref{sec:20230514080526} treats ``stability'' for $(\GL_{n+1}, \GL_n)$ over a general ring.  \S\ref{sec:transversality} establishes our key algebraic result (Theorem~\ref{theorem:main-transversality-general-ring}, or Theorem~\ref{theorem:cj2dzotgu8} above).  \S\ref{sec:volumebound} applies that result to derive the uniform volume bound (Theorem~\ref{theorem:volume-bound}), which we further apply in \S\ref{Sec:bilinear} to estimate bilinear forms relevant for~\eqref{eq:cj3tv4jpta}.  \S\ref{sec:part-body-paper} concerns ``microlocal analysis'' involving representations $\pi,\sigma$ and parameters $\tau,\tau_H$. \S\ref{sec:mainlocal} establishes our main local result (Theorem~\ref{theorem:main-local-result}), a uniform non-archimedean analogue of~\cite[Theorem 4.2]{2020arXiv201202187N}. \S\ref{Sec:final} combines our main local result with the auxiliary arguments of~\cite{2020arXiv201202187N} to complete the proof of Theorem~\ref{theorem:cj3ngw7u2s}.

\subsection{Acknowledgment}
Y.H.\ is supported by the National Key Research and Development Program of China (No. 2021YFA1000700).  P.N.\ worked on this project while he was a von Neumann Fellow at the Institute for Advanced Study, supported by the National Science Foundation under Grant No. DMS-1926686.  The paper was completed while P.N.\ was a Villum Investigator based at Aarhus University, supported by a research grant (VIL54509) from VILLUM FONDEN.  We are grateful to Peter Sarnak and Wei Zhang for their helpful feedback on an earlier draft.

\section{Notation}\label{sec:notation}
We record general notation and conventions, used throughout the paper.

\subsection{Vectors and matrices}\label{sec:vectors-matrices}
Let $\mathbf{V}$ be a finite free $\mathbb{Z}$-module, thus $\mathbf{V} \cong \mathbb{Z}^n$ for some $n$, the \emph{rank} of $\mathbf{V}$.  It will occasionally be convenient to denote the rank instead by $n+1$, so that $n$ refers instead to the rank of a codimension one submodule (see \S\ref{sec:general-linear-ggp-pairs} below).

We denote by $\mathbf{V}^* = \Hom(\mathbf{V},\mathbb{Z}) \cong \mathbb{Z}^n$ the dual module and by $\mathbf{M} := \End(\mathbf{V}) \cong \Mat_n(\mathbb{Z})$ the endomorphism ring.  We denote by juxtaposition the natural pairing
\begin{equation*}
  \mathbf{V}^* \otimes \mathbf{V} \rightarrow \mathbb{Z}, \quad \ell \otimes v \mapsto \ell v := \ell(v),
\end{equation*}
as well as the left action, right action and outer product, respectively:
\begin{align*}
  &\mathbf{M} \otimes \mathbf{V} \rightarrow \mathbf{V}, \quad   &&a \otimes v \mapsto a v := a(v), \\
  &\mathbf{V}^* \otimes \mathbf{M} \rightarrow \mathbf{V}^*, \quad  &&\ell \otimes a \mapsto \ell a := [u \mapsto \ell(a(u))] \\
  &\mathbf{V} \otimes \mathbf{V}^* \rightarrow \mathbf{M}, \quad &&v \otimes \ell \mapsto v \ell := [u \mapsto \ell(u) v]. 
\end{align*}
The motivation for the notation is that, by choosing a basis, we may regard $\mathbf{V}$ (resp.\ $\mathbf{V}^*$) as spaces of column (resp.\ row) vectors and $\mathbf{M}$ as a space of matrices, in which case the above pairings are given by matrix multiplication.  We use the same notation for the extensions of these pairings to more general rings given below.

For us, a \emph{ring} $R$ is a commutative ring with identity, an \emph{algebra} or \emph{ring extension} $R'$ of $R$ is a ring map $R \rightarrow R'$, and an \emph{affine scheme} $\mathbf{X}$ is a functor that assigns to each ring $R$ a set $\mathbf{X}(R)$ that functorially identifies with the set of ring maps $\mathbb{Z}[\mathbf{X}] \rightarrow R$ for some ring $\mathbb{Z}[\mathbf{X}]$, the \emph{coordinate ring} of $\mathbf{X}$.  Affine schemes over $R$ are defined similarly, but restricting to $R$-algebras; the above discussion then applies upon replacing $\mathbb{Z}$ with $R$.

By abuse of notation, we denote also by $\mathbf{V}, \mathbf{V}^*$ and $\mathbf{M}$ the affine group schemes over $\mathbb{Z}$ given by
\begin{align*}
  \mathbf{V}(R) &:= \mathbf{V} \otimes R \cong R^n, \\
  \mathbf{V}^*(R) &:= \mathbf{V}^* \otimes R \cong R^n, \\
  \mathbf{M}(R) &:= \End_R(\mathbf{V}(R)) \cong \Mat_n(R).
\end{align*}
The coordinate rings of $\mathbf{V}, \mathbf{V}^*$ and $\mathbf{M}$ identify with the polynomial rings over $\mathbb{Z}$ in $n, n$ and $n^2$ variables, respectively, where $n$ denotes the rank of $\mathbf{V}$.  We denote by $\mathbf{G}$ the affine group scheme
\begin{equation*}
  \mathbf{G}(R) := \Aut_R(\mathbf{V}(R)) \cong \GL_n(R),
\end{equation*}
whose coordinate ring is obtained from that of $\mathbf{M}$ by inverting the determinant.

We denote by $\mathbf{Z}$ the center of $\mathbf{G}$, consisting of scalar matrices:
\begin{equation*}
  \mathbf{Z}(R) = R^\times \hookrightarrow \mathbf{G}(R).
\end{equation*}

For $g \in \mathbf{G}(R)$, we denote by $\Ad(g) : \mathbf{M}(R) \rightarrow \mathbf{M}(R)$ the conjugation map
\begin{equation*}
  \Ad(g) x := g x g ^{-1}.
\end{equation*}

\subsection{General linear GGP pairs}\label{sec:general-linear-ggp-pairs}
In some parts of this paper, we consider just one general linear group $\mathbf{G}$.  In others, we work with inclusions $\mathbf{H} \leq \mathbf{G}$ of general linear groups of neighboring rank.  We then use the following notation.

Suppose given $e \in \mathbf{V}$ and $e^* \in \mathbf{V}^*$ such that $e^* e= 1$.  We then define the submodule
\begin{equation*}
  \mathbf{V}_H := \left\{ v \in \mathbf{V} : e^* v = 0 \right\},
\end{equation*}
which participates in the direct sum decomposition
\begin{equation}\label{eqn:lambda-=-lambda_h-oplus-mathbbz-e-}
  \mathbf{V} = \mathbf{V}_H \oplus \mathbb{Z} e,
\end{equation}
\begin{equation*}
  v = (1 -  e e^*) v + (e^* v) e.
\end{equation*}
This extends to any ring.  We define affine group schemes
\begin{equation*}
  \mathbf{V}_H, \quad \mathbf{V}_H^*, \quad \mathbf{M}_H, \quad \mathbf{H}
\end{equation*}
in terms of $\mathbf{V}_H$, by analogy to the definitions of \S\ref{sec:vectors-matrices}.  These define closed subgroup schemes of $\mathbf{V}, \mathbf{V}^*, \mathbf{M}, \mathbf{G}$, respectively.

For any ring $R$, we denote by $1_H \in \mathbf{M}_H(R)$ the identity operator.  It is given in terms of the identity operator $1 \in \mathbf{M}(R)$ by
\begin{equation*}
  1_H = 1 - e e^*.
\end{equation*}
Given $\tau \in \mathbf{H}(R)$, we denote by
\begin{equation*}
  \tau_H := 1_H \tau 1_H \in \mathbf{M}_H(R)
\end{equation*}
the element induced by the decomposition~\eqref{eqn:lambda-=-lambda_h-oplus-mathbbz-e-}.

\begin{example}
  Suppose $\mathbf{V}$ has rank three.  Let $e_1, e_2$ be a basis for $\mathbf{V}_H$.  Then $e_1,e_2,e$ is a basis for $\mathbf{V}$.  Using this basis to identify $\mathbf{M}$ with the space of $3 \times 3$ matrices, we have
  \begin{equation*}
    \mathbf{M}_H =
    \begin{pmatrix}
      \ast & \ast & 0 \\
      \ast & \ast & 0 \\
      0 & 0 & 0 \\
    \end{pmatrix},
    \quad 
    \mathbf{H} =
    \begin{pmatrix}
      \ast & \ast & 0 \\
      \ast & \ast & 0 \\
      0 & 0 & 1 \\
    \end{pmatrix},
    \quad
    1_H =
    \begin{pmatrix}
      1 & 0 & 0 \\
      0 & 1 & 0 \\
      0 & 0 & 0 \\
    \end{pmatrix},
  \end{equation*}
  \begin{equation*}
    \tau =
    \begin{pmatrix}
      \tau _{11} & \tau _{12} & \tau _{13} \\
      \tau _{21} & \tau _{22} & \tau _{23} \\
      \tau _{31} & \tau _{32} & \tau _{33} \\
    \end{pmatrix}
    \implies
    \tau_H
    =
    \begin{pmatrix}
      \tau _{11} & \tau _{12} & 0 \\
      \tau _{21} & \tau _{22} & 0 \\
      0 & 0 & 0 \\
    \end{pmatrix}.
  \end{equation*}
\end{example}

\subsection{Centralizers}
For $\tau \in \mathbf{M}(R)$ or $\sigma \in \mathbf{M}_{H}(R)$, we denote by
\begin{equation*}
  \mathbf{G}_\tau \leq \mathbf{G}, \quad \mathbf{M}_{\tau} \leq \mathbf{M},
  \quad   \mathbf{H}_{\sigma} \leq \mathbf{H}, \quad \mathbf{M}_{H,\sigma} \leq \mathbf{M}_H
\end{equation*}
the centralizers, regarded as affine group schemes over $R$.

\subsection{Local fields and congruence subgroups}\label{sec:local-fields-congruence-subgroups}
Let $F$ be a non-archimedean local field.  We denote then by
\begin{equation*}
  \mathfrak{o}, \qquad \mathfrak{p}, \qquad \varpi \in \mathfrak{p} ,\qquad q = \lvert \mathfrak{o} / \mathfrak{p}  \rvert
\end{equation*}
its ring of integers, its maximal ideal, a uniformizer, and the cardinality of the residue field.  We denote by $|.|_F$ the normalized valuation.

For each ideal $\mathfrak{a} \subseteq \mathfrak{o}$, we set
\begin{equation}\label{eqn:cj3m0de4ie}
  K(\mathfrak{a}) := \ker(\mathbf{G}(\mathfrak{o}) \rightarrow \mathbf{G}(\mathfrak{o}/\mathfrak{a})),
  \quad
  K_H(\mathfrak{a}) := \ker(\mathbf{H}(\mathfrak{o}) \rightarrow \mathbf{H}(\mathfrak{o}/\mathfrak{a})),
\end{equation}
We sometimes write simply
\begin{equation*}
  K := K(\mathfrak{o}) = \mathbf{G}(\mathfrak{o}), \quad K_H := K_H(\mathfrak{o}) = \mathbf{H}(\mathfrak{o}),
  \quad
  K_Z := \mathbf{Z}(\mathfrak{o})
\end{equation*}
for the standard maximal compact subgroups of $\mathbf{G}(F)$, $\mathbf{H}(F)$ and $\mathbf{Z}(F)$, respectively.  We equip the latter groups with the Haar measure that assigns volume one to their maximal compact subgroups $K _G, K _H$ and $K _Z$.

\section{Cyclic matrices}\label{sec:cyclic-matrices}
Let $R$ be a ring.  In this section, we write $V, M, H$, etc., for the sets of $R$-points $\mathbf{V}(R), \mathbf{M}(R), \mathbf{H}(R)$, etc., of the corresponding bold-faced group schemes.

Let $\tau \in M$.  We denote by
\begin{equation*}
  P_\tau(X) = \det(X - \tau) \in R[X]
\end{equation*}
its characteristic polynomial.  It is a monic polynomial whose degree is the rank of $\mathbf{V}$.  We recall the Cayley--Hamilton theorem (see~\cite[Theorem 4.3]{MR1322960}): $P_\tau(\tau) = 0$.

\begin{definition}\label{definition:cyclic}
  Let $\tau \in M$.  We say that $v \in V$ is \emph{$\tau$-cyclic} if $R[\tau] v = V$.  We say that \emph{$V$ is $\tau$-cyclic} (or simply that \emph{$\tau$ is cyclic} when $V$ is understood, or that \emph{$V$ is cyclic} when $\tau$ is understood) if there exists a $\tau$-cyclic vector $v \in V$.
\end{definition}
\begin{example}
  Any Jordan block is cyclic.
\end{example}
\begin{example}
  It follows from the Vandermonde determinant calculation that a diagonal matrix is cyclic when the pairwise differences between its diagonal entries are invertible.
\end{example}
\begin{example}
  If $v$ is $\tau$-cyclic, then it remains so upon passing to any ring extension $R'$ of $R$.  In particular, the cyclicity of $\tau$ is preserved under base extension.
\end{example}
\begin{example}
  Let $R$ be a local ring.  By Nakayama's lemma, we can test whether $\tau$ is cyclic over the residue field.  When $R$ is a field, the structure theorem for modules over the principal ideal domain $R[X]$ says that $V$ is isomorphic to a direct sum $\oplus_i R[X] / (p_i)$, where $p_i$ are monic primary polynomials (i.e., powers of irreducible polynomials) and $\tau$ acts as multiplication by $X$.  Then $\tau$ is cyclic if and only if the $p_i$ are pairwise relatively prime.  For instance, when $R$ is algebraically closed, this says that no two Jordan blocks for $\tau$ have the same eigenvalue.
\end{example}

\begin{lemma}\label{lemma:v-tau-cycl-if-only-if-map-modul-begin-rn-right-v-}
  Write $n = \rank(\mathbf{V})$.  Then $v$ is $\tau$-cyclic if and only if the map of modules
  \begin{equation}\label{eqn:rn-rightarrow-v}
    R^n \rightarrow V
  \end{equation}
  \begin{equation}\label{eqn:let-tau-in-c0-through-cn1-sum-cj-tau}
    (c_0,\dotsc,c_{n-1})  \mapsto \sum_{j < n} c_j \tau^j v
  \end{equation}
  is surjective, in which case it is an isomorphism.
\end{lemma}
\begin{proof}
  The first statement follows from the Cayley--Hamilton theorem, the second from the fact that any surjective map of free modules of the same finite rank is an isomorphism~\cite[Corollary 4.4a]{MR1322960}.
\end{proof}

In particular, each cyclic element $\tau$ admits a \emph{cyclic basis} $e_1, \dotsc, e_n$ such that
\begin{equation}\label{eqn:beginpmatrix-0--0--ast-}
  \tau e_j = e_{j+1} \text{ for } j < n,
  \quad
  \text{e.g., }
  \tau = 
  \begin{pmatrix}
    0 & 0 & \ast \\
    1 & 0 & \ast \\
    0 & 1 & \ast \\
  \end{pmatrix}.
\end{equation}

\begin{lemma}\label{lemma:centralizer-description}
  Let $\tau \in M$, and suppose that $v \in V$ is $\tau$-cyclic.  Then the map
  \begin{equation*}
    M_\tau \rightarrow V, \qquad x \mapsto x v
  \end{equation*}
  is an isomorphism of $R$-modules.  Moreover,
  \begin{equation*}
    M_\tau = R[\tau] = \left\{ \sum_{j=0}^{n-1} c_j \tau^j : (c_0,\dotsc,c_{n-1}) \in R^{n} \right\}.
  \end{equation*}
  In particular, $M_\tau$ is a commutative $R$-algebra, and $G_\tau$ is an abelian group.
\end{lemma}
\begin{proof}
  Since $V$ is spanned by the $\tau^j v$ and $x \tau^j v = \tau^j x v$, we see that $x$ is determined by $x v$.  The map $M_\tau \rightarrow V$ is thus injective.

  We next verify that $M_\tau = R[\tau]$.  Let $x \in M_\tau$.  By Lemma~\ref{lemma:v-tau-cycl-if-only-if-map-modul-begin-rn-right-v-}, we may choose $(c_0,\dotsc,c_{n-1}) \in R^n$ so that $x v = \sum _{j=0}^{n-1} c_j \tau^j v$.  By the noted injectivity of $M_\tau \rightarrow V$, it follows that $x = \sum_{j=0}^{n-1} c_j \tau^j$.  Thus $M_\tau \subseteq R[\tau]$.  The reverse containment is evident.

  The surjectivity of $M_\tau \rightarrow V$ follows from the identity $M_\tau = R[\tau]$ and Lemma~\ref{lemma:v-tau-cycl-if-only-if-map-modul-begin-rn-right-v-}.
\end{proof}

\begin{lemma}\label{lemma:regular-elements-same-characteristic-polynomial-are-conjugate}
  Suppose that $\tau_1, \tau_2 \in M$ are cyclic elements with the same characteristic polynomial.  Then they are conjugate under $G$.
\end{lemma}
\begin{proof}
  Let $e_1^{(j)},\dotsc,e_n^{(j)}$ be a cyclic basis for $\tau_1$ (resp.\ $\tau_2$).  By Cayley--Hamilton, we have $\tau_j e_n^{(j)} = \sum_{i=1}^{n} c_i e_i^{(j)}$, where the coefficients $c_i \in R$ depend only upon the characteristic polynomial, hence are the same for $\tau_1$ and $\tau_2$.  Let $g \in G$ be the unique map that sends $e_i^{(1)}$ to $e_i^{(2)}$.  Then $g \tau_1 = \tau_2 g$, so $\tau_1$ and $\tau_2$ are conjugate.
\end{proof}
\section{Stability}\label{sec:20230514080526}
\begin{definition}
  Let $R$ be a ring and $\tau \in \mathbf{M}(R)$.  We say that $\tau$ is \emph{stable} if the vectors $e$ and $e^*$ are both $\tau$-cyclic.  We denote by $\mathbf{M}_{\stab}(R)$ the subset of stable elements.
\end{definition}
\begin{remark}
  This condition was studied by Rallis--Schiffmann~\cite{2007arXiv0705.2168R} under the name ``regular'', and also in~\cite[\S14]{nelson-venkatesh-1}, where it was related (for $R$ a field of characteristic zero) to the geometric invariant theory notion of stability.  See especially~\cite[Thm 6.1]{2007arXiv0705.2168R} and~\cite[Lem 14.8]{nelson-venkatesh-1}.  Many of the results of this section were established in~\cite[\S13-14]{nelson-venkatesh-1} when $R$ is a field of characteristic zero, and in some respects, more systematically and generally (e.g., also for orthogonal GGP pairs).  We have chosen to give short proofs of what we require, repeating or adapting the arguments of the cited reference.  The adapted arguments could likely be extended to the orthogonal case along the same lines as in~\cite[\S14]{nelson-venkatesh-1}.
\end{remark}
\begin{example}\label{example:cj3twmtcpz}
  Suppose $R$ is a local ring, with maximal ideal $\mathfrak{p}$.  Let $\tau \in \mathbf{M}(R)$, with image $\bar{\tau} \in \mathbf{M}(R/\mathfrak{p})$. Nakayama's lemma implies that $\tau$ is stable if and only if $\bar{\tau}$ is stable.
\end{example}
We denote by $n+1$ the rank of $\mathbf{V}$, so that $n$ is the rank of $\mathbf{V}_H$.

\begin{remark}\label{remark:cj57b61swr}
  For a ring $R$, we define
  \begin{equation*}
    \Delta : \mathbf{M}(R) \rightarrow R
  \end{equation*}
  \begin{equation*}
    \Delta(\tau) := \det (e^*, e^* \tau, \dotsc, e^* \tau ^n ) \det (e, \tau e, \dotsc, \tau ^n e),
  \end{equation*}
  where $\det(\dotsb)$ denotes the determinant of the matrix having the indicated row vectors $e^* \tau^i \in \mathbf{V}^*(R)$ or column vectors $\tau^j e \in \mathbf{V}(R)$, with such vectors realized as $(n+1)$-tuples using some chosen basis for the free $\mathbb{Z}$-module $\mathbf{V}$.  Choosing a different basis has the effect of multiplying the two determinants by mutually inverse factors, hence has no effect on $\Delta(\tau)$.  This family of maps $\Delta$ defines a regular function on the scheme $\mathbf{M}$.  It follows from Lemma~\ref{lemma:v-tau-cycl-if-only-if-map-modul-begin-rn-right-v-} that $\tau$ is stable if and only if $\Delta(\tau)$ is a unit in $R$.  In particular, $\mathbf{M} _{\stab}$ defines an open affine subscheme of $\mathbf{M}$.
\end{remark}

In the remainder of this section, we focus on an individual (arbitrary) ring $R$ and abbreviate $V := \mathbf{V}(R), G := \mathbf{G}(R)$, etc.

For $\tau \in M$, we may form the characteristic polynomials
\begin{equation*}
  P_\tau \in R[X] \quad \text{and} \quad P_{\tau_H} \in R[X].
\end{equation*}
These are monic polynomials of degrees $n+1$ and $n$, respectively.

\begin{lemma}\label{lemma:let-p-p_h-in-rx-be-monic-polyn-degr-n+1-n-resp-the}
  Let $P, P_H \in R[X]$ be monic polynomials of degrees $n+1$ and $n$, respectively.  There exists $\tau \in M$ such that $P_{\tau} = P$ and $P_{\tau_H} = P_H$.
\end{lemma}
\begin{proof}
  One can deduce the existence of $\tau$ directly from the spherical property of $H \hookrightarrow G \times H$ (see~\cite[\S14.6]{nelson-venkatesh-1}).  For variety of exposition, we record an ``explicit'' construction, following~\cite[\S6.2]{MR3164988}.  Supposing for instance that $H = \GL_2(R) \hookrightarrow G = \GL_3(R)$, embedded as the upper-left block, we take
  \begin{equation*}
    \tau =
    \begin{pmatrix}
      a _1  & 1 & 0 \\
      a _2 & 0 & 1 \\
      b _2  & b _1  & b _0  \\
    \end{pmatrix}.
  \end{equation*}
  The coefficients of $\Of{charpoly}{\tau_H}$ are then $\pm a_j$, while the coefficients of $\Of{charpoly}{\tau}$ are of the form $\pm (b_j + c_j)$, where $c_j$ depends upon the $a_k$ and $b_l$ for $l<j$. By choosing suitable $a_j$ and $b_j$ inductively, we may thus arrange that $\tau$ and $\tau_H$ have prescribed characteristic polynomials.
\end{proof}

\begin{lemma}\label{lemma:stability-equivalences}
  For $\tau \in M$, the following are equivalent.
  \begin{enumerate}[(i)]
  \item\label{enumerate:p_tau-p_tau_h-generate-unit-ideal.-} $P_\tau$ and $P_{\tau_H}$ generate the unit ideal.
  \item\label{enumerate:there-are-no-nontr-tau-invar-subsp-v_h-or-v_h.-} There are no nonzero $\tau$-invariant submodules of $V_H$ or of $V_H^*$.
  \item\label{enumerate:vectors-e-f-are-tau-cyclic.-} $\tau \in M_{\stab}$.
  \end{enumerate}
\end{lemma}
\begin{proof}
  We show first that~\eqref{enumerate:p_tau-p_tau_h-generate-unit-ideal.-} implies~\eqref{enumerate:there-are-no-nontr-tau-invar-subsp-v_h-or-v_h.-}. By assumption, we may write $a P_\tau + b P_{\tau_H} = 1$ for some $a,b \in R[X]$.  Since $P_\tau(\tau) = 0$, it follows that $b(\tau) P_{\tau_H}(\tau) = 1$.  On the other hand, if $U$ is a $\tau$-invariant submodule of $V_H$ and $u \in U$, then $f(\tau) u = f(\tau_H) u$ for each $f \in R[x]$, hence $P_{\tau_H}(\tau) u = P_{\tau_H}(\tau_H) u = 0$.  Therefore $u =0$, as required.

  We show next that~\eqref{enumerate:there-are-no-nontr-tau-invar-subsp-v_h-or-v_h.-} implies~\eqref{enumerate:vectors-e-f-are-tau-cyclic.-}.  We must check that $e$ and $e^*$ are $\tau$-cyclic.  Suppose otherwise that $R[\tau] e \neq V$ or that $e^* R[\tau] \neq V^*$.  In the first case, we set $U := {(R[\tau] e)}^\perp \subseteq V^*$; in the second, $U := {(e^* R [\tau ])}^\perp \subseteq V$.  In either case, we see that $U$ is a non-zero $\tau$-invariant submodule of $V_H^*$ or $V_H$, contrary to~\eqref{enumerate:there-are-no-nontr-tau-invar-subsp-v_h-or-v_h.-}.

  We show finally that~\eqref{enumerate:vectors-e-f-are-tau-cyclic.-} implies~\eqref{enumerate:p_tau-p_tau_h-generate-unit-ideal.-}.  Suppose that $\tau$ is stable, but that the ideal $\mathfrak{a} \subseteq R[X]$ generated by $P_\tau$ and $P_{\tau_H}$ is not the unit ideal (1).  Let $\mathfrak{m}$ be a maximal ideal that contains $\mathfrak{a}$, and set $\mathfrak{p} := R \cap \mathfrak{m}$.  The stability of $\tau$ and the fact that $\mathfrak{a} \neq (1)$ are unaffected by modding out by $\mathfrak{p}$, passing to the field of fractions of $R / \mathfrak{p}$, and then passing to the algebraic closure of that field.  We thereby reduce to the case that $R$ is an algebraically closed field.  Then $P_\tau$ and $P_{\tau_H}$ share a common root $c$.  The remainder of the proof closely follows that of~\cite[Lemma 14.4]{nelson-venkatesh-1}.  We may find
  \begin{itemize}
  \item an eigenvector $v \in V_H$ for $\tau_H$ with eigenvalue $c$, and also
  \item an eigenvector $\ell \in V^*$ for $\tau$ with eigenvalue $c$.
  \end{itemize}
  Then
  \begin{equation}\label{eqn:ell-cdot-tau-c-cdot-v-=-0.-}
    \ell \cdot (\tau - c) \cdot v = 0.
  \end{equation}
  We consider two cases.
  \begin{itemize}
  \item $(\tau - c) v = 0$, so that $v \in V_H$ is an eigenvector for $\tau$.
  \item $(\tau - c) v \neq 0$.  Since $(\tau_H - c) v = 0$, it follows that $e$ is a multiple of $(\tau - c) v$, hence by~\eqref{eqn:ell-cdot-tau-c-cdot-v-=-0.-} that $\ell e = 0$, so that $\ell \in V_H^*$ is an eigenvector for $\tau$.
  \end{itemize}
  In either case, we have produced an eigenvector for $\tau$ in $V_H$ or in $V_H^*$.  Since $R$ is a field, the orthogonal complement of that eigenvector defines a proper $\tau$-invariant subspace of $V^*$ (resp.\ $V$) that contains the vector $e^*$ (resp.\ $e$), contrary to our assumption that that vector is cyclic.
\end{proof}

\begin{lemma}\label{lemma:stable-implies-trivial-stabilizer}
  Let $\tau \in M_{\stab}$.  Then $H \cap G_\tau = \{1\}$.
\end{lemma}
\begin{proof}
  This follows from the fact that $e$ is $\tau$-cyclic and fixed by $H$, which shows that if $h \in H$ commutes with $\tau$, then it fixes all of $V$, hence $h = 1$.
\end{proof}

\begin{lemma}\label{lemma:tau-stable-implies-tauH-cyclic}
  Let $\tau \in M_{\stab}$.  Then $\tau_H$ is cyclic, that is to say, $V_H$ and, equivalently, $V_H^*$, are $\tau_H$-cyclic.
\end{lemma}
\begin{proof}
  We will show that $V_H^*$ is cyclic; a similar argument applies to $V_H$.  We will verify more precisely that
  \begin{equation*}
    \ell : V_H \rightarrow R, \qquad \ell(v) := e^* \tau v
  \end{equation*}
  defines a $\tau_H$-cyclic vector $\ell \in V_H^*$, that is to say,
  \begin{equation}\label{eqn:ell-rtau_h-=-v_h.-}
    \ell R[\tau_H] = V_H^*.
  \end{equation}
  In verifying this, we may assume that $R$ is a local ring, and then (by Nakayama's lemma) that $R$ is a field.  In that case, if~\eqref{eqn:ell-rtau_h-=-v_h.-} fails, then the orthogonal complement of $\ell R[\tau_H]$ is a nonzero $\tau_H$-invariant subspace $U$ of $V_H$.  By construction, $0 = \ell(U) = e^* \tau U$, so $\tau|_{U} = \tau_H|_{U}$.  Thus $U$ is in fact a $\tau$-invariant subspace of $V_H$, which by part~\eqref{enumerate:there-are-no-nontr-tau-invar-subsp-v_h-or-v_h.-} of Lemma~\ref{lemma:stability-equivalences} yields $U = 0$, giving the required contradiction.
\end{proof}

\section{Transversality}\label{sec:transversality}

\subsection{Statement of results}\label{sec:transversality-statement-results}
For $\tau \in \mathbf{M}_{\stab}(R)$, $a \in \mathbf{G}(R)$, we define a subfunctor $\mathbf{X}_{\tau,a}$ of $\mathbf{H}_{\tau_H}$ over $R$, as follows: for each ring extension $R'$ of $R$,
\begin{equation*}
  \mathbf{X}_{\tau,a}(R') := \left\{ y \in \mathbf{H}_{\tau_H}(R') : a y \in \mathbf{H}(R') \mathbf{G}_{\tau}(R') \right\}.
\end{equation*}
As we explain below (Lemma~\ref{lemma:each-ring-extens-r-r-we-have-begin-mathbfhr-m}), it defines a closed subscheme.

We observe that if $a$ lies in $\mathbf{H}(R) \mathbf{Z}(R)$, then $\mathbf{X}_{\tau,a} = \mathbf{H}_{\tau_H}$, that is, their point sets coincide for all ring extensions $R'$ of $R$.  Indeed, if $a = h z$ with $(h,z) \in \mathbf{H}(R) \times  \mathbf{Z}(R)$, then for each $y \in \mathbf{H}_{\tau_H}(R')$, we have
\begin{equation*}
  a y = h z y = (h y) z \in \mathbf{H}(R') \mathbf{G}_{\tau}(R').
\end{equation*}
The main result of this section is the following converse:

\begin{theorem}\label{theorem:main-transversality-general-ring}
  Assume that $\rank(\mathbf{V}) \geq 3$.  Let $R$ be a ring in which $2$ is a unit.  Let $(\tau,a) \in \mathbf{M}_{\stab}(R) \times \mathbf{G}(R)$.  Suppose that $\mathbf{X}_{\tau,a} = \mathbf{H}_{\tau_H}$ (as schemes over $R$).  Then $a \in \mathbf{H}(R) \mathbf{Z}(R)$.
\end{theorem}

\begin{remark}[Sharpness with respect to the rank]
  The conclusion of Theorem~\ref{theorem:main-transversality-general-ring} clearly fails when $\rank(\mathbf{V}) = 1$, since in that case, $\mathbf{H}_{\tau_H}$ is trivial.  The conclusion may fail also when $\rank(\mathbf{V}) = 2$.  For example, take $\mathbf{V} = \mathbb{Z}^2$ with standard basis $e_1,e_2$.  Let $e_1^*,e_2^*$ denote the corresponding dual basis of $\mathbf{V}^*$.  Take $e := e_2$ and $e^* := e_2^*$, so that
  \begin{equation*}
    \mathbf{H} =
    \begin{pmatrix}
      \ast & 0 \\
      0 & 1 \\
    \end{pmatrix} \leq \mathbf{G}.
  \end{equation*}
  Take
  \begin{equation*}
    \tau := a :=
    \begin{pmatrix}
      0 & 1 \\
      1 & 0 \\
    \end{pmatrix}.
  \end{equation*}
  Then for any ring $R$, we see that $a$ does not lie in $\mathbf{H}(R) \mathbf{Z}(R)$ (although it does normalize it).  On the other hand, for any
  \begin{equation*}
    y
    =
    \begin{pmatrix}
      y_1 & 0 \\
      0 & 1 \\
    \end{pmatrix}
    \in \mathbf{H}_{\tau_H}(R) = \mathbf{H}(R),
  \end{equation*}
  we have
  \begin{equation*}
    a y = h b \quad \text{ with }
    h :=
    \begin{pmatrix}
      y_1^{-1} & 0 \\
      0 & 1 \\
    \end{pmatrix} \in \mathbf{H}(R),
    \quad
    b :=
    \begin{pmatrix}
      0 & y_1 \\
      y_1 & 0 \\
    \end{pmatrix} \in \mathbf{G}_\tau(R).    
  \end{equation*}
  Thus $\mathbf{X}_{\tau,a} = \mathbf{H}_{\tau_H}$.  The proof of Theorem~\ref{theorem:main-transversality-general-ring} can be adapted to show that these are essentially the only counterexamples when $\rank(\mathbf{V}) = 2$.  This discussion suggests that Theorem~\ref{theorem:main-transversality-general-ring} should be related to the fact that $\mathbf{H} \mathbf{Z}$ has trivial normalizer in $\mathbf{G}$ when $\rank(\mathbf{V}) \geq 3$, although we did not spot a direct way to relate the two conditions.
\end{remark}

\begin{remark}[Sharpness with respect to the ring]\label{remark:we-have-seen-comp-calc-with-grobn-bases-that-when-}
  We have seen by computer calculation with Gr\"{o}bner bases that when $\rank(\mathbf{V}) \in \{3,4\}$, there exist fields $R$ of characteristic $2$ over which there are counterexamples to the conclusion of Theorem~\ref{theorem:main-transversality-general-ring}.  We see no reason why such examples should not exist in any rank.
\end{remark}

Theorem~\ref{theorem:main-transversality-general-ring} is a consequence of some more precise assertions that we now formulate.

\begin{definition}\label{definition:let-r-be-ring.-let-a-in-mathbfgr-y-in-mathbfx_t-ar}
  Let $R$ be a ring.  We define the extension rings
  \begin{equation*}
    R' := R[\eps]/(\eps^2), \quad
    R'' := R[\eps_1,\eps_2]/(\eps_1^2,\eps_2^2),
  \end{equation*}
  which come with natural maps $R' \rightarrow R$ and $R'' \rightarrow R$ obtained by sending $\eps, \eps_1, \eps_2$ to zero.  For a scheme $\mathbf{X}$, we say that $y' \in \mathbf{X}(R')$ lies over $y \in \mathbf{X}(R)$ if $y' \mapsto y$ under the induced map $\mathbf{X}(R') \rightarrow \mathbf{X}(R)$, and similarly for $R''$.

  Let $\tau \in \mathbf{M}_{\stab}(R)$, $a \in \mathbf{G}(R)$ and $y \in \mathbf{X}_{\tau,a}(R)$.
  \begin{enumerate}
  \item We say that $\mathbf{X}_{\tau,a}$ is \emph{tangential at $y$ over $R$} if for all $y' \in \mathbf{H}_{\tau_H}(R')$ lying over $y$, we have $y' \in \mathbf{X}_{\tau,a}(R')$.
  \item We say that $\mathbf{X}_{\tau,a}$ is \emph{doubly-tangential at $y$ over $R$} if for all $y'' \in \mathbf{H}_{\tau_H}(R'')$ lying over $y$, we have $y'' \in \mathbf{X}_{\tau,a}(R'')$.
  \end{enumerate}
\end{definition}
The informal content is that $\mathbf{X}_{\tau,a}$ is tangential (resp.\ doubly-tangential) at $y$ over $R$ if the polynomials whose vanishing defines $\mathbf{X}_{\tau,a}$ over $R$ have the property that their Taylor series at $y$ vanish to order at least $2$ (resp.\ $3$).

\begin{theorem}\label{theorem:characterization-of-tangential-points}
  Retain the setting of Definition~\ref{definition:let-r-be-ring.-let-a-in-mathbfgr-y-in-mathbfx_t-ar}.  The following are equivalent:
  \begin{enumerate}[(i)]
  \item $\mathbf{X}_{\tau,a}$ is tangential at $y$ over $R$.
  \item For the (unique) elements $(h,b) \in \mathbf{H}(R) \times \mathbf{G}_\tau(R)$ defined by writing $a y = h b$, we have $b^2 \in \mathbf{Z}(R)$.
  \end{enumerate}
\end{theorem}

\begin{theorem}\label{theorem:reta-sett-defin-refd-r-be-ring.-let-mathbfgr-y-mat-doubly-tangential}
  Retain the setting of Definition~\ref{definition:let-r-be-ring.-let-a-in-mathbfgr-y-in-mathbfx_t-ar}.  Assume that
  \begin{itemize}
  \item $2$ is a unit in $R$, and
  \item $\rank(\mathbf{V}) \geq 3$.
  \end{itemize}
  Then the following are equivalent:
  \begin{enumerate}[(i)]
  \item $\mathbf{X}_{\tau,a}$ is doubly-tangential at $y$ over $R$.
  \item We have $a \in \mathbf{H}(R) \mathbf{Z}(R)$.
  \end{enumerate}
\end{theorem}

Theorem~\ref{theorem:reta-sett-defin-refd-r-be-ring.-let-mathbfgr-y-mat-doubly-tangential} implies in particular that, under the hypotheses of Theorem~\ref{theorem:main-transversality-general-ring}, if $a \notin \mathbf{H}(R) \mathbf{Z}(R)$, then $\mathbf{X}_{\tau,a}$ fails to be doubly-tangential at every point $y \in \mathbf{X}_{\tau,a}(R)$.  In particular, $\mathbf{X}_{\tau,a} \neq \mathbf{H}_{\tau_H}$, so the conclusion of Theorem~\ref{theorem:main-transversality-general-ring} holds.  The remainder of this section is devoted to the proofs of Theorems~\ref{theorem:characterization-of-tangential-points} and~\ref{theorem:reta-sett-defin-refd-r-be-ring.-let-mathbfgr-y-mat-doubly-tangential}.

\subsection{Defining $\mathbf{X}_{\tau,a}$ by polynomial equations}\label{Sec:polyEqn}
Let $R$ be a ring, and let $\tau \in \mathbf{M}_{\stab}(R)$.  By Lemma~\ref{lemma:centralizer-description}, the maps
\begin{align*}
  &\bullet e : \mathbf{M}_{\tau}(R) \rightarrow \mathbf{V}(R), \quad &&x \mapsto x e \\
  &e^* \bullet  : \mathbf{M}_{\tau}(R) \rightarrow \mathbf{V}^*(R), \quad &&x \mapsto e^* x 
\end{align*}
are linear isomorphisms (i.e., isomorphisms of $R$-modules).  We denote by
\begin{align*}
  &{[\bullet e]}^{-1} : \mathbf{V}(R) \rightarrow \mathbf{M}_\tau(R), \quad \\
  &{[e^* \bullet]}^{-1} : \mathbf{V}^*(R) \rightarrow \mathbf{M}_\tau(R) 
\end{align*}
their inverses.  These extend to linear isomorphisms over any ring extension $R'$ of $R$.

We note the following equivariance property:
\begin{lemma}\label{lemma:equivariance-property-of-bullet-e-e-star-inverse-maps}
  For any $v \in \mathbf{V}(R)$ (resp.\ $v^* \in \mathbf{V}^*(R)$) and $a \in \mathbf{G}_{\tau}(R)$, we have
  \begin{equation*} {[\bullet e]}^{-1}(a v) = a {[\bullet e]}^{-1}(v), \qquad {[e^* \bullet ]}^{-1}(v^* a) = {[e^* \bullet]}^{-1}(v^*) a.
  \end{equation*}
\end{lemma}
\begin{proof}
  By definition, ${[\bullet e]}^{-1}(v)$ is the unique element of $\mathbf{M}_\tau(R)$ with ${[\bullet e]}^{-1}(v) e = v$, while ${[\bullet e]}^{-1}(a v)$ is the unique element with ${[\bullet e]}^{-1}(a v) e = a v$.  It is clear then that $a {[\bullet e]}^{-1}(v) = {[\bullet e]}^{-1}(a v)$.  A similar argument gives the second identity.
\end{proof}

\begin{lemma}\label{lemma:each-ring-extens-r-r-we-have-begin-mathbfhr-m}
  For each ring extension $R'$ of $R$, we have
  \begin{equation}\label{eqn:mathbfhr-mathbfg_t-=-left-g-in-mathbfgr-:-bull-e-1}
    \mathbf{H}(R') \mathbf{G}_\tau(R') =
    \left\{
      g \in \mathbf{G}(R') :
      {[e^* \bullet]}^{-1}( e^* g) {[\bullet e]}^{-1}(g^{-1} e)  = 1
    \right\}.
  \end{equation}
  In particular, $R' \mapsto \mathbf{H}(R') \mathbf{G}_\tau(R')$ defines a closed subscheme of $\mathbf{G}$ over $R$.
\end{lemma}
\begin{proof}
  Suppose $g = h b$ with $(h,b) \in \mathbf{H}(R') \times \mathbf{G}_\tau(R')$.  Then $g^{-1} e = b^{-1} e$ and $e^* g = e^* b$, so
  \begin{equation*}
    {[e^* \bullet]}^{-1}( e^* g) {[\bullet e]}^{-1}(g^{-1} e) = b b^{-1} = 1.
  \end{equation*}
  
  Conversely, suppose $g$ belongs to the right hand side of~\eqref{eqn:mathbfhr-mathbfg_t-=-left-g-in-mathbfgr-:-bull-e-1}, Then, defining $b, b' \in \mathbf{M}_{\tau}(R')$ by $b' := {[\bullet e]}^{-1}(g^{-1} e)$ and $b := {[e^* \bullet]}^{-1}(e^* g)$, we have $b' b = 1$.  It follows that $b \in \mathbf{G}_\tau(R')$.  Moreover, $h := g b^{-1}$ satisfies $h e = e$ and $e^* h = e^*$, hence $h \in \mathbf{H}(R')$.
\end{proof}
\begin{corollary}\label{corollary:let-begin-f-:-mathbfh_t-right-mathbfm_t-}
  Let
  \begin{equation*}
    f : \mathbf{H}_{\tau_H} \rightarrow \mathbf{M}_\tau
  \end{equation*}
  denote the map of schemes over $R$ defined by
  \begin{equation*}
    f(y) := {[e^* \bullet]}^{-1}( e^* a y) {[\bullet e]}^{-1}({(a y)}^{-1} e).
  \end{equation*}
  Then
  \begin{equation*}
    \mathbf{X}_{\tau, a}(R')
    =
    \left\{
      y \in \mathbf{H}_{\tau_H}(R') :
      f(y) = 1
    \right\}.
  \end{equation*}
  In particular, $\mathbf{X}_{\tau,a}$ defines a closed subscheme of $\mathbf{H}_{\tau_H}$ over $R$.
\end{corollary}

\begin{remark}\label{remark:rewrite-defining-equations-for-X-as-polynomials}
  We may rewrite the defining equation $f(y) = 1$ for $\mathbf{X}_{\tau,a}$ using Cramer's rule for $y^{-1}$, as follows:
  \begin{equation*} {[e^* \bullet]}^{-1}( e^* a y) {[\bullet e]}^{-1}(y^{\mathrm{adj}} a^{-1} e) = \det(y).
  \end{equation*}
  Here $y^{\mathrm{adj}}$ denotes the adjugate or cofactor matrix, characterized by the identity $y y^{\mathrm{adj}} = \det(y)$.  Writing $n+1 = \rank(\mathbf{V})$, so that $\mathbf{M}_{\tau}$ has rank $n+1$, one can see that $\mathbf{X}_{\tau,a}$ is defined by $n+1$ polynomial equations of degree at most $n$ in the entries of $y$.  Indeed, one can parameterize $\mathbf{H}_{\tau_H}$ using Lemma~\ref{lemma:centralizer-description} with $n$ variables, say, $y_1,\cdots y_n$.  Then
  \begin{itemize}
  \item each entry of the row vector $e^* a y$ is an affine function of the $y_i$ (i.e., a linear combination of the $y_i$ and $1$),
  \item the coefficients of ${[e^* \bullet]}^{-1}( e^* a y)\in \mathbf{M}_\tau$ are likewise affine functions of the $y_i$,
  \item the matrices $y^{\mathrm{adj}}, y^{\mathrm{adj}} a^{-1} e$ and ${[\bullet e]}^{-1}(y^{\mathrm{adj}} a^{-1} e) $ have entries given by polynomials in the $y_i$ of total degree at most $n-1$, and
  \item $\det(y)$ is a polynomial in the $y_i$ of total degree $n$.
  \end{itemize}
  Finally, the condition
  \begin{equation*} {[e^* \bullet]}^{-1}( e^* a y) {[\bullet e]}^{-1}(y^{\mathrm{adj}} a^{-1} e) = \det(y)
  \end{equation*}
  is an equality of elements of $\mathbf{G}_\tau$, whose $n+1$ coefficients give rise to $n+1$ polynomial equations of degree at most $n$ in the $y_i$.
\end{remark}

\subsection{Reduction to the centralizer}\label{sec:reduction-centralizer}
The set $\mathbf{X}_{\tau,a}(R)$ is empty unless we can find $h \in \mathbf{H}(R)$ and $y \in \mathbf{H}_{\tau_H}(R)$ so that $b := h a y \in \mathbf{G}_\tau(R)$.  Then for all ring extensions $R'$ of $R$, we have
\begin{equation*}
  \mathbf{X}_{\tau,a}(R') = y \mathbf{X}_{\tau,b}(R').
\end{equation*}
Furthermore, for each point $z\in \mathbf{X}_{\tau, b}(R)$, the following are equivalent:
\begin{enumerate}
\item $\mathbf{X}_{\tau,a}$ is tangential (resp.\ doubly-tangential) at $yz$.
\item $\mathbf{X}_{\tau,b}$ is tangential (resp.\ doubly-tangential) at $z$.
\end{enumerate}
This last equivalence holds in particular for $z = 1$, which lies in $\mathbf{X}_{\tau,b}(R)$ because $b$ lies in $\mathbf{G}_\tau(R)$.  In this way, many questions concerning the $\mathbf{X}_{\tau,a}$ may be reduced to the case $a \in \mathbf{G}_\tau(R)$ and $y = 1$.

\begin{definition}
  Given $a \in \mathbf{G}_\tau(R)$, we denote by
  \begin{equation*}
    \mu, \nu : \mathbf{M}_{H,\tau_H} \rightarrow \mathbf{M}_{\tau}
  \end{equation*}
  the linear maps of free $R$-modules given by
  \begin{equation*}
    \mu(u) := {[e^* \bullet]}^{-1}( e^* a u a^{-1} ),
  \end{equation*}
  \begin{equation*}
    \nu(u) :=  {[\bullet e]}^{-1}(a u a^{-1}  e).
  \end{equation*}
\end{definition}
\begin{lemma}\label{lemma:cj56eyn0v4}
  For $a \in \mathbf{G}_\tau(R)$, the map $f$ from Corollary~\ref{corollary:let-begin-f-:-mathbfh_t-right-mathbfm_t-} is given by
  \begin{equation*}
    f(y) = \bigl(1 + \mu(y-1)\bigr) \bigl(1 + \nu(y^{-1} - 1)\bigr).
  \end{equation*}
\end{lemma}
\begin{proof}
  By the equivariance property noted in Lemma~\ref{lemma:equivariance-property-of-bullet-e-e-star-inverse-maps}, we have, for any extension $R'$ of $R$ and $y \in \mathbf{H}_{\tau_H}(R')$,
  \begin{equation*} [e^\ast \bullet ]^{-1} (e^\ast a y) = [e^\ast \bullet ]^{-1} (e^\ast a y a^{-1} ) a,
  \end{equation*}
  \begin{equation*} [\bullet e]^{-1} ( (a y)^{-1} e) = a^{-1} [\bullet e]^{-1} ( (a y a^{-1} )^{-1} e),
  \end{equation*}
  thus
  \begin{equation*}
    f(y) =
    [e^\ast \bullet ]^{-1} (e^\ast a y a^{-1} )
    [\bullet e]^{-1} ( (a y a^{-1} )^{-1} e).
  \end{equation*}
  The claimed formula follows by writing $y = 1 + (y-1)$ and $y^{-1} = 1 + (y^{-1} - 1)$, so that the parenthetical terms lie in $\mathbf{M}_{H,\tau_H}(R')$, and using that $[e^\ast \bullet ]^{-1} (e^\ast ) = [\bullet e ]^{-1} (e) = 1 \in \mathbf{M}_{\tau}(R')$.
\end{proof}

\subsection{Linear analysis}\label{sec:analysis-first-derivatives}
In this section, we prove Theorem~\ref{theorem:characterization-of-tangential-points}.  By replacing $a$ with $a y$, we may reduce to the case $y = 1$.  By then replacing $a$ with $h^{-1} a$, we may reduce further to the case that $a = b \in \mathbf{G}_\tau(R)$.  The required equivalence then follows from that between~\eqref{itemize:mathbfx_tau-a-has-first-order-tangency-at-1.-} and~\eqref{itemize:we-have-a2-in-mathbfzr.-} in Proposition~\ref{lemma:let-r-be-ring-tau-in-mathbfm_st-a-in-mathbfg_t-fol}, below.

\begin{proposition}\label{lemma:let-r-be-ring-tau-in-mathbfm_st-a-in-mathbfg_t-fol}
  Let $R$ be a ring, $\tau \in \mathbf{M}_{\stab}(R)$, $a \in \mathbf{G}_\tau(R)$.  The following are equivalent:
  \begin{enumerate}[(i)]
  \item\label{itemize:mathbfx_tau-a-has-first-order-tangency-at-1.-} $\mathbf{X}_{\tau,a}$ is tangential at $1$ over $R$.
  \item\label{itemize:line-isom-begin-mu-nu-:-mathbfm_h-tau_hr-right-mat} The maps $\mu$ and $\nu$ attached to $a$ as in \S\ref{sec:reduction-centralizer} satisfy
    \begin{equation}\label{eqn:mu-=-nu.-}
      \mu = \nu.
    \end{equation}
  \item\label{itemize:j-geq-0-define-begin-a_j-:=-e-tauj-e-quad-b_j-:=-e} For $j \geq 0$, define
    \begin{equation*}
      A_j := e^* a \tau^j e, \quad B_j := e^* a^{-1} \tau^j e.
    \end{equation*}
    Then
    \begin{equation}\label{eqn:a_j-a-1-=-b_j-}
      A_j a^{-1} = B_j a.
    \end{equation}
  \item\label{itemize:we-have-a2-in-mathbfzr.-} We have $a^2 \in \mathbf{Z}(R)$.
  \end{enumerate}
\end{proposition}

The proof uses the following key lemma.
\begin{lemma}\label{lemma:there-exist-c_ij-in-rtau-0-leq-i-leq-j-with-c_i-i-}
  There exist $c_{ij} \in M_\tau$ ($0 \leq i \leq j$), with $c_{i i} = 1$, so that for each $j \geq 0$,
  \begin{equation}\label{eqn:mu-tau_hj--nu-tau_hj--=-sum_i-leq-j-c_i-j--left-a}
    \mu  (\tau_H^j ) - \nu (\tau_H^j )
    = \sum_{i \leq j}
    c_{i j }  \left( B_i a - A_i a^{-1} \right).
  \end{equation}
  Here and henceforth, we adopt the convention
  \begin{equation*}
    \tau_H^0 := 1_H.
  \end{equation*}
\end{lemma}
\begin{proof}
  We induct on $j$.  We consider first the case $j = 0$.  By definition, $\nu(1_H)$ is the element of $\mathbf{M}_\tau(R)$ such that
  \begin{equation*}
    \nu(1_H) e = a 1_H a ^{-1} e.
  \end{equation*}
  Expanding out $1_H = 1 - e e^*$ gives
  \begin{equation*}
    a 1_H a^{-1} e = e - a e e^* a ^{-1} e.
  \end{equation*}
  Recognizing $e^* a^{-1} e$ as the scalar $B_0$, we deduce that
  \begin{equation*}
    \nu(1_H) e = (1 - B_0 a) e.
  \end{equation*}
  Since $1 - B_0 a \in \mathbf{M}_{\tau}(R)$, it follows from the definition of $\nu$ that
  \begin{equation*}
    \nu(1_H) = 1 - B_0 a.
  \end{equation*}
  A similar argument gives
  \begin{equation}\label{eqn:nu1_h-=-1-a_0-a-1.-}
    \mu(1_H) = 1 - A_0 a^{-1}.
  \end{equation}
  Thus the required identity~\eqref{eqn:mu-tau_hj--nu-tau_hj--=-sum_i-leq-j-c_i-j--left-a} holds for $j = 0$ (with $c_{0 0} = 1$).

  We now aim to verify~\eqref{eqn:mu-tau_hj--nu-tau_hj--=-sum_i-leq-j-c_i-j--left-a} for given $j \geq 1$, assuming that it holds for all smaller values of $j$.  We expand
  \begin{align*}
    \tau_H^j &= (1 - e e^*) \tau (1 - e e^\ast) \tau \dotsb (1 - e e^*) \tau (1 - e e^*) \\
          &=
            \tau^j - e e^* \tau^j - \tau^j e e^* + \dotsb,
  \end{align*}
  where $\dotsb$ denotes a linear combination of terms of the form
  \begin{equation}\label{eqn:tauj_1-e-e-tauj_2-quad-0-leq-j_1-j_2--j-}
    \tau^{j_1} e e^* \tau^{j_2} \quad (0 \leq j_1, j_2 < j)
  \end{equation}
  that is symmetric with respect to $j_1$ and $j_2$.  For instance, when $j = 1$, we have
  \begin{equation*}
    \dotsb =
    e e^\ast \tau e e^\ast
    =
    e (e^\ast \tau e) e^\ast
    =
    (e^* \tau e) e e^*,
  \end{equation*}
  which is a multiple of the vector $e e^\ast$ (i.e.,~\eqref{eqn:tauj_1-e-e-tauj_2-quad-0-leq-j_1-j_2--j-} with $j_1 = j_2 = 0$).  Arguing like in the $j=0$ case, we obtain
  \begin{align}\label{eqn:20230516005320}
    \mu(\tau_H^j) &= \tau^j - A_j a^{-1} - A_0 a^{-1} \tau ^j + \dotsb, \\  \nonumber
    \nu(\tau_H^j) &= \tau^j - B_j a - B_0 a \tau ^j + \dotsb,
  \end{align}
  where $\dotsb$ denotes a linear combination of terms given in the first case by $A_{j_1} \tau^{j_2} a^{-1}$ and in the second by $B_{j_1} a \tau^{j_2}$.  Taking differences, we obtain
  \begin{equation*}
    \mu(\tau_H^j) - \nu(\tau_H^j) = (B_j a - A_j a^{-1} ) + \tau^j (B_0 a -  A_0 a^{-1}) + \dotsb,
  \end{equation*}
  where $\dotsb$ denotes a linear combination of differences
  \begin{equation*}
    \tau^{j_2} (B_{j_1} a - A_{j_1} a^{-1}).
  \end{equation*}
  with $j_1,j_2 < j$.
\end{proof}

\begin{proof}[Proof of Proposition~\ref{lemma:let-r-be-ring-tau-in-mathbfm_st-a-in-mathbfg_t-fol}]
  We check first that~\eqref{itemize:mathbfx_tau-a-has-first-order-tangency-at-1.-} is equivalent to~\eqref{itemize:line-isom-begin-mu-nu-:-mathbfm_h-tau_hr-right-mat}.  For $u \in \mathbf{M}_{H,\tau_H}(R)$, set
  \begin{equation*}
    y := 1 + \eps u \in \mathbf{H}_{\tau_H}(R').
  \end{equation*}
  Then $y^{-1} = 1 - \eps u$.  By Lemma~\ref{lemma:cj56eyn0v4} and the linearity of $\mu$ and $\nu$, we obtain
  \begin{align*}
    f(y) &= \left( 1 + \eps \mu(u) \right) \left( 1 - \eps \nu(u) \right) \\
         &= 1 + \eps \left( \mu(u) - \nu(u) \right).
  \end{align*}
  Thus $f(y) = 1$ for all $u$ if and only if the identity~\eqref{eqn:mu-=-nu.-} holds.

  To verify the equivalence between~\eqref{itemize:line-isom-begin-mu-nu-:-mathbfm_h-tau_hr-right-mat} and~\eqref{itemize:j-geq-0-define-begin-a_j-:=-e-tauj-e-quad-b_j-:=-e}, observe first that powers $\tau_H^j$ for $j \geq 0$ span $\mathbf{M}_{H,\tau_H}(R)$, so the identity~\eqref{eqn:mu-=-nu.-} is equivalent to the identity
  \begin{equation}\label{eqn:mutau_hj-=-nutau_hj-quad-text-all--j-geq-0.-}
    \mu(\tau_H^j) = \nu(\tau_H^j)
  \end{equation}
  holding for all $j \geq 0$.  By inductive application of Lemma~\ref{lemma:there-exist-c_ij-in-rtau-0-leq-i-leq-j-with-c_i-i-}, we see that~\eqref{eqn:mutau_hj-=-nutau_hj-quad-text-all--j-geq-0.-} holds for all $j$ if and only if~\eqref{eqn:a_j-a-1-=-b_j-} does.  The required equivalence follows.

  We verify finally that~\eqref{itemize:j-geq-0-define-begin-a_j-:=-e-tauj-e-quad-b_j-:=-e} and~\eqref{itemize:we-have-a2-in-mathbfzr.-} are equivalent.  The relation~\eqref{eqn:a_j-a-1-=-b_j-} maybe be rewritten
  \begin{equation*}
    B_j a^2 = A_j 1_G.
  \end{equation*}
  For notational clarity in what follows, let us choose a basis of $\mathbf{V}(R)$, so that we may describe $a^2$ and $1_G$ by their matrix entries $a_{ij}$ and $\delta_{i j}$.  Let us also define the row vectors $A = (A _0, \dotsc, A _n )$ and $B = (B _0, \dotsc, B _n)$.  Then we further rewrite
  \begin{equation}\label{eqn:delta_i-j-=-b-a2_ij-.-} {(a^2)}_{ij } B = \delta_{i j} A.
  \end{equation}
  Let $P$ denote the matrix with columns $e, \tau e, \dotsc, \tau^n e$.  Since $\tau \in \mathbf{M}_{\stab}(R)$, the matrix $P$ is invertible.  We compute
  \begin{equation*}
    A = e^* a P, \quad  B = e^* a^{-1} P,
  \end{equation*}
  \begin{equation*}
    B P^{-1} a e = e^* e = 1, \quad 
  \end{equation*}
  \begin{equation*}
    A P^{-1} a e = e^* a^2 e.
  \end{equation*}
  Multiplying the relation~\eqref{eqn:delta_i-j-=-b-a2_ij-.-} on the right by $P^{-1} a e$ gives
  \begin{equation*} {(a^2)}_{i j} = \delta_{i j} e^* a^2 e,
  \end{equation*}
  hence
  \begin{equation*}
    a^2 = (e^* a^2 e) 1_G \in \mathbf{Z}(R).
  \end{equation*}
  This shows that~\eqref{itemize:j-geq-0-define-begin-a_j-:=-e-tauj-e-quad-b_j-:=-e} implies~\eqref{itemize:we-have-a2-in-mathbfzr.-}.  The converse may be verified by reversing the above steps, or directly from the definitions.
\end{proof}

\begin{remark}\label{Remark:recoverLiealgproof}
  The arguments used to establish Theorem~\ref{theorem:characterization-of-tangential-points} lead to a much simpler proof of the following generalization of \cite[Theorem 16.3]{2020arXiv201202187N}: \emph{for each ring} $R$ \emph{with} $2 \in R^\times$\emph{, each} $\tau \in \mathbf{M}_{\stab}(R)$ \emph{and} $x \in \mathbf{M}_\tau(R) - R$\emph{, we have}
  \begin{equation}\label{eqn:x-1_h-notin-mathbfm_hr-+-mathbfm_taur.-} [x,1_H] \notin \mathbf{M}_{H}(R) + \mathbf{M}_{\tau}(R).
  \end{equation}
  (The quoted result asserts the same conclusion when $R$ is a field of characteristic zero.)  To see this, suppose otherwise that $[x,1_H] \in t + \mathbf{M}_H(R)$ for some $t \in \mathbf{M}_\tau(R)$.  Then
  \begin{equation}\label{eqn:e-t-=-ex-1_h-quad-t-e-=-x-1_h-e.-}
    e^* t = e^*[x,1_H], \quad t e = [x,1_H] e.
  \end{equation}
  Using that $1_H = 1 - e e^*$ and identities like $e^* x e e^* = e^* (e^* x e)$, we compute that
  \begin{equation}\label{eqn:e-x-1_h-=-e-x-e-x-e-quad-x-1_h-e-=-x}
    e^* [x,1_H] = e^* (x - e^* x e), \quad [x,1_H] e = - (x - e^* x e) e.
  \end{equation}
  The stability of $\tau$ implies that $t \in \mathbf{M}_{\tau}(R)$ is determined by either of the identities in~\eqref{eqn:e-t-=-ex-1_h-quad-t-e-=-x-1_h-e.-}.  Since both factors in~\eqref{eqn:e-x-1_h-=-e-x-e-x-e-quad-x-1_h-e-=-x} lie in $\mathbf{M}_{\tau}(R)$, it follows that
  \begin{equation*}
    x - e^* x e = t = - x + e^* x e,
  \end{equation*}
  which simplifies to $2 x = 2 e^* x e$.  Since $2 \in R^\times$, we deduce that $x = e^* x e \in R$, contrary to assumption.
\end{remark}

\begin{remark}
  There are many examples of $\tau$ for which we have the implication
  \begin{equation*}
    a \in \mathbf{G}_\tau(R),
    \,
    a^2 \in \mathbf{Z}(R) \implies a \in \mathbf{Z}(R).
  \end{equation*}
  For instance, this implication holds if $R$ is a field of characteristic $\neq 2$ and $a$ is a cyclic element that is not semisimple.  Theorem~\ref{theorem:characterization-of-tangential-points} implies that the conclusion of Theorem~\ref{theorem:main-transversality-general-ring} holds for such $\tau$.
\end{remark}

\begin{remark}\label{remark:cj4u2pnri2}
  The basic cases where Theorem~\ref{theorem:characterization-of-tangential-points} on its own does not suffice to prove Theorem~\ref{theorem:main-transversality-general-ring} are when, with respect to some basis, $\tau$ and $a$ are both diagonal, the entries of $a$ equal to $\pm 1$ but not all the same (cf. Remark~\ref{remark:gl6-example}).  When $R$ is a field, the stability assumption then says that $e$ and $e^*$ are vectors (with inner product $e^* e = 1$) all of whose entries are nonzero.
\end{remark}

\subsection{Quadratic analysis}\label{sec:analysis-second-derivatives}
Here we prove Theorem~\ref{theorem:reta-sett-defin-refd-r-be-ring.-let-mathbfgr-y-mat-doubly-tangential}, following several preparatory lemmas.

\begin{lemma}\label{lemma:second-derivatives-yield-homomorphism-property}
  Let $R$ be ring, $\tau \in \mathbf{M}_{\stab}(R)$ and $a \in \mathbf{G}_{\tau}(R)$.  Assume that $\mathbf{X}_{\tau,a}$ is doubly-tangential at $1$ over $R$; in particular, it is tangential, so by Proposition~\ref{lemma:let-r-be-ring-tau-in-mathbfm_st-a-in-mathbfg_t-fol}, the linear maps $\mu$ and $\nu$ from \S\ref{sec:reduction-centralizer} coincide.  Then
  \begin{equation}\label{eqn:cool-homomorphism-property-u1-u2}
    2 \left( \mu (u _1 u _2 ) - \mu (u _1 ) \mu (u _2 ) \right) = 0
    \quad
    \text{ for all } u_1, u_2 \in \mathbf{M}_{H,\tau_H}(R).
  \end{equation}
\end{lemma}
\begin{proof}
  With $y := 1 + \eps_1 u_1 + \eps_2 u_2$, we see by a short calculation that
  \begin{equation*}
    y^{-1} = 1 - \eps_1 u_1 - \eps_2 u_2 + 2 \eps_1 \eps_2 u_1 u_2,
  \end{equation*}
  hence
  \begin{equation*}
    \mu (y - 1) = \eps _1 \mu (u _1 ) + \eps _2  \mu (u _2 ).
  \end{equation*}
  \begin{equation*}
    \nu(y^{-1} - 1)
    =  - \eps _1  \nu(u _1 )
    - \eps _2 \nu (u _2 ) + 2 \eps _1 \eps _2 \nu (u _1 u _2 ),
  \end{equation*}
  By our doubly-tangential hypothesis and Lemma~\ref{lemma:cj56eyn0v4}, we obtain the relation
  \begin{equation*}
    \bigl( 1 + \eps_1 \mu (u_1 ) + \eps_2 \mu (u_2 ) \bigr)
    \bigl( 1 - \eps_1 \nu (u_1 ) - \eps_2 \nu (u_2 ) + 2 \eps_1 \eps_2 \nu (u_1 u_2 ) \bigr)
    =
    1.
  \end{equation*}
  Expanding this out, and using that $\mu = \nu$, gives
  \begin{equation*}
    2 \mu (u _1 u _2 )
    =
    \mu  (u _1 ) \mu  (u _2 )
    +
    \mu (u _2 ) \mu (u _1 ).
  \end{equation*}
  The claimed identity follows now from the commutativity of $\mathbf{M}_\tau(R)$ (Lemma~\ref{lemma:centralizer-description}).
\end{proof}
\begin{lemma}\label{lemma:A0-times-a-minus-A0-vanishes}
  Under the same hypotheses as Lemma~\ref{lemma:second-derivatives-yield-homomorphism-property}, and with $A_0 := e^* a e$ as in Proposition~\ref{lemma:let-r-be-ring-tau-in-mathbfm_st-a-in-mathbfg_t-fol}, we have
  \begin{equation}\label{eqn:2-a_0--a_0-=-0.-}
    2 A_0 ( a - A_0) = 0.
  \end{equation}
\end{lemma}
\begin{proof}
  We specialize~\eqref{eqn:cool-homomorphism-property-u1-u2} to $u_1 = u_2 = 1_H$.  The identity~\eqref{eqn:nu1_h-=-1-a_0-a-1.-} for $\mu = \nu$ gives
  \begin{equation*}
    \mu(1_H) = 1 - A_0 a^{-1}.
  \end{equation*}
  A short calculation then gives
  \begin{equation*}
    A_0(a - A_0) = a^2 (\mu(1_H) - {\mu(1_H)}^2).
  \end{equation*}
  The conclusion now follows from~\eqref{eqn:cool-homomorphism-property-u1-u2}.
\end{proof}

\begin{remark}\label{remark:sketch-proof-when-A0-A1-vanish}
  We pause to sketch a proof of Theorem~\ref{theorem:main-transversality-general-ring} in the special case
  \begin{equation}\label{eq:cj3szcth5l}
    A_0 = A_1 = 0
  \end{equation}
  by an argument that seems more illuminating than the general argument given below.  (We note that the hypothesis~\eqref{eq:cj3szcth5l} excludes the case $a \in \mathbf{H}(R) \mathbf{Z}(R)$, since then $A_0$ is a unit.)  Since $A_0 = 0$, we have
  \begin{equation*}
    \mu(1_H) = 1.
  \end{equation*}
  Since $A_1 = 0$, we see from the formula~\eqref{eqn:20230516005320} (specialized to $j=1$) that
  \begin{equation*}
    \mu(\tau_H) = \tau.
  \end{equation*}
  Since $2 \in R^\times$, we deduce from~\eqref{eqn:cool-homomorphism-property-u1-u2} the homomorphism property
  \begin{equation*}
    \mu(u_1) \mu(u_2) = \mu(u_1 u_2).
  \end{equation*}
  By iterating this and applying the Cayley--Hamilton theorem to $\tau_H$, we obtain
  \begin{equation*}
    0 = \mu(P_{\tau_H}(\tau_H)) = P_{\tau_H}(\mu(\tau_H)) = P_{\tau_H}(\tau).
  \end{equation*}
  Since $\tau$ is stable, the characteristic polynomials $P_\tau$ and $P_{\tau_H}$ generate the unit ideal (Lemma~\ref{lemma:stability-equivalences}), contradicting the Cayley--Hamilton theorem for $\tau$.
\end{remark}

\begin{remark}\label{remark:gl6-example}
  The argument sketched in~\ref{remark:sketch-proof-when-A0-A1-vanish}, while aesthetically pleasing, does not suffice to establish the general case of Theorem~\ref{theorem:main-transversality-general-ring}.  We record a counterexample with $\rank(\mathbf{V}) = 6$.  Take $R = \mathbb{Q}(\alpha)$ with $\alpha^2 = 2$,
  \begin{equation*}
    \tau = \diag(0,1,2,2 \alpha,1 + 2 \alpha,2 + 2 \alpha),
  \end{equation*}
  \begin{equation*}
    e^* = (1,1,1,1,1,1), \quad e = \tfrac{1}{6} {(1,1,1,1,1,1)}^t,
  \end{equation*}
  \begin{equation*}
    a = \diag(1,1,1,-1,-1,-1).
  \end{equation*}
  One can verify that $A_0 =B_0= 0$, $A_1 = -\alpha$,
  \begin{equation*}
    \mu(1_H) = 1, \quad \mu(\tau_H) =
    \diag \left(
      \alpha, \alpha + 1, \alpha + 2, \alpha, \alpha + 1, \alpha + 2
    \right),
  \end{equation*}
  and that $P_{\tau_H}(\mu(\tau_H)) = 0$.  We have checked by computer calculation with Gr\"{o}bner bases that no similar examples exist when $\rank(\mathbf{V}) \in \{3,4\}$.

  One can also check in this example that $\mathbf{X}_{\tau,a}$ contains the center of $\mathbf{H}$.  (Of course, once we have proved Theorem~\ref{theorem:main-transversality-general-ring}, we will know that $\mathbf{X}_{\tau,a}$ does not contain the full centralizer $\mathbf{H}_{\tau_H}$.)
\end{remark}

The following lemma contains the final main idea for the proof of Theorem~\ref{theorem:reta-sett-defin-refd-r-be-ring.-let-mathbfgr-y-mat-doubly-tangential}.
\begin{lemma}\label{lemma:let-r-mathfr-be-local-ring-with-2-in-rtim-let-tau-a-in-Z-or-tau-in-rank-two}
  Let $(R,\mathfrak{m})$ be a local ring with $2 \in R^\times$.  Let $\tau \in \mathbf{M}_{\stab}(R)$, $a \in \mathbf{G}_{\tau}(R)$.  Assume that $\mathbf{X}_{\tau,a}$ is doubly-tangential at $1$ over $R$.  Then one of the following is true:
  \begin{enumerate}[(i)]
  \item $a \in \mathbf{Z}(R)$.
  \item $a \notin \mathbf{Z}(R)$, $a^2 \in \mathbf{Z}(R)$ and
    \begin{equation*}
      \tau \in R a + R \subseteq \mathbf{M}(R).
    \end{equation*}
  \end{enumerate}
\end{lemma}
\begin{proof}
  The hypotheses of Lemma~\ref{lemma:A0-times-a-minus-A0-vanishes} apply, so the identity~\eqref{eqn:2-a_0--a_0-=-0.-} holds.  The argument divides according to whether $A_0$ lies in $\mathfrak{m}$ or not. If $A_0 \notin \mathfrak{m}$, so that $A_0$ is a unit, then from~\eqref{eqn:2-a_0--a_0-=-0.-} and our assumption $2 \in R^\times$, we see that $a \in \mathbf{Z}(R)$.  The remainder of the proof concerns the case $A_0 \in \mathfrak{m}$.

  By Theorem~\ref{theorem:characterization-of-tangential-points} and our hypothesis that $\mathbf{X}_{\tau,a}$ is (doubly-)tangential at $1$ over $R$, we have
  \begin{equation*}
    a^2 \in \mathbf{Z}(R).
  \end{equation*}
  This hypothesis further implies that for
  \begin{equation*}
    R' := R[\eps] / (\eps^2), \quad
    y := 1 + \eps \tau _H \in \mathbf{H}_{\tau_H}(R'),
  \end{equation*}
  there exists $(h,b) \in \mathbf{H}(R') \times \mathbf{G}_\tau(R')$ such that $a y = h b$.  It implies finally that $\mathbf{X}_{\tau,a}$ is tangential at $y$ over $R'$, hence by Theorem~\ref{theorem:characterization-of-tangential-points} that
  \begin{equation*}
    b^2 \in \mathbf{Z}(R').
  \end{equation*}

  We retain the notation $A_j, B_j$ from Proposition~\ref{lemma:let-r-be-ring-tau-in-mathbfm_st-a-in-mathbfg_t-fol}.  From the calculation
  \begin{equation*}
    e^* a \tau_H = e^* a (1 - e e^*) \tau ( 1-  e e^*)
    =  e^* \left( a \tau - A_0 \tau - A_1 + A_0 (e^* \tau e) \right),
  \end{equation*}
  we see that
  \begin{equation*}
    b = a + \eps \left(a \tau - A _0 \tau - A _1 + A _0 (e^* \tau e)\right).
  \end{equation*}
  Squaring this relation gives
  \begin{equation}\label{eqn:b2-=-a2-+-2-eps-lefta-tau-_0-tau-_1-+-_0-e-tau}
    b^2 = a^2 + 2 \eps a \left(a \tau - A _0 \tau - A _1 + A _0 (e^* \tau e)\right),
  \end{equation}
  using here that $a$ and $\tau$ commute.  Using now that both $a^2 $ and $b^2$ lie in $\mathbf{Z}(R')$, together with our assumption $2 \in R^\times$, we deduce that
  \begin{equation*}
    a (a - A_0) \tau \in R a + R \subseteq \mathbf{M}(R).
  \end{equation*}
  Using next our assumption $A_0 \in \mathfrak{m}$, we see that $a(a-A_0)$ is invertible.  Using again that $a^2 \in \mathbf{Z}(R)$, we conclude that $\tau \in R a + R$.
\end{proof}

\begin{proof}[Proof of Theorem~\ref{theorem:reta-sett-defin-refd-r-be-ring.-let-mathbfgr-y-mat-doubly-tangential}]
  We have already noted, at the start of \S\ref{sec:transversality-statement-results}, that if $a \in \mathbf{H}(R) \mathbf{Z}(R)$, then $\mathbf{X}_{\tau,a} = \mathbf{H}_{\tau_H}$; in particular, $\mathbf{X}_{\tau,a}$ is doubly-tangential at $y$ over $R$.  The converse is the interesting direction.

  Suppose, thus, that $2$ is a unit in $R$, $\rank(\mathbf{V}) \geq 3$, $\tau \in \mathbf{M}_{\stab}$ and $a \notin\mathbf{H}(R) \mathbf{Z}(R)$, but $\mathbf{X}_{\tau,a}$ is doubly-tangential at some $y \in \mathbf{X}_{\tau,a}(R)$ over $R$.  We aim to derive a contradiction.

  By passing to the localization of $R$ at a prime $\mathfrak{p}$ for which the image of $a$ does not lie in $\mathbf{H}(R_\mathfrak{p}) \mathbf{Z}(R_\mathfrak{p})$, we may assume that $R$ is a local ring.  We may reduce further to the case $a \in \mathbf{G}_\tau(R) - \mathbf{Z}(R)$ and $y=1$, for the same reasons as noted at the beginning of \S\ref{sec:analysis-first-derivatives}.

  Since $\mathbf{X}_{\tau,a}$ is doubly-tangential at $1$ over $R$, we know from Theorem~\ref{theorem:characterization-of-tangential-points} and Lemma~\ref{lemma:let-r-mathfr-be-local-ring-with-2-in-rtim-let-tau-a-in-Z-or-tau-in-rank-two} that $a^2 \in \mathbf{Z}(R)$ and $\tau \in R a + R$.  In particular, we may find $c_0,c_1,c_2 \in R$ so that $\tau = c_0 + c_1 a$ and $a^2 = c_2$.  Then $\tau ^2 = c _0 ^2 + c _1^2 c _2 + 2 c _0 c _1 a$, hence $\tau ^2 - 2 c _0 \tau + c _0 ^2 - c _1 ^2 c _2 =0$.  Since $\tau$ is stable, the vector $e$ is $\tau$-cyclic, so
  \begin{equation*}
    \mathbf{V}(R) = R [\tau] e = R e + R \tau e.
  \end{equation*}
  Since the rank of $\mathbf{V}$ is $\geq 3$, we obtain a contradiction.
\end{proof}
\section{Volume bounds}\label{sec:volumebound}
We now apply the results of \S\ref{sec:transversality} to deduce a uniform solution to Problem~\ref{Probleminformal}.

\subsection{Bounds for polynomial congruences}\label{sec:cj3v60uw7y}
Let $n$ and $d$ be natural numbers.  In \S\ref{sec:cj3v60uw7y}, we use the equivalent notations $A \ll B$ and $A = \O(B)$ to denote that $|A| \leq C |B|$, where $C$ depends at most upon $n$ and $d$.

\begin{lemma}\label{lemma:hypersurface-count}
  Let $F$ be a finite field, of cardinality $q$.  Let $P \in F[x_1,\dotsc,x_n]$ be a polynomial of degree at most $d$ whose coefficients are not all zero.  Then
  \begin{equation*}
    \lvert \{x \in F^n : P(x) = 0\} \rvert \ll q^{n-1}.
  \end{equation*}
\end{lemma}
\begin{proof}
  By the fundamental theorem of algebra, we know that for all but $\O(1)$ many $x_1$, the polynomial that we get by specializing to that value is nonzero.  By iterating this observation for $x_2, x_3$, and so on, we obtain the required estimate.
\end{proof}

\begin{lemma}\label{lemma:let-d-geq-0.-supp-given-subs-mathc-subs-mathfr-p}
  Let $m$ be a natural number.  Suppose given a subset $\mathcal{D} \subseteq {(\mathfrak{o}/\mathfrak{p})}^n$ and a polynomial $P \in (\mathfrak{o}/\mathfrak{p}^m)[X_1,\dotsc,X_n]$ of degree $\leq d$.  Let $\mathcal{D}_m \subseteq {(\mathfrak{o}/\mathfrak{p}^m)}^n$ denote the inverse image of $\mathcal{D}$.
  \begin{enumerate}[(i)]
  \item\label{enumerate:assume-that-each-y-in-mathc-one-line-tayl-coeff-p-} Assume that for each $y \in \mathcal{D}_m$, at least one of the linear Taylor coefficients for $P$ at $y$ is a unit.  Then
    \begin{equation*}
      \lvert \left\{ y \in \mathcal{D}_m : P(y) \in \mathfrak{p}^m \right\} \rvert \ll q^{m n - m}.
    \end{equation*}
  \item\label{enumerate:assume-that-each-y-in-mathc-one-line-or-quadr-tayl} Assume that for each $y \in \mathcal{D}_m$, at least one of the linear or quadratic Taylor coefficients for $P$ at $y$ is a unit.  Then
    \begin{equation*}
      \lvert \left\{ y \in \mathcal{D}_m : P(y) \in \mathfrak{p}^m \right\} \rvert \ll q^{m n - \lceil m/2 \rceil}.
    \end{equation*}
  \end{enumerate}
\end{lemma}
\begin{proof}
  In the case $m=1$, either assertion reduces to the hypersurface bound (Lemma~\ref{lemma:hypersurface-count}).

  The case $m \geq 2$ of part~\eqref{enumerate:assume-that-each-y-in-mathc-one-line-tayl-coeff-p-} reduces via Hensel's lemma to the case $m = 1$.  Indeed, under the stated assumptions, we have more precisely that each fiber of $\mathcal{D}_m \rightarrow \mathcal{D}$ contains at most $q^{(m-1)(n-1)}$ elements $y$ for which $P(y) \in \mathfrak{p}^m$.  To see this, suppose for instance that the coefficient of $X_1$ in the Taylor expansion of $P$ at $y$ is a unit.  We may then use Hensel's lemma and the condition $P(y) \in \mathfrak{p}^m$ to determine $y_1 \in \mathfrak{o}/\mathfrak{p}^m$ in terms of the other coefficients.  The number of possibilities for each of $y_2,\dotsc,y_n$ is at most $q^{m-1}$, giving the required estimate.

  We turn to part~\eqref{enumerate:assume-that-each-y-in-mathc-one-line-or-quadr-tayl}.  The case $m=2$ follows formally from the case $m=1$, using that solutions modulo $\mathfrak{p}^2$ map to solutions modulo $\mathfrak{p}$.  We address the cases $m \geq 3$ by induction.  We first observe that if $P(y) \in \mathfrak{p}^m$, then certainly $P(y) \in \mathfrak{p}$, and the latter condition depends only upon the class of $y$ modulo $\mathfrak{p}$.  By the case $m=1$ that we have already addressed, the number of such classes is $\ll q^{n-1}$.  By replacing $P$ with a translate, we reduce to showing that when $P(0) \in \mathfrak{p}^m$, we have
  \begin{equation*}
    \lvert \left\{ y \in {(\mathfrak{p}/\mathfrak{p}^m)}^n : P(y) \in \mathfrak{p}^m \right\} \rvert \ll q^{m n  - \lceil m/2 \rceil - (n-1)}.
  \end{equation*}
  We now consider two cases separately:

  The first is when some linear Taylor coefficient of $P$ is a unit, say the coefficient of $X_1$.  Then, arguing as in part~\eqref{enumerate:assume-that-each-y-in-mathc-one-line-tayl-coeff-p-}, we see that the cardinality in question is at most $q^{(m-1) (n-1)}$, which is better than the required estimate because $m - 1 \geq \lceil m/2 \rceil$ for $m \geq 3$.

  The second is when all linear Taylor coefficients of $P$ lie in $\mathfrak{p}$.  We then form the polynomial
  \begin{equation*}
    Q(y) := \varpi^{-2} P(\varpi y) \in (\mathfrak{o} / \mathfrak{p}^{m-2})[X_1,\dotsc,X_n]
  \end{equation*}
  and we see that
  \begin{equation*}
    \lvert \left\{ y \in {(\mathfrak{p}/\mathfrak{p}^m)}^n : P(y) \in \mathfrak{p}^m \right\} \rvert
    =
    q^n \lvert \left\{ y \in {(\mathfrak{o}/\mathfrak{p}^{m-2})}^n : Q(y) \in \mathfrak{p}^{m-2} \right\} \rvert.
  \end{equation*}
  By hypothesis, some quadratic Taylor coefficient of $P$ is a unit, hence the same holds for $Q$.  We may thus apply our inductive hypothesis to bound the cardinality on the right hand side, giving that the left hand side satisfies the estimate
  \begin{equation*}
    \ll q^n q ^{(m- 2) n - \lceil (m-2)/2 \rceil}
    =
    q ^{m n - \lceil m/2 \rceil - (n-1)},
  \end{equation*}
  as required.
\end{proof}

\subsection{Norms and distance functions}\label{sec:norms-distance-functions}
For each $m \geq 0$, we equip the vector space $F^m$ with the norm
\begin{equation*}
  \left\lvert  (v_1,\dotsc,v_m)  \right\rvert_F := \max_{1 \leq j \leq m} \lvert v _j  \rvert_F.
\end{equation*}
We apply this notation more generally to any vector space that comes with a natural basis, e.g., to the space of $m_1 \times m_2$ matrices over $F$.

We choose a basis $\{e_1,\dotsc,e_n\}$ of $\mathbf{V}_H(\mathfrak{o})$, so that $\{e_1,\dotsc,e_n,e\}$ is a basis of $\mathbf{V}(\mathfrak{o})$, and use these bases to identify $\mathbf{V}(F), \mathbf{V}^*(F)$ and $\mathbf{G}(F)$ with matrices over $F$.  The above notation then applies.

For $g \in \mathbf{G}(F)$, we write
\begin{equation*}
  g =
  \begin{pmatrix}
    a & b \\
    c & d \\
  \end{pmatrix},
  \quad
  g^{-1} =
  \begin{pmatrix}
    a' & b' \\
    c' & d' \\
  \end{pmatrix},
\end{equation*}
where $a$ and $a'$ are $n \times n$ block matrices, and set
\begin{equation*}
  d_H(g) := \min\left(1,
    \max
    \left\{
      \lvert b/d \rvert_F,
      \lvert b'/d' \rvert_F,
      \lvert c/d \rvert_F,
      \lvert c'/d' \rvert_F
    \right\}
  \right).
\end{equation*}
This definition differs mildly from that in~\cite[\S4.2]{2020arXiv201202187N}, but agrees up to a constant factor, for reasons explained there.

\begin{lemma}\label{lemma:let-d-geq-0-g-in-k_g-be-such-that-begin-g-in-k_h-k}
  Let ${\ell} \geq 0$ and $g \in K$ be such that
  \begin{equation}\label{eqn:g-in-k_h-k_z-k_gmathfrakpd.-}
    g \in K_H K_Z K(\mathfrak{p}^\ell).
  \end{equation}
  Then $d_H(g) \leq q^{-\ell}$, with equality if $g \notin K_H K_Z K(\mathfrak{p}^{\ell+1})$.
\end{lemma}
\begin{proof}
  This follows readily from the following observation: for $\ell \geq 1$, the membership~\eqref{eqn:g-in-k_h-k_z-k_gmathfrakpd.-} says that $d, d' \in \mathfrak{o}^\times$ and that at least one $b$ or $c$ (and $b'$ or $c'$) lies in $\mathfrak{p}^{\ell}$.
\end{proof}

In other words, for $g \in K = \mathbf{G}(\mathfrak{o})$, the quantity $d_H(g)$ is the infimum of $q^{-\ell}$ taken over all $\ell \geq 0$ for which the image of $g$ in $\mathbf{G}(\mathfrak{o}/\mathfrak{p}^{\ell})$ lies in $\mathbf{H}(\mathfrak{o}/\mathfrak{p}^{\ell})\mathbf{Z}(\mathfrak{o}/\mathfrak{p}^{\ell})$. For example, if the image of $g$ does not lie in $\mathbf{H} (\mathfrak{o}/\mathfrak{p}) $, then $d_H(g) = 1$, while if $g$ lies in $\mathbf{Z}(\mathfrak{o}) \mathbf{H}(\mathfrak{o})$, then $d_H(g) = 0$.  We abbreviate
\begin{equation*} {d_H(g)}^\infty :=
  \begin{cases}
    1 & \text{ if } d_H(g) = 1, \\
    0 &  \text{ if } d_H(g) < 1.
  \end{cases}
\end{equation*}
We note that for a nonzero ideal $\mfq \subseteq \mathfrak{p}$ and $a \in \mathbf{G}(\mathfrak{o}/\mfq) - \mathbf{H}(\mathfrak{o}/\mfq) \mathbf{Z}(\mathfrak{o}/\mfq)$, the quantity $d_H(g)$ is well-defined.

\subsection{Miscellaneous lemmas}
Let $(F,\mathfrak{o},\mathfrak{p},q)$ be a non-archimedean local field, and let $\mfq \subseteq \mathfrak{p}$ be a nonzero $\mathfrak{o}$-ideal.

\begin{lemma}\label{lemma:supp-char-mathfr-odd.-let-a-in-mathbfgm}
  Assume that $q$ is odd.  Let $a \in \mathbf{G}(\mathfrak{o}/\mfq)$ such that
  \begin{enumerate}[(i)]
  \item the image of $a$ in $\mathbf{G}(\mathfrak{o}/\mathfrak{p})$ lies in $\mathbf{Z}(\mathfrak{o}/\mathfrak{p})$, and
  \item $a^2 \in \mathbf{Z}(\mathfrak{o}/\mfq)$.
  \end{enumerate}
  Then $a \in \mathbf{Z}(\mathfrak{o}/\mfq)$.
\end{lemma}
\begin{proof}
  Write $\mfq = \mathfrak{p}^m$.  If the conclusion fails, then we may write $a = \lambda (1 + \varpi^{\ell} x)$, where $\lambda \in {(\mathfrak{o}/\mfq)}^\times$, $1 \leq \ell < m$ and $x \in \mathbf{M}(\mathfrak{o}/\mfq)$ has the property that its image $\bar{x}$ in $\mathbf{M}(\mathfrak{o}/\mathfrak{p})$ does not lie in $\mathbf{Z}(\mathfrak{o}/\mathfrak{p})$.  Then
  \begin{equation*}
    \mathbf{Z}(\mathfrak{o}/\mfq) \ni a^2 / \lambda^2 = 1 + 2 \varpi^{\ell} x + \varpi^{2 \ell} x^2.
  \end{equation*}
  Reducing this identity modulo $\mathfrak{p}^{\ell+1}$ and using that $2$ is a unit, we deduce that $\bar{x}$ lies in $\mathbf{Z}(\mathfrak{o}/\mathfrak{p})$, giving the required contradiction.
\end{proof}

Recall from Definition~\ref{definition:let-r-be-ring.-let-a-in-mathbfgr-y-in-mathbfx_t-ar} the meaning of ``tangential''.
\begin{lemma}\label{lemma:supp-char-mathfr-odd.-let-a-in-mathbfgm-not-tangential}
  Assume that $q$ is odd.  Let $a \in \mathbf{G}(\mathfrak{o}/\mfq) - \mathbf{H}(\mathfrak{o}/\mfq) \mathbf{Z}(\mathfrak{o}/\mfq)$ with $d_H(a) < 1$.  Then for each $y \in \mathbf{X}_{\tau,a}(\mathfrak{o}/\mfq)$, we have that $\mathbf{X}_{\tau,a}$ is not tangential at $y$ over $\mathfrak{o}/\mfq$.
\end{lemma}
\begin{proof}
  Suppose otherwise that $\mathbf{X}_{\tau,a}$ is tangential at $y$ over $\mathfrak{o}/\mfq$. Writing $a y = h b$ with $(h,b) \in \mathbf{H}(\mathfrak{o}/\mfq) \times \mathbf{G}_\tau(\mathfrak{o}/\mfq)$, it follows then by Theorem~\ref{theorem:characterization-of-tangential-points} that $b^2 \in \mathbf{Z}(\mathfrak{o}/\mfq)$.  On the other hand, since $d_H(a) < 1$, the image $\bar{a}$ of $a$ modulo $\mathfrak{p}$ lies in $\mathbf{H}(\mathfrak{o}/\mathfrak{p}) \mathbf{Z}(\mathfrak{o}/\mathfrak{p})$. By the uniqueness of the decomposition $\bar{a} \bar{y} = \bar{h} \bar{b}$, it follows that $\bar{b}$ lies in $\mathbf{Z}(\mathfrak{o}/\mathfrak{p})$.  By Lemma~\ref{lemma:supp-char-mathfr-odd.-let-a-in-mathbfgm}, it follows that $b \in \mathbf{Z}(\mathfrak{o}/\mfq)$.  But then $a \in \mathbf{H}(\mathfrak{o}/\mfq) \mathbf{Z}(\mathfrak{o}/\mfq)$, contrary to assumption.
\end{proof}

\begin{lemma}\label{lemma:cardinality-centralizer}
  Let $\tau \in \mathbf{M}(\mathfrak{o}/\mfq)$ be cyclic.  Write $n$ for the rank of $\mathbf{V}$.  Then the group $\mathbf{G}_{\tau}(\mathfrak{o}/\mfq)$ has cardinality at least ${(1 - q^{-1})}^n {[\mathfrak{o}:\mfq]}^n$.
\end{lemma}
\begin{proof}
  Since $\mathbf{G}_\tau(\mathfrak{o}/\mfq) = {\left( \mathbf{M}_\tau(\mathfrak{o}/\mfq) \twoheadrightarrow \mathbf{M}_\tau(\mathfrak{o}/\mathfrak{p}) \right)}^{-1}(\mathbf{G}_\tau(\mathfrak{o}/\mathfrak{p}))$, we may reduce to the case $\mfq = \mathfrak{p}$.  Since $k := \mathfrak{o}/\mathfrak{p}$ is a field, we may appeal to Lemma~\ref{lemma:centralizer-description} and the structure theorem for modules over the PID to see that the ring $\mathbf{M}_\tau(k)$ is isomorphic to a product of rings $k[X]/(p_i)$, where the $p_i$ are primary polynomials whose degrees sum to $n$.  Let $d_i$ denote the degree of the irreducible polyomial of which $p_i$ is a power.  Taking unit groups, we obtain $\lvert \mathbf{G}_\tau(k) \rvert = q^n \prod_i (1 - q^{-d_i})$, and the claim follows from the fact that $\sum_i d_i \leq n$.
\end{proof}

\subsection{Main result}\label{sec:cj3v60wplb}

\begin{theorem}\label{theorem:volume-bound}
  Let $(F,\mathfrak{o},\mathfrak{p},q)$ be a non-archimedean local field with $q$ odd.  Let $\mfq \subseteq \mathfrak{p}$ be a nonzero $\mathfrak{o}$-ideal.  Denote by $Q := [\mathfrak{o}:\mfq]$ its absolute norm.  Set
  \begin{equation*}
    Q^* := q^{\lceil m/2 \rceil} \text{ if } Q = q^m.
  \end{equation*}
  Let $\tau \in \mathbf{M}_{\stab}(\mathfrak{o}/\mfq)$ and $a \in \mathbf{G}(\mathfrak{o}/\mfq) - \mathbf{H}(\mathfrak{o}/\mfq) \mathbf{Z}(\mathfrak{o}/\mfq)$.  Then
  \begin{equation*}
    \left\lvert \mathbf{X}_{\tau,a}(\mathfrak{o}/\mfq) \right\rvert
    \leq C \left( \frac{1}{1 + Q d_H(a)} + \frac{{d_H(a)}^\infty }{Q^{*}} \right) \lvert \mathbf{H}_{\tau_H}(\mathfrak{o}/\mfq) \rvert,
  \end{equation*}
  where $C \geq 0$ depends at most upon the rank of $\mathbf{V}$.
\end{theorem}
\begin{proof}
  By Lemma~\ref{lemma:cardinality-centralizer}, the cardinality of $\mathbf{H}_{\tau_H}(\mathfrak{o}/\mfq)$ is comparable to $Q^{n}$, where we write
  \begin{equation*}
    \rank(\mathbf{V}) = n+1.
  \end{equation*}
  For this reason, it suffices to establish the modified estimate obtained by replacing $\lvert \mathbf{H}_{\tau_H}(\mathfrak{o}/\mfq) \rvert$ with $Q^n$.  Moreover, in view of the trivial bound $\lvert \mathbf{X}_{\tau,a}(\mathfrak{o}/\mfq) \rvert \leq \lvert \mathbf{H}_{\tau_H}(\mathfrak{o}/\mfq) \rvert$, it suffices to show that
  \begin{equation*}
    \left\lvert \mathbf{X}_{\tau,a}(\mathfrak{o}/\mfq) \right\rvert
    \ll
    \left( \frac{1}{Q d_H(a)} + \frac{{d_H(a)}^\infty }{Q^{*}} \right) Q^n,
  \end{equation*}
  where $\ll$ means ``bounded in magnitude by a scalar depending at most upon the rank of $\mathbf{V}$''.  We will establish this separately when $d_H(a) < 1$ and when $d_H(a) = 1$.

  Recall from Remark~\ref{remark:rewrite-defining-equations-for-X-as-polynomials} that $\mathbf{X}_{\tau,a}$ is defined by a system of $n+1$ polynomial equations in the entries of $y \in \mathbf{H}_{\tau_H}(\mathfrak{o}/\mfq)$, each of degree at most $n$.

  Consider first the case $d_H(a) < 1$.  Write $Q = q^m$ and $d_H(a) = q^{-\ell}$, so that $1 \leq \ell < m$.  The image of $a$ lies in $\mathbf{H}(\mathfrak{o}/\mathfrak{p}^{\ell}) \mathbf{Z}(\mathfrak{o}/\mathfrak{p}^{\ell})$, thus the defining polynomial equations are trivial mod $\mathfrak{p}^{\ell}$. The image of $a$ does not lie in $\mathbf{H}(\mathfrak{o}/\mathfrak{p}^{\ell+1}) \mathbf{Z}(\mathfrak{o}/\mathfrak{p}^{\ell+1})$, so by Lemma~\ref{lemma:supp-char-mathfr-odd.-let-a-in-mathbfgm-not-tangential}, we see that $\mathbf{X}_{\tau,a}$ is not tangential over $\mathfrak{o} / \mathfrak{p}^{\ell+1}$ at any $y \in \mathbf{X}_{\tau,a}(\mathfrak{o}/\mathfrak{p}^{\ell+1})$.  This implies that at least one of the polynomial equations in our system has the property that at least one of the linear Taylor coefficients at $y$ lies in $\mathfrak{p}^{\ell}-\mathfrak{p}^{\ell+1}$.  Let us divide that polynomial congruence by $\varpi^{\ell}$ and view it now as a polynomial congruence taken modulo $\mathfrak{p}^{m-\ell}$.  Each solution of the new congruence corresponds to exactly $q^{\ell n}$ solutions of the old congruence, but at least one linear Taylor coefficient for the new congruence is a unit, so by part~\eqref{enumerate:assume-that-each-y-in-mathc-one-line-tayl-coeff-p-} of Lemma~\ref{lemma:let-d-geq-0.-supp-given-subs-mathc-subs-mathfr-p} (applied with $\mathcal{D} = \mathbf{H}_{\tau_H}(\mathfrak{o}/\mathfrak{p})$), we obtain
  \begin{equation*}
    \lvert \mathbf{X}_{\tau,a}(\mathfrak{o}/\mfq) \rvert \ll
    q^{\ell n}
    q^{(m-\ell) n - (m-\ell)}
    = Q^{n-1} q^{\ell} = \frac{1}{Q d_H(a)} Q^n,
  \end{equation*}
  as required in this case.

  It remains to consider the case $d_H(a) = 1$, where our goal bound is now
  \begin{equation*}
    \lvert \mathbf{X}_{\tau,a}(\mathfrak{o}/\mfq) \rvert \ll
    \frac{1}{Q^*} Q^n.
  \end{equation*}
  We now have that the image of $a$ does not lie in $\mathbf{H}(\mathfrak{o}/\mathfrak{p}) \mathbf{Z}(\mathfrak{o}/\mathfrak{p})$.  We may thus apply Theorem~\ref{theorem:reta-sett-defin-refd-r-be-ring.-let-mathbfgr-y-mat-doubly-tangential} to see that for each $y \in \mathbf{X}_{\tau,a}(\mathfrak{o}/\mathfrak{p})$, at least one of the linear or quadratic Taylor coefficients for one of the defining polynomials for $\mathbf{X}_{\tau,a}$ is a unit.  By part~\eqref{enumerate:assume-that-each-y-in-mathc-one-line-or-quadr-tayl} of Lemma~\ref{lemma:let-d-geq-0.-supp-given-subs-mathc-subs-mathfr-p}, we deduce the required bound.
\end{proof}

\begin{remark}
  For applications to subconvexity involving a \emph{fixed} local field of characteristic zero, the ``near-identity reduction'' (see~\S\ref{sec:cj4t69an73} and~\cite[\S15.6]{2020arXiv201202187N}) reduces our task, as far as volume bounds like in Theorem~\ref{theorem:volume-bound} are concerned, to the case that $a$ is close to the identity element.  In that case, an adequate volume bound follows from~\cite[Theorem 15.2]{2020arXiv201202187N}.  The novelty of Theorem~\ref{theorem:volume-bound} relative to the cited result is its validity for all $a$ in the maximal compact subgroup, uniformly in the local field.  The restriction to characteristic $\neq 2$ in Theorem~\ref{theorem:volume-bound} is thus harmless in applications.  With additional work, this restriction could be eliminated in the case of a non-archimedean local field of characteristic zero and residue characteristic $2$, basically because the proof of Theorem~\ref{theorem:main-transversality-general-ring} requires only a bounded number of divisions by $2$.

  There remains the case of a local field of characteristic $2$, where the conclusion of Theorem~\ref{theorem:volume-bound} fails for the reasons indicated in Remark~\ref{remark:we-have-seen-comp-calc-with-grobn-bases-that-when-}.  That case is not relevant for the subconvexity problem over number fields, but would be relevant for studying moments of $L$-functions in horizontal aspects over a function field of characteristic $2$.  Estimating such moments thus presents an interesting challenge to which our results do not apply.
\end{remark}


\section{Bilinear forms estimates}\label{Sec:bilinear}
We now apply the volume bound to estimate bilinear forms relevant for~\eqref{eq:cj3tv4jpta}.

Let $(F,\mathfrak{o},\mathfrak{p},\varpi,q)$ be a non-archimedean local field of odd residue characteristic.  Let $\mfq \subseteq \mathfrak{p}$ be a nonzero $\mathfrak{o}$-ideal.
In this section, we abbreviate
\begin{equation*}
  G := \mathbf{G}(\mathfrak{o}/\mfq),
  \qquad
  H := \mathbf{H}(\mathfrak{o}/\mfq),
\end{equation*}
etc.  Write $Q := [\mathfrak{o}:\mfq]$, and define $Q^*$ as in the statement of Theorem~\ref{theorem:volume-bound}.  Recall from \S\ref{sec:norms-distance-functions} the ``distance from $H Z$'' function $d_H$.

\begin{theorem}\label{theorem:bilinear-forms-estimate}
  Let $\tau \in M_{\stab}$.  Let $u_1, u_2 : H \rightarrow \mathbb{R}_{\geq 0}$ be nonnegative functions that are right-invariant under $H_{\tau_H}$.  Let $\gamma \in G$.  Then the quantity
  \begin{equation}\label{eqn:i-:=-frac1lv-h-rvert2-sum-_-subst-x-y-in-h-:-}
    I := 
    \frac{1}{\lvert H \rvert^2}
    \sum _{
      \substack{
        x, y \in H :  \\
        x ^{-1} \gamma y \in G_\tau
      }
    }
    u_1(x) u_2(y)
  \end{equation}
  satisfies the trivial bound
  \begin{equation}\label{eqn:volume-bound-trivial}
    I \leq \frac{1}{|H|} \lVert u_1 \rVert_{L^2} \lVert u_2 \rVert_{L^2}
  \end{equation}
  and, for $\gamma \notin H Z$, the refined bound
  \begin{equation}\label{eqn:volume-bound-refined}
    I \leq
    \frac{C}{|H|} \left( \frac{1}{1 + Q d_H(\gamma)} + \frac{{d_H(\gamma)}^\infty }{Q^*} \right) \lVert u_1 \rVert_{L^2} \lVert u_2 \rVert_{L^2}.
  \end{equation}
  Here $C$ depends at most upon $\rank(\mathbf{V})$, and may be taken to be the same constant as in Theorem~\ref{theorem:volume-bound}.  The $L^2$-norms are defined using the invariant probability measures.
\end{theorem}
\begin{proof}
  We closely follow the proof of~\cite[Thm 15.1, Lem 15.3]{2020arXiv201202187N}.  By Cauchy--Schwarz, we have $I \leq \sqrt{I_1 I_2}$, where
  \begin{equation*}
    I_1 :=\frac{1}{\lvert H \rvert^2}
    \sum _{
      \substack{
        x, y \in H :  \\
        x ^{-1} \gamma y \in G_{\tau} 
      }
    }
    {u_1(x)}^2,
  \end{equation*}
  \begin{equation*}
    I_2 :=\frac{1}{\lvert H \rvert^2}
    \sum _{
      \substack{
        x, y \in H :  \\
        x ^{-1} \gamma y \in G_{\tau} 
      }
    }
    {u_2(y)}^2.
  \end{equation*}
  The set $G_{\tau}$, the condition $\gamma \notin H Z$ and the quantity $d_H(\gamma)$ are invariant under inversion, so it will suffice to obtain a suitable estimate for $I_1$, as the same argument then applies to $I_2$. By definition,
  \begin{equation*}
    I_1 = \frac{1}{|H|^2} \sum _{x \in H} {u_1(x)}^2 \nu(x),
    \quad
    \nu(x) := \left\lvert \left\{ y \in H : x ^{-1} \gamma y \in G_{\tau}  \right\} \right\rvert.
  \end{equation*}
  By Lemma~\ref{lemma:stable-implies-trivial-stabilizer}, we have $H \cap G_{\tau} = \{1\}$, and so $\nu(x) \in \{0, 1\}$.  Thus
  \begin{equation*}
    I_1 \leq \frac{1}{|H|} I_1', \quad
    I_1' := \frac{1}{|H|} \sum _{
      \substack{
        x \in H :  \\
        \gamma ^{-1} x \in H G_{\tau}
      }
    }
    {u_1(x)}^2.
  \end{equation*}
  The trivial bound $I_1' \leq \lVert u_1 \rVert_{L^2}^2$, obtained by dropping the summation condition, yields the first required estimate~\eqref{eqn:volume-bound-trivial}.

  For the second required estimate, we appeal to the right $H_{\tau_H}$-invariance of $u_1$ to write
  \begin{equation*}
    I_1' := \frac{1}{|H|} \sum _{
      x \in H
    }
    {u_1(x)}^2
    \mu(x),
    \quad
    \mu(x) :=
    \frac{
      \left\lvert
        \left\{
          u \in H_{\tau_H} : 
          \gamma ^{-1} x u \in H G_{\tau}
        \right\}
      \right\rvert
    }{
      \left\lvert H_{\tau_H}  \right\rvert
    }.    
  \end{equation*}
  If $\gamma \notin H Z$, then we have $\gamma ^{-1} x \notin H Z$ and $d_H(\gamma^{-1} x) = d_H(\gamma)$ for all $x \in H$.  Estimating $\mu(x)$ via Theorem~\ref{theorem:volume-bound} yields the second estimate~\eqref{eqn:volume-bound-refined}.
\end{proof}

\section{Construction of test vectors}\label{sec:part-body-paper}
The primary aim of this section is to make Theorem~\ref{theorem:cj3ngw7u2s} precise by specifying the local conditions required at the distinguished place (Remark~\ref{remark:cj3u9047n5}).  To that end, we define ``stable'' pairs of representations of general linear groups (Definition~\ref{definition:let-pi-sigma-be-repr-pair-gener-line-groups-over-f}); informally, such representations contain ``microlocalized'' test vectors whose localization parameters have no matching eigenvalues.  To apply this definition, we need to check when it holds.  Since representation-theoretic issues are orthogonal to the main novelty of this paper, we are content here to check that our definition is preserved under parabolic induction (Example~\ref{example:cj3m0jsjik}) and to investigate fully the case of principal series (Example~\ref{example:20230516204714}), leaving open the case of general supercuspidals (see Example~\ref{example:cj3m0jzefl} for some special cases).

We note that many of the ideas and results recorded in this section are well-known in the type theory literature, going back to work of Howe from the 1970s (see, e.g.,~\cite{MR0579176, MR0327982} and~\cite[\S5]{MR2016587}).  The cited works would be most relevant to the special cases of our results where $\tau$ and $\tau_H$ are diagonalizable, a condition that corresponds to the ``wall-avoidance'' hypothesis~\eqref{enumerate:20230516210021} mentioned in \S\ref{sec:cj4vjryk3u}.  Since we do not wish to impose such hypotheses, we give short proofs of what we need here, without claiming particular novelty.

\subsection{Overview}\label{sec:20230516194628}
In \S\ref{sec:20230516194628}, we summarize the main definitions and results of \S\ref{sec:part-body-paper}; the remaining subsections are then devoted to the proofs of these results.

Let $F$ be a non-archimedean local field, with ring of integers $\mathfrak{o}$ and maximal ideal $\mathfrak{p}$. Let $\mathfrak{q} \subseteq \mathfrak{p}$ be an $\mathfrak{o}$-ideal.  Let $\psi$ be a unitary character of $\mfq$ that is trivial on $\mfq^2$, but not on $\mathfrak{p}^{-1}\mfq^2$.

For $\tau \in \mathbf{M}(\mathfrak{o}/\mfq)$, we denote by $\chi_\tau$ the character of the group $K(\mfq)/K(\mfq^2)$ given by
\begin{equation}\label{eqn:definition-of-chi-tau-as-psi-of-trace-x-tau}
  \chi _\tau (1 + x) = \psi (\trace (x \tau )).
\end{equation}
Every character of that group arises in this way.

In what follows, \emph{representation} always means ``complex representation''.  We introduce some terminology concerning representations $\pi$ of $\mathbf{G}(\mathfrak{o})$.  We will apply this terminology also to representations of $\mathbf{G}(F)$, through their restrictions.
\begin{definition}\label{definition:we-say-that-pi-emphr-at-depth-mathfr-if-there-cycl-regular-depth-parameter-polynomial}
  Let $\pi$ be a representation of $\mathbf{G}(\mathfrak{o})$.  We say that $\pi$ is \emph{regular at depth $\mfq^2$} if there is a cyclic element $\tau \in \mathbf{M}(\mathfrak{o}/\mfq)$ (see \S\ref{sec:cyclic-matrices}) such that $\pi$ contains a nonzero vector $v$ that transforms under $K (\mfq )$ via the character $\chi_\tau$, i.e.,
  \begin{equation}\label{eqn:g-v-=-chi_taug-v-quad-text-all--g-in-kmathfrakq.-}
    g v = \chi_\tau(g) v \quad \text{ for all } g \in K(\mfq).
  \end{equation}
  In that case, we refer to $\tau$ as a \emph{regular parameter for} $\pi$ \emph{at depth} $\mfq^2$, to its characteristic polynomial $P \in (\mathfrak{o}/\mfq)[X]$ as a \emph{polynomial for} $\pi$ \emph{at depth} $\mfq^2 $, and to the dimension of the space of $v$ satisfying~\eqref{eqn:g-v-=-chi_taug-v-quad-text-all--g-in-kmathfrakq.-} as the \emph{multiplicity} of $\tau$ or of $P$ in $\pi$.  (These notions depend, of course, upon the choice of $\psi$.)
\end{definition}

\begin{example}\label{example:cj2i1fuxey}
  Let $\chi$ be a character of $\GL_1(\mathfrak{o}) = \mathfrak{o}^\times$, regarded also as a one-dimensional representation.  We denote by
  \begin{equation*}
    c(\chi) \subseteq \mathfrak{o}
  \end{equation*}
  the largest ideal $\mathfrak{a}$ of $\mathfrak{o}$ for which $\chi$ has trivial restriction to $\mathfrak{o}^\times \cap (1 + \mathfrak{a})$.  It is clear then that
  \begin{equation*}
    c(\chi) \supseteq \mfq^2 \iff \text{$\chi$ is regular at depth $\mfq^2$}.
  \end{equation*}
\end{example}

\begin{example}\label{example:cj2i1fu0az}
  We verify in Proposition~\ref{proposition:regular-stable-closed-under-parabolic-induction} that being regular at depth $\mfq^2$ is closed under parabolic induction.  That is to say, suppose given a partition $n = m_1 + \dotsb + m_k$, corresponding to a standard parabolic subgroup of $\GL_n(\mathfrak{o})$.  For each $j \in \{1, \dotsc, k\}$, let $\pi_j$ be a representation of $\GL_{m_j}(\mathfrak{o})$ that is regular at depth $\mfq^2$ with polynomial $P_j$.  Then the same holds for the parabolic induction of $\pi_1 \otimes \dotsb \otimes \pi_k$ to $\GL_n(\mathfrak{o})$, with polynomial at depth $\mfq^2$ given by $P=\prod\limits_{i} P_i$ (unique if the $P_i$ are).
\end{example}
\begin{example}\label{Example:principalregular}
  By combining Examples~\ref{example:cj2i1fuxey} and~\ref{example:cj2i1fu0az}, we deduce that a principal series representation induced by characters $\chi_i$ with $c(\chi_i) \supseteq \mfq^2$ is regular at depth $\mfq^2$.  Its unique polynomial at depth $\mfq^2$ is the product of linear factors $\prod_{i} (X - \xi_i)$, where $\xi_i$ is characterized by the identity $\chi_i(1+y) = \psi(y \xi_i)$ for all $y \in \mfq$.
\end{example}
\begin{example}\label{Example:supercuspidaloverview}
  We verify in \S\ref{sec:supercuspidals} that certain supercuspidal representations are regular at depth $\mfq^2$.  We do not determine precisely which supercuspidal representations are regular at depth $\mfq^2$, but it may be possible to do so using known classifications of the latter (see Remark~\ref{Rem:supercuspidalclassification}).
\end{example}

\begin{definition}\label{definition:let-pi-sigma-be-repr-pair-gener-line-groups-over-f}
  Let $m$ and $n$ be natural numbers.  Let $\pi$ and $\sigma$ be representations of $\GL_{n}(\mathfrak{o})$ and $\GL_m(\mathfrak{o})$.  We say that $(\pi,\sigma)$ is \emph{stable at depth} $\mfq^2$ if
  \begin{itemize}
  \item both $\pi$ and $\sigma$ are regular at depth $\mfq^2$, and
  \item there are polynomials $P_\pi$ (resp.\ $P_\sigma$) for $\pi$ (resp.\ $\sigma$) at depth $\mfq^2$ that generate the unit ideal of $(\mathfrak{o}/\mfq)[X]$.
  \end{itemize}
\end{definition}

\begin{example}\label{example:cj3m0jsjik}
  Being stable at depth $\mfq^2$ is closed under parabolic induction (Proposition~\ref{proposition:stable-pairs-closed-under-parabolic-induction}): if $\pi$ (resp.\ $\sigma$) is induced by some collection of representations $\pi_i$ (resp.\ $\sigma_j$) of $\GL_{n_i}(\mathfrak{o})$ (resp.\ $\GL_{m_j}(\mathfrak{o})$) such that each pair $(\pi_i,\sigma_j)$ is stable at depth $\mfq^2$, then the same holds for $(\pi,\sigma)$.
\end{example}
\begin{example}\label{example:20230516204714}
  If $\pi$ (resp.\ $\sigma$) is a principal series representation induced by characters $\chi_i$ (resp.\ $\eta_j$), then the pair $(\pi,\sigma)$ is stable at depth $\mfq^2$ provided that for all relevant indices,
  \begin{equation}\label{eqn:20230516200125}
    c(\chi_i) \supseteq \mfq^2, \qquad  c(\eta_j) \supseteq \mfq^2,
    \qquad
    c (\chi_i / \eta_j) = \mfq^2.
  \end{equation}
  Assuming the first two conditions in~\eqref{eqn:20230516200125}, the third condition is equivalent to
  \begin{equation}\label{eqn:20230516200301}
    \prod_{i,j} [\mathfrak{o}:c(\chi_i/\eta_j)] = {[\mathfrak{o}:\mfq^2]}^{n(n+1)}.
  \end{equation}
  The left hand side may be understood as the conductor of the Rankin--Selberg convolution $\pi \times \sigma^\vee$, and the equality~\eqref{eqn:20230516200301} may be understood as an analogue of the ``no conductor dropping'' or ``uniform growth'' conditions considered in~\cite{2020arXiv201202187N} and~\cite{2021arXiv210915230N}.
\end{example}

\begin{example}\label{example:cj3m0jzefl}
  When $(\pi,\sigma)$ is a pair of regular supercuspidal representations of $(\GL_{m},\GL_{n})$ as discussed in Example~\ref{Example:supercuspidaloverview} or \S\ref{sec:supercuspidals}, it is automatically stable when $m\neq n$; when $m=n$, one needs to impose additional conditions similar to the case of principal series representations (see Lemma~\ref{Lem:supercuspidalstablepair}).
\end{example}

\subsection{Characters}\label{sec:characters-in-construction-test-vectors}
Let $\mfq \subseteq \mathfrak{p}$ be an $\mathfrak{o}$-ideal.  Let $\psi : \mfq / \mfq^2 \rightarrow \mathbb{C}^\times$ be a character that is nontrivial on $\mathfrak{p} ^{-1} \mfq^2 / \mfq^2$.  The character $\psi$ induces an identification
\begin{equation*}
  \mathfrak{o} / \mfq \xrightarrow{\cong } \left\{ \text{characters } \mfq / \mfq^2 \rightarrow \mathbb{C} ^\times  \right\},
\end{equation*}
\begin{equation*}
  \xi \mapsto \psi(\xi \bullet).
\end{equation*}
For a character $\chi$ of $1 + \mfq$ that is trivial on $1 + \mfq^2$, we denote by $\xi_\chi \in \mathfrak{o}/\mfq$ the unique element with
\begin{equation*}
  \chi(1 + x) = \psi(\xi_\chi x) \quad \text{ for } x \in \mfq.
\end{equation*}

\subsection{Centralizers of characters}\label{sec:centralizers-characters-J-tau}
Given $\tau \in \mathbf{M}(\mathfrak{o}/\mfq)$, we denote by
\begin{equation*}
  J_\tau \leq \mathbf{G}(\mathfrak{o})
\end{equation*}
the inverse image of the mod-$\mfq$ centralizer $\mathbf{G}_\tau(\mathfrak{o}/\mfq)$.  It is a subgroup that contains $K(\mfq)$.

\subsection{Preliminary lemmas regarding regular parameters}\label{sec:regular-uniform-representations}

\begin{lemma}\label{lemma:let-p-be-polyn-pi-at-depth-mathfr-let-tau-in-mathb}
  Let $P$ be a polynomial for $\pi$ at depth $\mfq^2$.  Then any cyclic element $\tau \in \mathbf{M}(\mathfrak{o}/\mfq)$ whose characteristic polynomial is $P$ is a regular parameter for $\pi$ at depth $\mfq^2$.
\end{lemma}
\begin{proof}
  We observe first that the set of regular parameters for $\pi$ at depth $\mfq^2$ is invariant under conjugation by $\mathbf{G}(\mathfrak{o}/\mfq)$.  The conclusion of the lemma then follows from the fact (Lemma~\ref{lemma:regular-elements-same-characteristic-polynomial-are-conjugate}) that any two cyclic elements having the same characteristic polynomial are conjugate.
\end{proof}

\begin{lemma}\label{lemma:cj3m0fjios}
  Let $\tau \in \mathbf{M}(\mathfrak{o}/\mathfrak{q})$ be cyclic.  Then $\ker(\chi_\tau)$ is a normal subgroup of $J_\tau$, and the quotient group $J_\tau / \ker(\chi_\tau)$ is abelian.
\end{lemma}
\begin{proof}
  We verify first that the kernel is normal.  Let $x \in J_\tau$ and $y \in \ker(\chi_\tau)$.  Then $x y x^{-1}$ lies in $K(\mathfrak{q})$ (because $y$ lies in $K(\mathfrak{q})$, which is normal in $\mathbf{G}(\mathfrak{o})$) and also in $\ker(\chi_\tau)$ (because $x$ commutes with $\tau$ modulo $\mathfrak{q}$).

  We verify next that the quotient is abelian.  Choose an arbitrary lift $\tilde{\tau} \in \mathbf{G}(\mathfrak{o})$ of $\tau$.  Then, since $\mathbf{M}_\tau(\mathfrak{o}/\mathfrak{q})$ consists of polynomials in $\tau$ (Lemma~\ref{lemma:centralizer-description}), we see that any element of $J_\tau$ may be written as a product $a b$, where $a$ is a polynomial in $\tilde{\tau}$ and $b$ lies in $K(\mathfrak{q})$.  The images of such elements generate the quotient group $J_\tau / \ker(\chi_\tau)$, so to verify that the latter is abelian, we reduce to checking that any commutator $(a_1,a_2)$ or $(a_1,b_1)$ or $(b_1,b_2)$ lies in $\ker(\chi_\tau)$, where the $a_i$ and $b_j$ are as before.  Indeed, we have $(a_1,a_2) = 1$ (because any two polynomials in $\tilde{\tau}$ commute with one another) and $\chi_\tau((a_1,b_1)) = 1$ (because $a_1$ commutes with $\tau$ modulo $\mathfrak{q}$) and $\chi_\tau((b_1,b_2)) = 1$ (because $\chi_\tau$ is a character of $K(\mathfrak{q})$).
\end{proof}

\begin{lemma}\label{lemma:let-tau-in-mathbfm-mathfr--mathfr-be-regul-param-p-extension-chi-tau-to-J-tau}
  Let $\tau \in \mathbf{M} (\mathfrak{o} / \mfq)$ be a regular parameter for $\pi$ at depth $\mfq^2$.  Then there is an extension of $\chi_\tau$ to a character $\tilde{\chi}_\tau$ of $J_\tau$ and a nonzero vector $v \in \pi$ that transforms under $J_\tau $ via $\tilde{\chi }_\tau $.  If $\tau$ occurs with multiplicity one, then $v$ is unique up to scalar multiple.
\end{lemma}
\begin{proof}
  Let $v_0$ be a nonzero vector in $\pi$ that transforms under $K(\mfq)$ via $\chi_\tau$.  Let $V$ denote the span of $v_0$ under $J_\tau$.  This is a vector space of dimension at most $[J_\tau:K(\mfq)]$, on which $K(\mfq^2)$ acts trivially.  By Lemma~\ref{lemma:cj3m0fjios}, the action of $J_\tau$ on $V$ factors through that of the quotient group $J_\tau / \ker(\chi_\tau)$.  The latter is a finite abelian group, so its action on $V$ decomposes as a direct sum of one-dimensional subspaces.  Take $v$ to lie in one of these subspaces, and $\tilde{\chi}_\tau$ to be the character describing the eigenvalues for $J_\tau$ acting on $v$.  The final assertion is clear.
\end{proof}

\begin{example}\label{example:let-tau-in-mathbfmm-be-cycl-it-defin-char-chi_t-o}
  Let $\tau \in \mathbf{M}(\mathfrak{o}/\mfq)$ be cyclic.  It defines a character $\chi_\tau$ of the subgroup $K(\mfq)$ of $J_\tau$.  Suppose given an extension $\tilde{\chi}_\tau$ of $\chi_\tau$ to a character of $J_\tau$.  Let $\pi$ denote the representation of $\mathbf{G}(\mathfrak{o})$ induced from $\tilde{\chi}_\tau$.  Then $\pi$ is regular at depth $\mfq^2$, with regular parameter $\tau$.  These elementary observations imply that certain supercuspidal representations satisfy the above definitions (specifically, those of ``odd depth'' and ``generic inducing data'' --- see \S\ref{sec:supercuspidals}).
\end{example}

\subsection{Some linear algebra}\label{sec:cj3m0c59nb}
We record some general lemmas to be applied in the following subsection.  Here we focus on an individual ring $R$ and abbreviate $V := \mathbf{V}(R)$, $M := \mathbf{M}(R)$, etc.

\begin{lemma}\label{lemma:v-tau-cycl-if-only-if-v-tau-cycl-more-prec-if-vect}
  Let $\tau \in M$.  Then $V$ is $\tau$-cyclic if and only if $V^*$ is $\tau$-cyclic.  More precisely, if the vector $v \in V$ is $\tau$-cyclic, then functional $v^* \in V^*$ given by $\sum c _j \tau ^j v \mapsto c_{n-1}$ is $\tau$-cyclic.
\end{lemma}
\begin{proof}
  This is the observation that the transpose of a matrix like~\eqref{eqn:beginpmatrix-0--0--ast-} is cyclic; we leave it to the reader.
\end{proof}

In what follows, by a \emph{flag} $V_0 \subset \dotsb \subset V_k$ in $V$, we mean a sequence of free submodules, with $V_0 = \{0\}$ and $V_k = V$, such that each quotient $V_j / V_{j-1}$ is also free.  We say that $\tau \in M$ preserves this flag if $\tau V_j \subseteq V_j$ for each $j \in \{1, \dotsc, k\}$.

\begin{lemma}\label{lemma:let-v_0-be-free-tau-invar-subm-v.-assume-tau-regul}
  Let $\tau \in M$ be cyclic.  Let $V_0 \subset \dotsb \subset V_k$ be a flag.  Suppose that $\tau$ preserves the flag, so that $\tau$ induces endomorphisms $\tau_j \in \End(V_j/ V_{j-1})$.  Then each $\tau_j$ is cyclic.
\end{lemma}
\begin{proof}
  By induction, we may reduce to the case $k = 2$.  It is clear that $\tau_2$ is cyclic, because the image in $V/V_1$ of any cyclic vector for $\tau$ is a cyclic vector for $\tau_2$.  It remains to show that $\tau_1$ is cyclic.  By Lemma~\ref{lemma:v-tau-cycl-if-only-if-v-tau-cycl-more-prec-if-vect}, we know that $\tau^* \in \End(V^*)$ is cyclic, and it suffices to show that $\tau_1^* \in \End(V_1^*)$ is cyclic.  By identifying $V_1^*$ with the quotient $V^* / V_1^\perp$, we reduce to the previous observation.
\end{proof}

\begin{lemma}\label{lemma:if-tau-cycl-then-char-polyn-p_tau-gener-annih-idea}
  If $\tau$ is cyclic, then its characteristic polynomial $P_\tau$ generates the annihilator ideal $\{f \in R[X] : f(\tau) = 0\}$.
\end{lemma}
\begin{proof}
  Otherwise, let $f$ be polynomial of minimal degree that is not divisible by $P_\tau$ and yet for which $f(\tau) = 0$.  If the degree of $f$ is less than the rank of $\mathbf{V}$, then we obtain a contradiction from Lemma~\ref{lemma:v-tau-cycl-if-only-if-map-modul-begin-rn-right-v-}.  Otherwise, we can apply division with remainder using the monic polynomial $P_\tau$ to contradict the minimality of $f$.
\end{proof}

\begin{lemma}\label{lemma:let-v_1-subs-dotsb-subs-v_k-be-flag-let-tau_j-in-e}
  Let $V_1 \subset \dotsb \subset V_k$ be a flag, and let $\tau_j \in \End(V_j / V_{j-1})$ be cyclic elements.  Then there is a cyclic element $\tau \in \End(V)$ that preserves the flag and for which the induced action on $V_j / V_{j-1}$ is $\tau_j$.
\end{lemma}
\begin{proof}
  By induction, we may reduce to the case $k = 2$ and $V_k = V$.
  
  We may find a basis $e_1,\dotsc,e_m$ for $V_1$ with respect to which $\tau_1$ is cyclic and a basis $\bar{f}_1,\dotsc,\bar{f}_n$ for $V/V_1$ with respect to which $\tau_2$ is cyclic, i.e.,
  \begin{equation*}
    \tau_1 e_i = e_{i+1}  \quad (i  < m), \qquad
    \tau_2 \bar{f}_j = \bar{f}_{j+1} \quad (j < n).
  \end{equation*}

  By choosing a splitting $V = V_1 \oplus V_2$, we identify $\tau_2$ with an endomorphism of $V_2$, extended by zero to $V$.  We then define $\tau \in \End(V)$ by taking
  \begin{equation*}
    \tau e_i = \tau_1 e_{i},
    \qquad
    \tau f_j =
    \tau_2 f_j
    +
    \begin{cases}
      e_1 & \text{ if } j = n, \\
      0 & \text{ otherwise.}
    \end{cases}
  \end{equation*}
  We claim then that $f_1$ is a cyclic vector for $\tau$.  Indeed, the set $\{f_1, \tau f_1, \dotsc, \tau^{n-1} f_1\} = \{f_1, f_2, \dotsc, f_{n}\}$ spans $V_2$, while the set $\{\tau^n f_1, \dotsc, \tau^{n+m-1} f_1\}$ is congruent modulo $V_2$ to the spanning set $\{e_1, \dotsc, e_m\}$ for $V_1$.
\end{proof}

\begin{lemma}\label{lemma:let-r-be-ring.-write-v-:=-mathbfvr-m-:=-mathbfmr-e}
  Let $V_1 \subset \dotsb \subset V_k$ be a flag.  Let $\tau \in M$ be a cyclic element that preserves this flag.  Suppose that $g \in G$ has the following properties:
  \begin{enumerate}[(i)]
  \item $\Ad(g) \tau$ preserves the flag.
  \item For each $j$, the induced actions of $\tau$ and $\Ad(g) \tau$ on $V_j / V_{j-1}$ have the same characteristic polynomial.
  \end{enumerate}
  Then $g$ preserves the flag.
\end{lemma}
\begin{proof}
  By induction, we may reduce to the case $k = 2$.  Let $v \in V$ be $\tau$-cyclic.  Then every element of $V$ may be written $f(\tau) v$ for some polynomial $f \in R[X]$, so it suffices to show that if $f(\tau) v \in V_1$, then $g f(\tau) v \in V_1$.  Suppose, thus, that $f(\tau) v \in V_1$.  Let $v_2 \in V/ V_1$ denote the image of $v$, and write $\tau_2 \in \End(V/V_1)$ for the induced action of $\tau$ on the quotient.  Then $v_2$ is $\tau_2$-cyclic and $f(\tau_2) v_2 = 0$, hence $f(\tau_2) = 0$.  It follows that $f$ is divisible by the characteristic polynomial of $\tau_2$ (Lemma~\ref{lemma:if-tau-cycl-then-char-polyn-p_tau-gener-annih-idea}).  Since the induced action of $\Ad(g) \tau$ on $V / V_1$ has the same characteristic polynomial as $\tau_2$, it follows from the Cayley--Hamilton theorem that $f(\Ad(g) \tau) V \subseteq V_1$.  But then $g f(\tau) v = f(\Ad(g) \tau) g v \in V_1$, as required.
\end{proof}

\subsection{Parabolic induction}\label{sec:cj3m0c6a7b}
The following proposition shows that the existence and uniqueness of localized vectors is preserved by induction.  (The ``uniqueness'' assertions will not be used here, but might be useful for certain extensions of the work of this paper.)
\begin{proposition}\label{proposition:regular-stable-closed-under-parabolic-induction}
  The class of representations of general linear groups over $\mathfrak{o}$ that are regular at depth $\mfq^2$ is closed under parabolic induction.  That is to say, writing $n = \rank(\mathbf{V})$, suppose given a partition $n = m_1 + \dotsb + m_k$, corresponding to a parabolic subgroup $\mathbf{P}$ of $\mathbf{G}$, together with a collection of representations $\pi_j$ ($j = 1, \dotsc, k$) of $\GL_{m_j}(\mathfrak{o})$ each of which is regular at depth $\mfq^2$.  Then the same holds for the parabolic induction
  \begin{equation*}
    \pi := \Ind_{\mathbf{P}(\mathfrak{o})}^{\mathbf{G}(\mathfrak{o})}(\pi_1 \otimes \dotsb \otimes \pi_k).
  \end{equation*}
  Moreover:
  \begin{enumerate}[(i)]
  \item Every polynomial for $\pi$ at depth $\mfq^2$ is a product of polynomials for $\pi_1,\dotsc,\pi_k$ at depth $\mfq^2$.
  \item For any polynomials $P_1,\dotsc,P_k$ for $\pi_1,\dotsc,\pi_k$ at depth $\mfq^2$, the product $P_1 \dotsm P_k$ is a polynomial for $\pi$ at depth $\mfq^2$.
  \item If $\pi_i$ admits a unique polynomial $P_i$ at depth $\mfq^2$, and if $P_i$ occurs with multiplicity one, then $P_1 \dotsb P_k$ is likewise the unique polynomial for $\pi$ at depth $\mfq^2$, and occurs with multiplicity one.
  \end{enumerate}
\end{proposition}

\begin{proof}
  For $\tau \in \mathbf{M}(\mathfrak{o}/\mfq)$, let us temporarily write $\pi^{\tau}$ for the subspace of $\pi$ on which $K(\mfq)$ acts via $\chi_\tau$.  Similarly, for an $m_i \times m_i$ matrix $\tau_i$ over $\mathfrak{o}/\mfq$, write $\pi_i^{\tau_i}$ for the associated subspace of $\pi_i$.  Let us take for $\mathbf{P}$ the standard upper-triangular parabolic subgroup associated to the given partition.  Let $\mathbf{U} < \mathbf{P}$ denote the unipotent radical and $\mathbf{L}$ the block-diagonal Levi factor, thus $\mathbf{L} \cong \prod_{i=1}^k \mathbf{G}_i$, where each $\mathbf{G}_i$ is a general linear group of rank $m_i$.  We may regard
  \begin{equation*}
    \Pi := \otimes_i \pi_i
  \end{equation*}
  as a representation of $\mathbf{L}(\mathfrak{o})$.  We extend it to a representation of $\mathbf{P}(\mathfrak{o})$ by letting $\mathbf{U}(\mathfrak{o})$ act trivially, and realize $\pi$ as the space of functions $v : \mathbf{G}(\mathfrak{o}) \rightarrow \Pi$ such that for all $p \in \mathbf{P}(\mathfrak{o})$,
  \begin{equation*}
    v(p g) = \Pi(p) v(g),
  \end{equation*}
  with the action given by right translation.

  By Mackey theory, a basis for $\pi^\tau$ is indexed by representatives $g$ for the double quotient
  \begin{equation}\label{eqn:mathbfp-mathfr-backsl-mathbfgm--kmathfr-cong-mathb}
    \mathbf{P} (\mathfrak{o}) \backslash \mathbf{G}(\mathfrak{o}) / K(\mfq)
    \cong \mathbf{P}(\mathfrak{o}/\mfq) \backslash \mathbf{G}(\mathfrak{o}/\mfq)    
  \end{equation}
  together with a basis for the space of vectors $t \in \Pi$ such that the formula
  \begin{equation*}
    v_g (p g h) =
    \begin{cases}
      \chi_\tau(h) \Pi(p) t  & \text{ if } (p,h) \in \mathbf{P}(\mathfrak{o}) \times K(\mfq), \\
      0 & \text{ otherwise}
    \end{cases}
  \end{equation*}
  is well-defined, that is, for which $p g h = g \implies \chi_\tau(h) \Pi(p) t = t$, or equivalently,
  \begin{equation}\label{eqn:p-in-mathbfpm-times-g-kmathfr-g-1-impl-chi_-tau-p-}
    p \in \mathbf{P}(\mathfrak{o}) \cap  g K(\mfq) g^{-1}
    \implies
    \chi_{\Ad(g) \tau }(p) t = \chi_\tau(g^{-1} p g) t = \Pi(p) t.
  \end{equation}
  Since $K(\mfq)$ is normal in $\mathbf{G}(\mathfrak{o})$, we have
  \begin{equation*}
    \mathbf{P}(\mathfrak{o}) \cap g K(\mfq) g^{-1} = K_P(\mfq) := \ker (
    \mathbf{P}(\mathfrak{o}) \rightarrow \mathbf{P}(\mathfrak{o}/\mfq)
    ).
  \end{equation*}
  Any element $p$ of the latter may be written uniquely as $p = m u$, where $(m, u)$ lies in $K_L(\mfq) \times K_U(\mfq)$ (with each factor defined like $K_P(\mathfrak{q})$).  Assuming that the vector $t$ is nonzero, the condition~\eqref{eqn:p-in-mathbfpm-times-g-kmathfr-g-1-impl-chi_-tau-p-} is then equivalent to the following pair of assertions.
  \begin{itemize}
  \item $\chi_{\Ad(g) \tau}(u) = 1$ for all $u \in K_U(\mfq)$.  Translating this assertion via the trace pairing, it says that the element $\Ad(g) \tau$ of $\mathbf{M}(\mathfrak{o}/\mfq)$ preserves the flag $\mathcal{F}$ in $\mathbf{M}(\mathfrak{o}/\mfq)$ of which $\mathbf{P}(\mathfrak{o}/\mfq)$ is the stabilizer.
  \item $\Pi(m) t = \chi_{\Ad(g) \tau}(m) t$ for all $m \in K_L(\mfq)$.  Writing ${(\Ad(g) \tau )}_{ii}$ for the $i$th block diagonal $m_i \times m_i$ matrix, this is equivalent to the condition $t \in \otimes_{i=1}^k \pi_i^{{(\Ad(g) \tau)}_{i i}}$.
  \end{itemize}
  In summary,
  \begin{equation}\label{eqn:pitau--=-oplus_-subst-g-in-mathbfpm-backsl-mathbfg}
    \pi^{\tau } = \oplus_{
      \substack{
        g \in \mathbf{P}(\mathfrak{o}) \backslash \mathbf{G}(\mathfrak{o}) / K(\mfq) :  \\
        \Ad(g) \tau \text{ preserves } \mathcal{F} 
      }
    }
    \otimes_{i=1}^k \pi_i^{{(\Ad(g) \tau)}_{i i}}.
  \end{equation}

  The matrices ${(\Ad(g) \tau)}_{i i}$ are cyclic, as their images in $\mathbf{M}_{m_i}(\mathfrak{o}/\mfq)$ are cyclic by Lemma~\ref{lemma:let-v_0-be-free-tau-invar-subm-v.-assume-tau-regul}.  If the spaces $\pi_i^{{(\Ad(g) \tau )}_{i i}}$ are nonzero, then the characteristic polynomial of ${(\Ad(g) \tau )}_{i i}$ must be a polynomial for $\pi_j$ at depth $\mfq^2$.  Thus every polynomial for $\pi$ is a product of polynomials for the $\pi_j$ inside the ring $(\mathfrak{o}/\mfq)[X]$.

  Conversely, suppose given polynomials $P_1, \dotsc, P_k$ for $\pi_1,\dotsc,\pi_k$ at depth $\mfq^2$.  By Lemma~\ref{lemma:let-v_1-subs-dotsb-subs-v_k-be-flag-let-tau_j-in-e}, there is a cyclic element $\tau$ which preserves the flag and for which each $\tau_{ii}$ has characteristic polynomial $P_i$.  Then the characteristic polynomial of $\tau$ is $P_1 \dotsb P_k$.  On the other hand, the contribution to~\eqref{eqn:pitau--=-oplus_-subst-g-in-mathbfpm-backsl-mathbfg} from $g = 1$ is positive-dimensional, so $\pi^\tau$ is nonzero.  Thus $P_1 \dotsb P_k$ is a polynomial for $\pi$.

  It remains to verify the final assertion concerning uniqueness and multiplicity one.  Suppose that $\pi^\tau$ is nonzero for some cyclic $\tau \in M$.  By~\eqref{eqn:pitau--=-oplus_-subst-g-in-mathbfpm-backsl-mathbfg}, we may assume (after conjugating $\tau$ if necessary) that $\tau$ preserves $\mathcal{F}$ and that each $\pi_i^{\tau_{ii}}$ is nonzero.  By hypothesis, the characteristic polynomial of $\tau_{ii}$ is $P_i$ and $\pi_i^{\tau_{ii}}$ is one-dimensional.  Moreover, by Lemma~\ref{lemma:let-r-be-ring.-write-v-:=-mathbfvr-m-:=-mathbfmr-e} (applied over $R = \mathfrak{o} / \mfq$), the only $g \in \mathbf{G}(\mathfrak{o})$ for which $\Ad(g) \tau$ preserves $\mathcal{F}$ and ${(\Ad(g) \tau)}_{ii}$ has characteristic polynomial $P_i$ lie in the trivial double coset. Thus~\eqref{eqn:pitau--=-oplus_-subst-g-in-mathbfpm-backsl-mathbfg} reduces to the one-dimensional summand $\otimes_{i=1}^k \pi_i^{\tau_{ii}}$.  This proves that $P_1 \dotsb P_k$ is the unique polynomial for $\pi$ and occurs with multiplicity one.
\end{proof}

\subsection{Stable pairs of representations}\label{sec:cj3m0c6cif}
Recall that in Definition~\ref{definition:let-pi-sigma-be-repr-pair-gener-line-groups-over-f}, we defined when a pair of representations $(\pi,\sigma)$ is stable at depth $\mfq^2$.

\begin{example}\label{example:stable-pairs-of-characters}
  Let $\chi$ and $\eta$ be characters of $\mathfrak{o}^\times = \GL_1(\mathfrak{o})$. Then the pair $(\chi,\eta)$ is stable at depth $\mfq^2$ precisely when $\chi$ and $\eta$ have conductor dividing $\mfq^2$ and the ratio $\chi/\eta$ has conductor exactly $\mfq^2$.  In verifying this, we observe that $\chi(x) / \eta (x) = \psi (x (\xi_\chi - \xi_\eta ))$, which is nontrivial for some $x \in \mathfrak{p}^{-1} \mfq^2$ precisely when $\xi_\chi - \xi_\eta$ is a unit, or equivalently, when the polynomials $X - \xi_\chi$ and $X - \xi_\eta$ generate the unit ideal in $(\mathfrak{o} / \mfq )[X]$.
\end{example}

\begin{proposition}\label{proposition:stable-pairs-closed-under-parabolic-induction}
  The class of pairs of representations $(\pi,\sigma)$ that are stable at depth $\mfq^2$ is closed under parabolic induction, in the following sense.  Suppose given a finite collection of representations $\pi_i$ and $\sigma_j$, where each pair $(\pi_i,\sigma_j)$ is stable at depth $\mfq^2$.  Let $\pi$ (resp.\ $\sigma$) denote the parabolic induction of the $\pi_i$ (resp.\ $\sigma_j$), as in Proposition~\ref{proposition:regular-stable-closed-under-parabolic-induction}.  Then the pair $(\pi,\sigma)$ is likewise stable at depth $\mfq^2$.
\end{proposition}
\begin{proof}
  This reduces to Proposition~\ref{proposition:regular-stable-closed-under-parabolic-induction} and the following observation: given finite collections of polynomials $\{P_i\}$ and $\{Q_j\}$ in $(\mathfrak{o} / \mfq )[X]$ such that each pair $(P_i,Q_j)$ generates the unit ideal, the same holds for the pair $(P,Q)$ given by $P := \prod_i P_i$ and $Q := \prod_j Q_j$.  (To see this, write $1 = a_{i,j} P_i + b_{i,j} Q_j$ and observe that $\prod_{i,j} (a_{i,j} P_i + b_{i,j} Q_j)$ lies in the ideal generated by $P$ and $Q$.)
\end{proof}

\begin{example}
  Let $\pi$ (resp.\ $\sigma$) be a principal series representation induced by characters $\chi_i$ (resp.\ $\eta_j$), each of conductor dividing $\mfq^2$.  Suppose that each ratio $\chi_i / \eta_j$ has conductor $\mfq^2$.  Then $(\pi,\sigma)$ is stable at depth $\mfq^2$.  In verifying this, we may reduce via Proposition~\ref{proposition:stable-pairs-closed-under-parabolic-induction} to the case of a pair of characters $(\chi,\eta)$ discussed in Example~\ref{example:stable-pairs-of-characters}.
\end{example}

\subsection{Supercuspidals}\label{sec:supercuspidals}
We verify that the above terminology applies to certain supercuspidal representations of $\mathbf{G}(F)$.  We do not aim here for an exhaustive treatment, as representation-theoretic issues are not the focus of this paper.

Write $n := \dim(V)$.  Let $\tau \in \mathbf{M}(\mathfrak{o})$ be such that the $F$-subalgebra
\begin{equation*}
  E := F[\tau] \subseteq \mathbf{M}(F)
\end{equation*}
is a field extension of $F$ having the largest possibly degree, namely $n$.  Let
\begin{equation*}
  \eta : E^\times \rightarrow \mathbb{C}^\times
\end{equation*}
be a character.  We consider the following special case:
\begin{Assumption}\label{Assumption:supercuspidal}
  The element $\tau$ satisfies the following conditions:
  \begin{enumerate}
  \item\label{enumerate:d1bc2d5fbd54} The image $\overline{\tau}$ of $\tau$ in $\mathbf{M}(\mathfrak{o}/\mathfrak{p})$ also generates a degree $n$ extension of the residue field.  In particular, $\det(\tau)\in \mathfrak{o}^\times$.
  \item For each $e \in E^\times \cap K (\mfq)$, we have the compatibility condition
    \begin{equation*}
      \eta(e) = \psi(\trace(\tau (e-1))).
    \end{equation*}
  \end{enumerate}
\end{Assumption}

Recall that $\psi : \mfq / \mfq^2 \rightarrow \mathbb{C}^\times$ is a character nontrivial on $\mathfrak{p} ^{-1} \mfq^2 / \mfq^2$. As $\tau$ normalizes $K(\mfq)$, so does $E^\times$. The subgroup
\begin{equation*}
  J_E := E^\times K(\mfq)
\end{equation*}
is well-defined. It admits the character $\chi : J_E \rightarrow \mathbb{C}^\times$ defined by the formula
\begin{equation}\label{Eq:chi}
  \chi(e g) := \eta(e) \psi(\trace(\tau (g-1))) \quad \text{ for } (e,g) \in E^\times \times K(\mfq).
\end{equation}
\begin{remark}\label{Remark:assumption1}
  The assumption~\eqref{enumerate:d1bc2d5fbd54} implies that $E$ is an unramified extension of $F$. It also implies that any nonzero element in $\mathbf{V}(\mathfrak{o}/\mathfrak{p})$ is cyclic with respect to the image $\overline{\tau}$ of $\tau$ in $\mathbf{M}(\mathfrak{o}/\mathfrak{p})$, as $\mathbf{V}(\mathfrak{o}/\mathfrak{p})$ is then a one-dimensional vector space over the field $(\mathfrak{o}/\mathfrak{p})[\overline{\tau}] \subseteq \mathbf{M}(\mathfrak{o}/\mathfrak{p})$.
\end{remark}
\begin{lemma}\label{lemma:d1bfac67927b}
  The compact induction
  \begin{equation*}
    \pi :=\ind_{J_E}^{\mathbf{G}(F)} \chi
  \end{equation*}
  is an irreducible supercuspidal representation of $\mathbf{G}(F)$.
\end{lemma}
\begin{proof}
  For the notations as above, the pair $(J_E, \chi)$ is a maximal type in the sense of~\cite[(6.2)]{BK}.  Thus by~\cite[(6.2.2)]{BK}, the representation $\pi$ constructed above is a supercuspidal representation.
\end{proof}
\begin{lemma}\label{Lem:supercuspidalregular}
  The representation $\pi$ constructed in Lemma~\ref{lemma:d1bfac67927b} is regular at depth $\mfq^2$, with regular parameter given by $\tau$ modulo $\mathfrak{q}$ (see \S\ref{sec:20230516194628} for definitions).
\end{lemma}
\begin{proof}
  By definition of compact induction, there exists an element $v\in \pi$ which identifies with $\chi$ on $J_E$ and vanishes elsewhere. In particular $J_E$ acts on $v$ by $\chi$, and so $K(\mfq)\subset J_E$ acts on $v$ by $\chi_\tau$ as in Definition~\ref{definition:we-say-that-pi-emphr-at-depth-mathfr-if-there-cycl-regular-depth-parameter-polynomial}.

  The element $\tau$ is also cyclic in the sense of Definition~\ref{definition:cyclic}, by Remark~\ref{Remark:assumption1} above.  Thus $\tau$ is a regular parameter for $\pi$ at depth $\mfq^2$.
\end{proof}

\begin{remark}\label{Rem:supercuspidalclassification}
  We expect that one could verify that every supercuspidal representation with ``odd integral depth'' and ``generic induction datum'' arises via Lemma~\ref{lemma:d1bfac67927b}.  Here ``depth'' should be understood according to the conventions of~\cite[\S5.1]{moyUnrefinedMinimalKtypes1994}, while by ``generic induction datum'', we mean that the parameter $\beta$ associated to a simple type contained in $\pi$ (see~\cite[(3.2.1)]{BK}) generates a field extension of $F$ of maximal degree, and that $\beta$ is minimal in the sense of~\cite[(1.4.14)]{BK}.  For the above construction, if $\mfq = \mathfrak{p}^k$, then the ``depth'' is $2k - 1$, while we may take $\beta$ to be $\varpi^{- 2k + 1} \tau$.
\end{remark}

\begin{lemma}\label{Lem:supercuspidalstablepair}
  Let $m,n$ be positive integers, and let $\pi, \sigma$ be regular supercuspidal representations of $\GL_{m},\GL_{n}$ at depth $\mfq^2$, constructed above in terms of parameters $\tau_\pi$ and $\tau_\sigma$.  We consider the following two cases:
  \begin{itemize}
  \item Case $m\neq n$.
  \item Case $m=n$.  Then the degree $n$ inert field extensions generated by $\tau_\pi$ and $\tau_\sigma$ can be identified, uniquely up to the action of $\Gal(E/F) \cong \Gal(k_E/k_F)$.  We suppose that the images of $\tau_\pi$ and $\tau_\sigma$ in the residue field $k_E$ of $E$ satisfy
    \begin{equation*}
      \overline{\tau}_\sigma-\overline{\tau}_\pi^\iota\neq 0
    \end{equation*}
    for all $\iota \in \Gal(k_E/k_F)$.
  \end{itemize}
  Then, in either case, the pair $(\pi,\sigma)$ is stable at depth $\mfq^2$.  Furthermore, the local analytic conductor of the Rankin--Selberg pair $\pi\times \sigma$ is
  \begin{equation}\label{eq:cj3u9x1ked}
    C (\pi\times \sigma)={[\mathfrak{o}:\mfq]}^{2mn}.
  \end{equation}
\end{lemma}
Note that~\eqref{eq:cj3u9x1ked} generalizes the conductor formula given in Example~\ref{example:stable-pairs-of-characters}, which concerns the special case of characters (i.e., $m = n = 1$).
\begin{proof}
  Let $P_\pi, P_\sigma$ be the polynomials for $\pi$ and $\sigma$ at depth $\mfq^2$.  By Assumption~\ref{Assumption:supercuspidal}, part~\eqref{enumerate:d1bc2d5fbd54}, they are irreducible polynomials of degrees $m$ and $n$.  When $m\neq n$, they necessarily generate the unit ideal.

  When $m=n$, the assumption on $\overline\tau_\pi,\overline\tau_\sigma$ implies that the images $\overline{P}_\pi, \overline{P}_\sigma$ of $P_\pi,P_\sigma$ mod $\mfp$ satisfy $(\overline{P}_\pi, \overline{P}_\sigma)=1$.  By Nakayama's lemma, we then have $(P_\pi,P_\sigma)=1$.

  By Definition~\ref{definition:let-pi-sigma-be-repr-pair-gener-line-groups-over-f} and Lemma~\ref{Lem:supercuspidalregular}, we deduce that the pair $(\pi,\sigma)$ is stable at depth $\mfq^2$.

  The conductor formula follows readily from~\cite[Theorem 6.5(ii)]{MR1606410}.  Indeed, observe first that $\pi,\sigma$ are ``completely distinct'' in the sense of~\cite[\S6.2]{MR1606410}, so the cited result is applicable. The ramification indices $e_i$ there are all $1$ in our case. The claimed conductor formula follows by writing $\mfq=\mfp ^k$ and taking $m=2k-1$ in~\cite[Theorem 6.5(ii)]{MR1606410}.
\end{proof}

\begin{remark}
  Note that the additional condition in the $m=n$ case is necessary, as the conductor would otherwise be smaller.  It is the analogue of the condition we imposed for a pair of characters in Example~\ref{example:stable-pairs-of-characters}.
\end{remark}

\section{Main local result}\label{sec:mainlocal}
In this section, we retain the general notation and conventions of \S\ref{sec:notation}, and denote by $n+1$ the rank of $\mathbf{V}$.  We use the notation $A \ll B$ to denote that $\lvert A \rvert \leq C \lvert B \rvert$, where $C$ depends at most upon $n$.  We write $A \asymp B$ to denote that $A \ll B \ll A$.

Let $(F,\mathfrak{o},\mathfrak{p},\varpi,q)$ be a non-archimedean local field, with associated data as in \S\ref{sec:local-fields-congruence-subgroups}.

\subsection{Statement of result}
The following may be understood as a non-archimedean analogue, uniform with respect to variation of the local field, of the archimedean result~\cite[Thm 4.2]{2020arXiv201202187N}.  It encapsulates the local results needed to prove Theorem~\ref{theorem:cj3ngw7u2s}.

\begin{theorem}\label{theorem:main-local-result}
  Let $\mfq \subseteq \mathfrak{p}$ be a nonzero $\mathfrak{o}$-ideal.  Denote by $Q := [\mathfrak{o}:\mfq]$ its absolute norm.

  Let $\pi$ and $\sigma$ be representations of $\mathbf{G}(F)$ and $\mathbf{H}(F)$ equipped with inner products $\langle , \rangle$ invariant by $K(\mfq)$ and $K_H(\mfq)$, respectively.  Assume that the pair $(\pi,\sigma)$ is stable at depth $\mfq^2$ (Definition~\ref{definition:let-pi-sigma-be-repr-pair-gener-line-groups-over-f}).  Then there exists
  \begin{itemize}
  \item a unit vector $v \in \pi$,
  \item a self-adjoint idempotent $\omega \in C_c^\infty(K)$ with $\pi(\omega) v = v$,
  \item a compact open subgroup $J_H$ of $K_H$, and
  \item a unit vector $u \in \sigma$ that is a $J_H$-eigenvector,
  \end{itemize}
  with the following properties.

  Let $\pi|_{Z} : \mathbf{Z}(F) \rightarrow \U(1)$ denote the central character of $\pi$.  Define
  \begin{equation*}
    \omega ^\sharp (g) := \int _{z \in \mathbf{Z}(F) } \pi|_{Z} (z) \omega (z g ) \, d z.
  \end{equation*}
  \begin{enumerate}[(i)]
  \item\label{enumerate:20230517151810} We have
    \begin{equation}\label{eqn:sum-_v-in-mathc-mathc-pif-v-otim-u-gg-q-n2-}
      \int _{h \in \mathbf{H}(F)} \langle h v, v \rangle \langle u, h u \rangle \, d h \asymp Q^{-n^2},
    \end{equation}
    where the integrand is compactly-supported, hence converges absolutely.
  \item\label{enumerate:20230517151812} We have
    \begin{equation}\label{eqn:int-_mathbfhf-lvert-f-sharp-rvert-ll-q-n.-}
      \int _{\mathbf{H}(F)} \lvert \omega ^\sharp  \rvert \ll Q ^{n}.
    \end{equation}
  \item\label{enumerate:20230517165012} Let $\Psi_1, \Psi_2 : K_H \rightarrow \mathbb{C}$ be functions satisfying
    \begin{equation*}
      \lvert \Psi_j(g z) \rvert = \left\lvert \Psi_j(g) \right\rvert \text{ for all } (g,z) \in K_H \times  J_H.
    \end{equation*}
    Let
    \begin{equation*}
      \gamma \in \mathbf{G}(F) - \mathbf{H}(F) \mathbf{Z}(F).
    \end{equation*}
    Then
    \begin{equation}\label{eqn:int-_x-y-in-k_h-leftlv-psi_1x-psi_2y-omega-sharp-x}
      \int _{x, y \in K_H} \left\lvert \Psi_1(x) \Psi_2(y) \omega ^\sharp (x ^{-1} \gamma y) \right\rvert
      \ll
      Q^n
      \left( \frac{1}{1 + Q d_H(\gamma)} + \frac{{d_H(\gamma)}^{\infty}}{Q^*} \right)
      \lVert \Psi_1 \rVert_{L^2}    
      \lVert \Psi_2 \rVert_{L^2}.
    \end{equation}
    Here the notation $d_H(\gamma)$ and $d_H(\gamma)^\infty$ is as in \S\ref{sec:norms-distance-functions}, while $Q^*$ is as in the statement of Theorem~\ref{theorem:volume-bound}.
  \end{enumerate}
\end{theorem}

The proof is given in \S\ref{sec:proof-main-local-compact}, following some preliminaries.

\subsection{Stability and matrix coefficients}\label{sec:20230517161735}
Here we adapt some arguments from~\cite[\S19]{nelson-venkatesh-1} to the non-archimedean setting, where they simplify considerably.

\begin{lemma}\label{lemma:matrix-coefficients-and-stability-one-rep}
  Let $\tau \in \mathbf{M}_{\stab}(\mathfrak{o}/\mfq)$.
  \begin{enumerate}[(i)]
  \item\label{enumerate:cj3twlzoes} Let $\pi$ be a representation of $\mathbf{G}(\mathfrak{o})$, equipped with a $K(\mfq)$-invariant inner product $\langle , \rangle$, and let $v \in \pi$ be a vector that transforms under $K(\mfq)$ according to $\chi_\tau$.  Then for $h \in \mathbf{H}(\mathfrak{o})$, we have
    \begin{equation}\label{eqn:20230516004612}
      \langle h v, v \rangle = 0 \quad \text{unless } h \in K_H(\mfq).
    \end{equation}
  \item\label{enumerate:cj3twlzxgy} Let $\pi$ be a representation of $\mathbf{G}(F)$ satisfying the same hypotheses as above.  Then~\eqref{eqn:20230516004612} holds for $h \in \mathbf{H}(F)$.
  \end{enumerate}
\end{lemma}
\begin{proof}
  The group
  \begin{equation}\label{eqn:kmathfrakq-cap-h-kmathfrakq-h-1-subseteq-k_g-intersection-where-u-goes}
    K(\mfq) \cap h K(\mfq) h^{-1} \subseteq K
  \end{equation}
  acts on $h v$ by the character $\chi_\tau(h^{-1} \bullet h)$ and on $v$ by $\chi_\tau$, so by orthogonality of characters, it suffices to show that whenever $h \notin K_H(\mfq)$, there exists $u$ in the intersection~\eqref{eqn:kmathfrakq-cap-h-kmathfrakq-h-1-subseteq-k_g-intersection-where-u-goes} such that
  \begin{equation*}
    \chi_\tau (h ^{-1} u h) \neq \chi _\tau (u).
  \end{equation*}
  We consider two cases.
  \begin{enumerate}[(i)]
  \item If $h \in \mathbf{H}(\mathfrak{o})$, then $\chi_\tau(h ^{-1} u h) = \chi_{\Ad(h) \tau}(u)$, and $K(\mfq) = h K(\mfq) h^{-1}$. For this reason, it suffices to show that $\Ad(h) \tau \equiv \tau \mod\mfq$ if and only if $h \in K_H(\mfq)$.  This is the content of Lemma~\ref{lemma:stable-implies-trivial-stabilizer}.
  \item The case $h \in \mathbf{H}(F) - \mathbf{H}(\mathfrak{o})$ is addressed by Lemma~\ref{lemma:let-tau-in-barm_st-h-in-mathbfhf-mathbfhm-then-the}, below.
  \end{enumerate}
\end{proof}

\begin{lemma}\label{lemma:let-tau-in-barm_st-h-in-mathbfhf-mathbfhm-then-the}
  Let $\tau \in \mathbf{M}_{\stab}(\mathfrak{o}/\mfq)$ and $h \in \mathbf{H}(F) - \mathbf{H}(\mathfrak{o})$.  Then there exists $u \in K(\mfq) \cap h K(\mfq) h^{-1}$ such that
  \begin{equation}\label{eqn:k1-cap-h-k1-h-1-}
    \chi_\tau(h ^{-1} u h) \neq \chi _\tau (u).
  \end{equation}
\end{lemma}
\begin{proof}
  The proof is an adaptation of~\cite[Lemma 19.7]{nelson-venkatesh-1}. We will show more precisely that there either exists
  \begin{equation*}
    u \in K(\mfq^2) \cap h K(\mfq) h ^{-1}
  \end{equation*}
  such that $\chi _\tau (h ^{-1} u h) \neq 1$, which suffices in view of the fact that $\chi_\tau(u) = 1$.

  We choose a basis $e_1,\dotsc,e_n$ for $\mathbf{V}_H(\mathfrak{o})$.  We denote by $A_H$ the subgroup of $\mathbf{H}(F)$ diagonalized by this basis.  Let $A_H^+$ denote the subgroup consisting of diagonal matrices whose entries are integral powers of the uniformizer $\varpi$ of $\mathfrak{p}$.

  We first reduce to the special case that $h \in A_H^+ - K_H$.  To that end, we apply the Cartan decomposition to write
  \begin{equation*}
    h = k_1 a k_2,
  \end{equation*}
  where $k_1, k_2 \in K_H$ and $a \in A_H^+ - K_H$.  Then
  \begin{equation*}
    \chi_\tau(h^{-1} u h) = \chi_{\Ad(k_2) \tau}(a^{-1} k_1^{-1} u k_1 a),
  \end{equation*}
  \begin{equation*}
    K(\mfq^2) \cap h K(\mfq) h^{-1}
    =
    K(\mfq^2) \cap k_1 a K(\mfq) a^{-1}k_1^{-1}.
  \end{equation*}
  We observe that $\Ad(k_2) \tau$ satisfies the same hypotheses as $\tau$.  Suppose we can find some $u \in K(\mfq^2) \cap a K(\mfq) a^{-1}$ so that
  \begin{equation*}
    \chi_{\Ad(k_2) \tau}(a^{-1} u a) \neq 1.
  \end{equation*}
  Then, setting
  \begin{equation*}
    u' := k_1 u k_1^{-1} \in K(\mfq^2) \cap h K(\mfq) h^{-1},
  \end{equation*}
  we obtain
  \begin{equation*}
    \chi_\tau(h^{-1} u' h) = \chi_{\Ad(k_2) \tau}(a^{-1} u a ) \neq 1,
  \end{equation*}
  as required.

  Thus, let $a \in A_H^+ - K_H$.  Consider the adjoint action of $a$ on $\mathbf{M}(F)$.  Since $a \notin K_H$, there are nontrivial weights for this action.  Suppose for instance that there are positive weights (an identical argument will apply if there are negative weights).  By conjugating $a$ by a permutation matrix (as we may, by the preceding argument), we may assume that the largest diagonal coordinate of $a$, say $\varpi^{-\ell}$ with $\ell \geq 1$, appears in components $1, \dotsc, m$, where $m \geq 1$. 
If we use the partition
  \begin{equation*}
    n+1 = m + (n -m) + 1
  \end{equation*}
  to describe $\mathbf{M}(F)$ as a space of $3 \times 3$ block matrices, then we may describe the weight spaces for $a$ as follows:
  \begin{equation*}
    \begin{pmatrix}
      0 & + & + \\
      - & \ast & \ast \\
      - & \ast & \ast \\
    \end{pmatrix}.
  \end{equation*}
  Here the symbol $0$ indicates where $a$ acts trivially, while $+$ (resp.\ $-$) describe $\mathfrak{o}$-submodules
  \begin{equation*}
    \mathbf{M}_{\pm} \subseteq \mathbf{M}
  \end{equation*}
  such that for $x \in \mathbf{M}_{\pm}(F)$, we have
  \begin{equation*}
    a^{-1} x a = \varpi^{\pm \ell } x;
  \end{equation*}
  asterisks denote some unspecified combination of trivial, positive and negative weights.  We obtain in particular direct summands
  \begin{equation*}
    \bar{M}_{\pm } := \mathbf{M}_{\pm}(\mathfrak{o}/\mfq) \subseteq \bar{M} = \mathbf{M}(\mathfrak{o}/\mfq).
  \end{equation*}

  Let $\tau_+ \in \bar{M}_+$ denote the component of $\tau \in \bar{M}$.  We claim that $\tau_+$ is nonzero modulo $\mathfrak{p}$, i.e., that $\tau$ modulo $\mathfrak{p}$ is not of the form
  \begin{equation*}
    \begin{pmatrix}
      \ast & 0 & 0 \\
      \ast & \ast & \ast \\
      \ast & \ast & \ast \\
    \end{pmatrix}.
  \end{equation*}
  Indeed, if it were, then $\tau$ modulo $\mathfrak{p}$ would stabilize an $m$-dimensional subspace of $\bar{V}_H^*$ modulo $\mathfrak{p}$.  This contradicts the characterization~\eqref{enumerate:there-are-no-nontr-tau-invar-subsp-v_h-or-v_h.-} of stability recorded in Lemma~\ref{lemma:stability-equivalences}, noting that $\tau$ modulo $\mathfrak{p}$ is likewise stable (Example~\ref{example:cj3twmtcpz}).

  The trace pairing puts $M_+$ and $M_-$ in duality, so that for each $x \in M_-$, we have
  \begin{equation}\label{eq:cj3twm8wr8}
    \trace(x \tau) = \trace(x \tau_+).
  \end{equation}
  Since $\tau_+$ is nonzero modulo $\mathfrak{p}$, we may find such an $x$ for which
  \begin{equation}\label{eqn:tracex-tau-in-o-minus-p}
    \trace(x \tau_+) \in \mathfrak{o} - \mathfrak{p}.
  \end{equation}
  Let $t \in \mathfrak{p} ^{-1} \mfq ^2 - \mfq ^2 $, to be determined later, and take
  \begin{equation*}
    u := 1 + \varpi^{\ell} t x.
  \end{equation*}
  We note that, since $\ell \geq 1$, we have
  \begin{equation}\label{eqn:u-in-kmathfrakq2}
    u \in K(\mfq^2).
  \end{equation}
  On the other hand,
  \begin{equation}\label{eqn:a-1-u}
    a^{-1} u a = 1 + t x \in K(\mathfrak{p}^{-1} \mfq^2) \subseteq K(\mfq),
  \end{equation}
  while
  \begin{equation*}
    \chi_\tau(a^{-1} u a) = \chi_\tau(1 + t x) = \psi(\trace(t x \tau )).
  \end{equation*}
  By~\eqref{eq:cj3twm8wr8} and~\eqref{eqn:tracex-tau-in-o-minus-p}, we have $t \trace( x \tau) \in \mathfrak{p}^{-1} \mfq^2 - \mfq^2$.  By choosing $t$ suitably, we may thus arrange that
  \begin{equation}\label{eqn:chi_taua-1-u}
    \chi_\tau(a^{-1} u a) = \psi(t \trace(x \tau)) \neq 1.
  \end{equation}
  The required conclusion is then immediate from~\eqref{eqn:u-in-kmathfrakq2},~\eqref{eqn:a-1-u} and~\eqref{eqn:chi_taua-1-u}.
  
%
\end{proof}
\begin{lemma}\label{lemma:let-pi-sigma-be-pair-repr-mathbfgm-mathbfhm-test-vectors-stable-pairs}
  Let $\pi$ and $\sigma$ be representations of $\mathbf{G}(\mathfrak{o})$ and $\mathbf{H}(\mathfrak{o})$, respectively, such that the pair $(\pi,\sigma)$ is stable at depth $\mfq^2$.
  \begin{enumerate}[(i)]
  \item There is a stable element $\tau \in \mathbf{M}(\mathfrak{o}/\mfq)$ such that $\tau$ (resp.\ $\tau_H$) is a regular parameter for $\pi$ (resp.\ $\sigma$) at depth $\mfq^2$.
  \item\label{enumerate:cj3twlwnkr} Suppose $\pi$ (resp.\ $\sigma$) is equipped with an inner product $\langle , \rangle$ invariant by $K(\mfq)$ (resp.\ $K_H(\mfq)$). Let $\tau$ be any element as above.  Let $v$ (resp.\ $u$) be unit vectors that transform under $K(\mfq)$ (resp.\ $K_H(\mfq)$) according to $\chi_\tau$ (resp.\ $\chi_{\tau_H}$).  Then
    \begin{equation}\label{eq:cj3twlxge3}
      \int_{h \in \mathbf{H}(\mathfrak{o})} \langle h v, v \rangle \langle u, h u \rangle \, d h
      = \vol(K_H(\mfq)).  
    \end{equation}
  \item\label{enumerate:cj3twly3hj} Retaining the hypotheses of~\eqref{enumerate:cj3twlwnkr}, if $\pi$ and $\sigma$ arise as restrictions of representations of $\mathbf{G}(F)$ and $\mathbf{H}(F)$, then the formula~\eqref{eq:cj3twlxge3} remains valid after extending the integral to $\mathbf{H}(F)$.
  \end{enumerate}
\end{lemma}
\begin{proof}
  Let $P_\pi$ and $P_\sigma$ be polynomials for $\pi$ and $\sigma$ at depth $\mfq^2$ that generate the unit ideal.  By Lemma~\ref{lemma:let-p-p_h-in-rx-be-monic-polyn-degr-n+1-n-resp-the}, there exists $\tau \in \mathbf{M}(\mathfrak{o}/\mfq)$ such that $P_{\tau} = P_{\pi}$ and $P_{\tau_H} = P_{\sigma}$.  By Lemma~\ref{lemma:stability-equivalences}, $\tau$ is stable, and in particular, cyclic.  By Lemma~\ref{lemma:tau-stable-implies-tauH-cyclic}, $\tau_H \in \mathbf{M}_H(\mathfrak{o}/\mfq)$ is cyclic.  By Lemma~\ref{lemma:let-p-be-polyn-pi-at-depth-mathfr-let-tau-in-mathb}, it follows that $\tau$ (resp.\ $\tau_H$) is a regular parameter for $\pi$ (resp.\ $\sigma$) at depth $\mfq^2$.

  For the second assertion~\eqref{enumerate:cj3twlwnkr} concerning the matrix coefficient integral, we apply part~\eqref{enumerate:cj3twlzoes} of Lemma~\ref{lemma:matrix-coefficients-and-stability-one-rep} to truncate that integral to $h \in K_H(\mfq)$, where the integrand evaluates to
  \begin{equation*}
    \chi_\tau(h) {\chi_{\tau_H}(h)}^{-1} = 1.
  \end{equation*}
  For the final assertion~\eqref{enumerate:cj3twly3hj}, we argue similarly using part~\eqref{enumerate:cj3twlzxgy} of Lemma~\ref{lemma:matrix-coefficients-and-stability-one-rep}.
\end{proof}

\subsection{Setup for the proof}\label{sec:cj4t6ynizh}
We retain the hypotheses of Theorem~\ref{theorem:main-local-result}.  By Lemma~\ref{lemma:let-pi-sigma-be-pair-repr-mathbfgm-mathbfhm-test-vectors-stable-pairs}, we obtain a stable element $\tau \in \mathbf{M}(\mathfrak{o}/\mfq)$ such that $\tau$ (resp.\ $\tau_H$) is a regular parameter for $\pi$ (resp.\ $\sigma$) at depth $\mfq^2$.  By Lemma~\ref{lemma:let-tau-in-mathbfm-mathfr--mathfr-be-regul-param-p-extension-chi-tau-to-J-tau}, we may find unit vectors $v \in \pi$ and $u \in \sigma$ that transform under $J_\tau$ and $J_{\tau_H}$ by characters $\tilde{\chi}_\tau$ and $\tilde{\chi}_{\tau_H}$ that extend $\chi_\tau$ and $\chi_{\tau_H}$, respectively.  We set $(J_G, J_H, \chi_G, \chi_H) := (J_\tau, J_{\tau_H}, \tilde{\chi}_\tau, \tilde{\chi}_{\tau_H})$ and
\begin{equation*}
  \omega := {\vol(J_G)}^{-1} \chi_G^{-1} \in C_c^\infty(K),
\end{equation*}
so that $\pi(\omega)$ is a self-adjoint idempotent with $\pi(\omega) v = v$.

We pause to clarify the shape of the function $\omega ^\sharp$.
\begin{lemma}\label{lemma:g-in-mathbfgf-we-have-begin-lvert-f-sharp-g-rvert-clarify-f-sharp}
  For $g \in \mathbf{G}(F)$, we have
  \begin{equation*}
    \lvert \omega ^\sharp (g) \rvert \ll {\vol(J_G)}^{-1} 1 _{g \in \mathbf{Z}(F) J_G}.
  \end{equation*}
\end{lemma}
\begin{proof}
  By definition,
  \begin{equation*}
    \omega ^\sharp (g) = \int _{z \in \mathbf{Z}(F)} \pi|_{Z}(z) \omega(z g) \, d z.
  \end{equation*}
  The function $\omega$ is supported on $J_G$ and has $L^\infty$-norm ${\vol(J_G)}^{-1}$.  It follows that if $g \notin \mathbf{Z}(F) J_G$, then $\omega ^\sharp (g)$ vanishes.  We have $\mathbf{Z}(F) \cap J_G = K_Z$, so the set of $z$ for which $\omega(z g) \neq 0$ is a $K_Z$-coset, and so has volume $\ll 1$.  We have assumed that $\pi$ is unitary, so $\lvert \pi|_{Z} (z) \rvert = 1$.  The stated estimate follows from these observations and the triangle inequality.
\end{proof}

\subsection{The proof}\label{sec:proof-main-local-compact}
We now verify each numbered assertion in turn.
\begin{enumerate}[(i)]
\item We see from Lemmas~\ref{lemma:matrix-coefficients-and-stability-one-rep} and~\ref{lemma:let-pi-sigma-be-pair-repr-mathbfgm-mathbfhm-test-vectors-stable-pairs} that the integrand in~\eqref{eqn:sum-_v-in-mathc-mathc-pif-v-otim-u-gg-q-n2-} is supported on $h \in K_H(\mfq)$ and the integral evaluates to $\vol(K_H(\mfq)) \asymp Q^{-n^2}$, as required.
\item By Lemma~\ref{lemma:g-in-mathbfgf-we-have-begin-lvert-f-sharp-g-rvert-clarify-f-sharp}, we have
  \begin{equation}\label{eqn:int-_mathbfhf-lvert-f-sharp-rvert-ll-frac1v-int-_h}
    \int _{\mathbf{H}(F)} \lvert \omega ^\sharp  \rvert \ll
    \frac{1}{\vol(J_G)}
    \int _{h \in \mathbf{H}(F)} 1 _{\mathbf{Z}(F) J_G }(h) \, d h.
  \end{equation}
  We must verify that the above is $\ll Q^n$.  Suppose $h \in \mathbf{H}(F)$ may be written $h = z g$ with $(z, g ) \in \mathbf{Z} (F) \times J _G$.  Then
  \begin{equation*}
    z^{-1} h \in J_G \cap \mathbf{H}(F) \mathbf{Z}(F) = J_G \cap K_H K_Z
    = (J_G \cap K_H)K_Z,
  \end{equation*}
  since $K_Z \subseteq J_G$.  By Lemma~\ref{lemma:stable-implies-trivial-stabilizer}, we have
  \begin{equation*}
    J_G \cap K_H = K_H(\mfq).
  \end{equation*}
  We have
  \begin{equation}\label{eqn:volj_g-=-k_g-:-j_g-1-asymp-qn+12-n+1-=-qnn+1.-vol-J-G} {\vol(J_G)}^{-1} \vol(K_H(\mfq))
    = \frac{[K : J_G]}{[K_H : K_H(\mfq)]} \asymp \frac{Q^{{(n+1)}^2 - (n+1)}}{Q^{n^2}} = Q^{n}.
  \end{equation}
  The required estimate follows.
\item We are given $\Psi_1, \Psi_2 : K_H \rightarrow \mathbb{C}$ whose magnitudes are right-invariant by $J_{\tau_H}$, and $\gamma \in \mathbf{G}(F) - \mathbf{H}(F) \mathbf{Z}(F)$.  We must verify that
  \begin{equation}\label{eqn:i_1-:=-int-_x-y-in-k_h-leftlv-psi_1x-psi_2y-f-shar-task}
    I_1 := \int _{x, y \in K_H} \left\lvert \Psi_1(x) \Psi_2(y) \omega ^\sharp (x ^{-1} \gamma y) \right\rvert
    \ll \Delta Q^n \lVert \Psi_1 \rVert_{L^2} \lVert \Psi_2 \rVert_{L^2},
  \end{equation}
  where $\Delta$ is the parenthetical quantity on the right hand side of~\eqref{eqn:int-_x-y-in-k_h-leftlv-psi_1x-psi_2y-omega-sharp-x}.  By the bound for $\omega ^\sharp$ recorded in Lemma~\ref{lemma:g-in-mathbfgf-we-have-begin-lvert-f-sharp-g-rvert-clarify-f-sharp}, we see that the left hand side of~\eqref{eqn:i_1-:=-int-_x-y-in-k_h-leftlv-psi_1x-psi_2y-f-shar-task} vanishes unless there exists $z \in \mathbf{Z}(F)$ so that
  \begin{equation*}
    z \gamma \in K_H J_G K_H \subseteq K.
  \end{equation*}
  Replacing $\gamma$ by $z^{-1} \gamma$ has no effect on either side of the desired estimate, so we may suppose that $\gamma \in K$.  Then, since $x,y \in K_H \subseteq K$ and $\mathbf{Z}(F) J_G \cap K = J_G$, we have
  \begin{equation*}
    \lvert \omega ^\sharp (x ^{-1} \gamma y) \rvert \ll {\vol(J_G)}^{-1} 1 _{x ^{-1} \gamma y \in J_G}.
  \end{equation*}
  The integral $I_1$ descends to the quotient
  \begin{equation*}
    \bar{H} := K_H / K_H(\mfq).
  \end{equation*}
  For $j=1,2$, let $u_j : \bar{H} \rightarrow \mathbb{R}_{\geq 0}$ denote the function induced by $\lvert \Psi_j \rvert$.  By definition, the image of $J_G$ in $K / K(\mfq)$ is the centralizer of $\tau$.  The integral $I_1$ thus satisfies
  \begin{equation*}
    I_1 \ll {\vol(J_G)}^{-1} I,
  \end{equation*}
  where $I$ is the integral~\eqref{eqn:i-:=-frac1lv-h-rvert2-sum-_-subst-x-y-in-h-:-} defined in the statement of Theorem~\ref{theorem:bilinear-forms-estimate}, which gives
  \begin{equation*}
    I \ll \Delta \lvert \bar{H} \rvert^{-1} \lVert u_1 \rVert _{L^2} \lVert u_2 \rVert_{L^2}.
  \end{equation*}
  The required estimate for $I_1$ follows now from~\eqref{eqn:volj_g-=-k_g-:-j_g-1-asymp-qn+12-n+1-=-qnn+1.-vol-J-G}.
\end{enumerate}

\section{Completion of the proof}\label{Sec:final}
We now deduce our main result Theorem~\ref{theorem:cj3ngw7u2s}, from our main local result, Theorem~\ref{theorem:main-local-result}.  The deduction is exactly as in~\cite[\S6]{2020arXiv201202187N}, so we will be brief.  We denote in what follows by $\mathbb{Z}_{\mathfrak{l}} \leq F_{\mathfrak{l}}$ the ring of integers in the completion of $F$ at a finite place $\mathfrak{l}$.
\begin{proposition}\label{proposition:20230517165122}
  Fix $\alpha > 0$ and $\eps > 0$.  Retain the setting of Theorem~\ref{theorem:cj3ngw7u2s}; in particular, $F$ is a number field, $(G,H) = (\U(V), \U(W))$ is a pair of unitary groups attached to a nondegenerate codimension one inclusion $W \hookrightarrow V$ of positive-definite hermitian spaces over $F$, $S$ is a large enough finite set of places, $(\pi,\sigma,\mfp,\mfq) \in \mathcal{F}$, $T$ is the absolute norm of $\mathfrak{q}^2$, and $\vartheta \in [0,1/2)$ is such that $\sigma$ is $\vartheta$-tempered at every finite place $\mfop \notin S \cup \{\mfp\}$ that splits in $E$.  Set $L := T^{\alpha}$.  Let $\omega \in C_c^\infty(G(\mathbb{Z}_\mfp))$ and $J_H \leq H(\mathbb{Z}_\mfp)$ be as in Theorem~\ref{theorem:main-local-result} applied to $(\pi_\mfp,\sigma_\mfp)$.
  There exists $c > 0$, depending only upon $\alpha$ and $(G,H,S)$, with the following property.  For $j=0,\dotsc,n+1$, let $\Delta_j$ denote the infimum of all nonnegative quantities with the following property: for all $u_1, u_2 \in L^2(H(\mathbb{Z}_\mfp))$ of unit norm that transform on the right under $J_H$ via some unitary character, and all $\gamma \in G(F_\mfp)$ with $d_{H_\mfp}(\gamma) \geq c L^{-j}$, we have
  \begin{equation}\label{eqn:20230517152904}
    \int_{x, y \in H(\mathbb{Z}_\mfp)}
    \left\lvert
      u_1(x)
      u_2(y)
      \omega^\sharp (x^{-1} \gamma y )
    \right\rvert
    \, d x \, d y
    \leq T^{n / 2} \Delta_j.
  \end{equation}
  Then
  \begin{align*}
    \frac{\mathcal{L} (\pi,\sigma)}{T^{n(n+1)/2 + \eps}}
    &\ll
      \sum _{j=1}^{n+1}
      \left(
      L^{-(1-2 \vartheta) j}
      +
      L^{(2 {(n+1)}^2 - n)j}
      \Delta_j
      \right)\\
    &\quad +
      L^{-1}
      \sum _{j=0}^{n+1}
      \left(
      L^{-(1-2 \vartheta) j}
      +
      L^{(2 {(n+1)}^2 - n )j}
      \Delta_j
      \right),
  \end{align*}
  where the meaning of $\ll$ is ``bounded in magnitude up to a factor depending only upon $(\mathcal{F},\alpha,\eps, \vartheta)$''.
\end{proposition}
\begin{proof}
  This follows from the arguments of~\cite[\S6]{2020arXiv201202187N}, applied with $Q=T^{1/2}$.  The positive-definiteness assumption is used in two ways: so that the quotient $H(F) \backslash H(\mathbb{A})$ is compact, and so that each of the groups $G$ and $H$ are compact at infinity.  This last property is used in applying~\cite[Lem 5.1]{2020arXiv201202187N}.

  In more detail, we repeat~\cite[\S6.1-6.4]{2020arXiv201202187N} verbatim (omitting the adjective ``archimedean'' in a couple places --- this property is never used).  In the first line of~\cite[\S6.5.1]{2020arXiv201202187N}, we observe that the integrand in~\cite[(6.5)]{2020arXiv201202187N} is now invariant merely under $\prod_{\mfop \notin S \cup \{\mfp \}} H(\mathbb{Z}_\mfop)$, rather than the larger group $\prod_{\mfop \notin S} H(\mathbb{Z}_\mfop)$.  (This is because of a difference in conventions: in~\cite{2020arXiv201202187N}, $S$ contains the interesting place, while here, it does not.) In the subsequent estimates, we accordingly replace the factorizable neighborhood $\Theta_S$ with
  \begin{equation*}
    \Theta_{S \cup \{\mfp \}} := \Theta_S H(\mathbb{Z}_\mfp).
  \end{equation*}
  The remainder of~\cite[\S6.5.1-6.5.2]{2020arXiv201202187N} is unchanged.  The integral on the left hand side of~\eqref{eqn:20230517152904} arises naturally following the proof\footnote{We note that the statement of~\cite[Lem 6.7]{2020arXiv201202187N} contains a regrettable typo, introduced in the final revision: ${d_{H_\mathfrak{q}}(\gamma)}^{-1/2}$ should be ${d_{H_\mathfrak{q}}(\gamma)}^{-1}$.  This does not affect the present argument.}  of~\cite[Lem 6.7]{2020arXiv201202187N}.  We feed the resulting estimate into the summary of~\cite[\S6.6]{2020arXiv201202187N} to obtain the stated bound.
\end{proof}
We now apply Theorem~\ref{theorem:main-local-result}, part~\eqref{enumerate:20230517165012} to see that the quantities $\Delta_j$ as in Proposition~\ref{proposition:20230517165122} satisfy
\begin{equation}\label{eqn:20230517165322}
  \Delta_j \ll \frac{L^j}{T^{1/2}} + \frac{1}{R },
\end{equation}
where $R$ is given by
\begin{equation*}
  T = q^{2 m} \implies R = q^{\lceil m/2 \rceil}.
\end{equation*}
In particular, $R$ satisfies the slightly wasteful estimate $R \geq T^{1/4} $.  Substituting this estimate into the right hand side of the bound stated in Proposition~\ref{proposition:20230517165122}, we see that two terms dominate, giving
\begin{equation}\label{eqn:20230517170454}
  \frac{\mathcal{L}(\pi,\sigma)}{ T^{n (n + 1 )/2 + \eps }}
  \ll
  L^{-(1 - 2 \vartheta )} + L^{(2 {(n + 1 )}^2 - n ) (n+1)} \left( \frac{L^{n+1}}{ T^{1/2} }
    + \frac{1}{T^{1/4} }\right).
\end{equation}
Expanding the right hand side as a sum of three terms, the first and third terms are equal if
\begin{equation*}
  L^{- (1 - 2 \vartheta )}
  =
  L^{(2 {(n + 1 )}^2 - n ) (n+1)}
  \frac{1}{T^{1/4} },
\end{equation*}
or equivalently,
\begin{equation*}
  \alpha = \frac{1}{4 ( A + 1 - 2 \vartheta)}, \quad A := (2 {(n + 1 )}^2 - n ) (n + 1).
\end{equation*}
We then have $A \geq n+1$, hence $(n+1) \alpha \leq 1/4$, so the second term on the right hand side of~\eqref{eqn:20230517170454} is dominated by the third term.  The estimate~\eqref{eqn:20230517170454} thus simplifies to the required bound
\begin{equation*}
  \frac{\mathcal{L}(\pi,\sigma)}{ T^{n (n + 1 )/2 + \eps }}
  \ll
  T^{-\delta}, \quad
  \delta := \frac{1 - 2 \vartheta }{4 (A + 1 - 2 \vartheta )}.
\end{equation*}

\def\cprime{$'$} \def\cprime{$'$} \def\cprime{$'$} \def\cprime{$'$}

\end{document}